\documentclass[oneside,english]{amsart}
\usepackage[T1]{fontenc}
\usepackage[latin9]{inputenc}
\usepackage{amstext}
\usepackage{amsthm}
\usepackage{amssymb}
\usepackage{graphicx}
\usepackage[all]{xy}

\makeatletter
\numberwithin{equation}{section}
\numberwithin{figure}{section}
\theoremstyle{plain}
\newtheorem{thm}{\protect\theoremname}
  \theoremstyle{definition}
  \newtheorem{defn}[thm]{\protect\definitionname}
  \theoremstyle{plain}
  \newtheorem{prop}[thm]{\protect\propositionname}
  \theoremstyle{remark}
  \newtheorem{rem}[thm]{\protect\remarkname}
  \theoremstyle{plain}
  \newtheorem{lem}[thm]{\protect\lemmaname}
  \theoremstyle{definition}
  
  \theoremstyle{plain}
  \newtheorem{cor}[thm]{\protect\corollaryname}
  \theoremstyle{plain}
  \newtheorem{fact}[thm]{Fact}
  \theoremstyle{plain}

\newcommand\Conv{\mathbf{Conv}}

\newcommand\Cv[1]{{#1}^{\mathrm{conv}}}
\newcommand\sCv[1]{{#1}^{\mathrm{sconv}}}
\newcommand\Cvcl[1]{{#1}^{\mathrm{conv{*}}}}
\newcommand\sCvcl[1]{{#1}^{\mathrm{sconv{*}}}}
\newcommand\Limc{\mathbf{Lim}}
\newcommand\Lim[1]{{#1}^{\mathrm{lim}}}
\newcommand\Limcl[1]{{#1}^{\mathrm{lim{*}}}}
\newcommand\sLim[1]{{#1}^{\mathrm{slim}}}
\newcommand\sLimcl[1]{{#1}^{\mathrm{slim{*}}}}
\newcommand\PreTop{\mathbf{PreTop}}
\newcommand\preTop[1]{{#1}^{\mathrm{pretop}}}
\newcommand\preTopcl[1]{{#1}^{\mathrm{pretop{*}}}}

\newcommand\Adhc{\mathbf{Adh}}
\newcommand\adhname{\mathrm{adh}}
\newcommand\Adh[1]{{#1}^{\adhname}}
\newcommand\Topc{\mathbf{Top}}
\newcommand\Top[1]{{#1}^{\mathrm{top}}}

\newcommand\Frm{\mathbf{Frm}}
\newcommand\coFrm{\mathbf{CF}}

\newcommand{\Filt}{\mathbb{F}}
\newcommand\Idl{\mathbf{I}}
\newcommand{\pow}{\mathbb{P}}
\newcommand{\kow}[1]{{#1}^\circ}

\newcommand{\limp}{\mathrel{\Rightarrow}}

\newcommand{\dc}{\mathop{\downarrow}}
\newcommand{\upc}{\mathop{\uparrow}}

\newcommand{\Sl}{{S\ell}}
\newcommand{\clos}{\mathfrak{c}}
\newcommand{\ouv}{\mathfrak{o}}

\newcommand{\Sierp}{\mathbb{S}}

\newcommand\identity[1]{\mathrm{id}_{#1}}
\newcommand\C{\mathop{\mathcal C}}

\newcommand\nat{\mathbb{N}}

\makeatother

\usepackage{babel}
  \providecommand{\corollaryname}{Corollary}
  \providecommand{\definitionname}{Definition}
  \providecommand{\examplename}{Example}
  \providecommand{\lemmaname}{Lemma}
  \providecommand{\propositionname}{Proposition}
  \providecommand{\remarkname}{Remark}
\providecommand{\theoremname}{Theorem}

\begin{document}
\global\long\def\G{\mathcal{G}}
 \global\long\def\F{\mathcal{F}}
 \global\long\def\H{\mathcal{H}}
\global\long\def\Z{\mathcal{Z}}
 \global\long\def\L{\mathcal{L}}
\global\long\def\U{\mathcal{U}}
\global\long\def\W{\mathcal{W}}
 \global\long\def\E{\mathcal{E}}
\global\long\def\B{\mathcal{B}}
 \global\long\def\A{\mathcal{A}}
\global\long\def\D{\mathcal{D}}
\global\long\def\O{\mathcal{O}}
 \global\long\def\N{\mathcal{N}}
 \global\long\def\X{\mathcal{X}}
 \global\long\def\lm{\lim\nolimits}
 \global\long\def\then{\Longrightarrow}

\global\long\def\V{\mathcal{V}}
\global\long\def\C{\mathcal{C}}
\global\long\def\adh{\operatorname{adh}\nolimits}
\global\long\def\adhr{\adh^0}
\global\long\def\Seq{\operatorname{Seq}\nolimits}
\global\long\def\intr{\operatorname{int}\nolimits}
\global\long\def\cl{\operatorname{cl}\nolimits}
\global\long\def\inh{\operatorname{inh}\nolimits}
\global\long\def\diam{\operatorname{diam}\nolimits\ }
\global\long\def\card{\operatorname{card}}
\global\long\def\T{\operatorname{T}}
\global\long\def\S{\operatorname{S}}

\global\long\def\fix{\operatorname{fix}\nolimits}
\global\long\def\Epi{\operatorname{Epi}\nolimits}
\global\long\def\op#1{\operatorname#1}
\global\long\def\pt{\operatorname{pt}\nolimits}
\global\long\def\ptAdh{\operatorname{pt}\nolimits^{\mathrm{a}}} %{\adhname}}
\global\long\def\ptS{\operatorname{pt}\nolimits^{\mathrm{s}}}

\newcommand\catc{\mathbf{C}}
\newcommand\catd{\mathbf{D}}
\newcommand\cate{\mathbf{E}}

\title{Convergence without Points}
\author{Jean Goubault-Larrecq and Frédéric Mynard}
\thanks{The second author acknowledges the generous support of the NJCU Separately Budgeted Research program.}
\begin{abstract}
  We introduce a pointfree theory of convergence on lattices and
  coframes.  A convergence lattice is a lattice $L$ with a monotonic
  map $\lm_L$ from the lattice of filters on $L$ to $L$, meant to be
  an abstract version of the map sending every filter of subsets to
  its set of limits.  This construction exhibits the category of
  convergence spaces as a coreflective subcategory of the opposite of
  the category of convergence lattices.  We extend this construction
  to coreflections between limit spaces and the opposite of so-called
  limit lattices and limit coframes, between pretopological
  convergence spaces and the opposite of so-called pretopological
  convergence coframes, between adherence spaces and the opposite of
  so-called adherence coframes, between topological spaces and the
  opposite of so-called topological coframes.  All of our pointfree
  categories are complete and cocomplete, and topological over the
  category of coframes.  Our final pointfree category, that of
  topological coframes, shares with the category of frames the
  property of being in a dual adjunction with the category of
  topological spaces.  We show that the latter arises as a retract of
  the former, and that this retraction restricts to a reflection
  between frames and so-called strong topological coframes.
% We introduce a pointfree theory of convergence on coframes and establish
% a natural adjunction between the category of convergence spaces and
% the (opposite of the) new category of convergence frames. We investigate
% generalizations to this setting of various classical convergence notions
% including adherence, closedness, openness, interior and closure, pseudotopologicity,
% pretopologicity and topologicity. In particular, we consider the behavior
% of these notions through our adjunction. Finally, we indicate why
% it seems unlikely to  relate this version of our theory of pointfree
% convergence to classical pointfree topology in a direct fashion.
\end{abstract}

\maketitle

\section{Introduction}
\label{sec:introduction}

Locales are pointfree analogues of topological spaces.  This is by now
a well-known concept \cite{Johnstone:Stone,framesandlocales}.  Our aim
is to propose a pointfree analogue of convergence spaces.  That may
seem curious at best, since convergence is a relation between
sequences, or better, filters, and \emph{points}.  Nevertheless, as we
shall see, there is a natural way of doing so, generalizing the
classical Stone duality between topological spaces and locales to one
between the category $\Conv$ of convergence spaces and the opposite
category of a category $\Cv\coFrm$ of coframes with additional
structure representing an abstract form of convergence.  We extend
those dualities to limit spaces, %pseudotopological spaces,
pretopological spaces, and topological spaces.

Our developments do not require the Axiom of Choice.
% , except for the
% case of pseudotopological spaces.

We introduce our basic notion of convergence $\catc$-object, where
$\catc$ is a category of inf-semilattices, in
Section~\ref{sec:convergence-lattices}.  When $\catc$ is a category of
lattices, we obtain a coreflective adjunction $\pow\dashv\pt$ between
$\Conv$ and the category $\Cv\catc$ of convergence $\catc$-objects.
When $\catc$ is a category of coframes, we show that $\Cv\catc$ is
complete and cocomplete in Section~\ref{sec:compl-cocompl}, by showing
that it is topological over the category $\coFrm$ of coframes.

We then proceed to similar constructions for limit spaces,
pretopological and topological spaces.  We organize our main results
in the following commutative diagram of adjunctions.  The bottom row
consists of familiar categories of convergence spaces, limit spaces,
pretopological convergence spaces and adherence spaces, and
topological spaces.  The top row are the matching pointfree
situations.
\begin{equation}
  \label{eq:adj}
  \xymatrix@C=10pt{
    &
    {(\Cv\coFrm)}^{op}
    \ar@<1ex>[r]\ar@{}[r]|{\perp}
    \ar@<1ex>[d]^{\pt}\ar@{}[d]|{\dashv}
    &
    {(\Lim\coFrm)}^{op}
    \ar@<1ex>[l]
    \ar@<1ex>[r]\ar@{}[r]|{\perp}
    \ar@<1ex>[d]^{\pt}\ar@{}[d]|{\dashv}
    &
    {(\preTop\coFrm)}^{op}
    \ar@<1ex>[l]
    \ar@<1ex>[r]\ar@{}[r]|{\perp}
    \ar@<1ex>[d]^{\pt}\ar@{}[d]|{\dashv}
    &
    {(\Adh\coFrm)}^{op}
    \ar@<1ex>[l]
    \ar@<1ex>[r]\ar@{}[r]|{\perp}
    \ar@<1ex>[d]^{\pt}\ar@{}[d]|{\dashv}
    &
    {(\Top\coFrm)}^{op}
    \ar@<1ex>[l]
    \ar@<1ex>[d]^{\pt}\ar@{}[d]|{\dashv}
    \\
    % \ar@{}[ru]|{\text{(Sec.~\ref{sec:convergence-lattices})}}
    &
    \Conv
    \ar@<1ex>[u]^{\pow}
    \ar@<1ex>[r]^{M_{\mathrm{lim}}} \ar@{}[r]|{\perp}
    \ar@{}[ru]|{\text{(Sec.~\ref{sec:limit-spaces})}}
    &
    \Limc
    \ar@<1ex>[l]^{\supseteq}
    \ar@<1ex>[u]^{\pow}
    \ar@<1ex>[r]^{M_{\mathrm{pretop}}}\ar@{}[r]|{\perp}
    \ar@{}[ru]|{\text{(Sec.~\ref{sec:pret-cofr})}}
    &
    \PreTop
    \ar@<1ex>[l]^{\supseteq}
    \ar@<1ex>[u]^{\pow}
    \ar@{=}[r]
    \ar@{}[ru]|{\text{(Sec.~\ref{sec:adherence-coframes})}}
    &
    \Adhc
    \ar@<1ex>[u]^{\pow}
    \ar@<1ex>[r]^{M_{\mathrm{top}}} \ar@{}[r]|{\perp}
    \ar@{}[ru]|{\text{(Sec.~\ref{sec:topological-coframes})}}
    &
    \Topc
    \ar@<1ex>[u]^{\pow}
    \ar@<1ex>[l]^{\supseteq}
  }
\end{equation}
In addition, Section~\ref{sec:sierp-conv-cofr} introduces the notions
of open and closed elements in a convergence $\catc$-object, and
adherence operators associated with a convergence.  This is notably
required in the study of the connection between pretopological
$\catc$-objects and adherence $\catc$-objects of
Section~\ref{sec:adherence-coframes}.

% We have not depicted the situation with pseudotopological spaces in
% this diagram.  This is dealt with in Section~\ref{sec:pseud-cofr}.  As
% we have said, this is the only part of the paper that requires the
% Axiom of Choice.  That would add another vertical adjunction between
% $\Conv$ and $\PreTop$ in (\ref{eq:adj}).

\section{Convergence Lattices}
\label{sec:convergence-lattices}

We use \cite{DM:convergence} as a modern reference on convergence
spaces.

An \emph{inf-semilattice} is a poset with all finite infima.  Some
authors call that a bounded inf-semilattice, and reserve the term
``inf-semilattice'' for posets with binary infima, possibly missing
the top element $\top$.  Given a poset $L$, $L^{op}$ is its opposite
poset, whose ordering is reversed.  A \emph{sup-semilattice} $L$ is
such that $L^{op}$ is an inf-semilattice.  A \emph{lattice} has all
finite infima and all finite suprema.

Given an inf-semilattice $L$, a \emph{filter} on $L$ is a non-empty
upwards-closed subset of $L$ that is closed under binary meets. We do
allow $L$ itself as a filter, contrarily to, say,
\cite[Section~II.2]{DM:convergence}, although that is not
essential. All others are called \emph{proper}.  In general, an
upwards-closed subset is proper, i.e., different from $L$, if and only
if it does not contain the least element $\bot$ of $L$.

We write $\mathbb{F}L$ for the set of filters on $L$. For $L=\pow X$,
the lattice of all subsets of a set $X$, a filter on $L$ is called a
\emph{filter of subsets} of $X$. A maximal proper filter of subsets of
$X$ is an \emph{ultrafilter} on $X$.

A \emph{convergence space} is a pair $(X,\to)$ of a set $X$ and
a binary relation $\to$ between filters $\F$ of subsets
of $X$ and points $x$ of $X$ (in notation, $\F\to x$,
meaning ``$\F$ converges to $x$'', or ``$x$ is a limit
of $\F$'') such that: 
\begin{itemize}
\item (Point Axiom.) $\dot{x}\to x$, where
  $\dot{x}=\{S\subseteq X\mid x\in S\}$ is the \emph{principal
    ultrafilter} at $x$;
\item (Monotonicity Axiom.) If $\F\to x$ and $\F\subseteq\G$ then
  $\G\to x$.
\end{itemize}
We shall write $X$ for $(X,\to)$, leaving the notion of convergence
$\to$ implicit.
%  If the relation $\to$ only fulfills the Monotonicity
% Axiom but not necessarily the Point Axiom, we say that $(X,\to)$
% is a \emph{preconvergence space}. 

A standard example of a convergence space is given by topological
spaces.  The \emph{standard convergence} on a topological space $X$ is
given by $\F \to x$ if and only if every open neighborhood of $x$
belongs to $\F$.  There are many examples of convergence spaces that
do not arise from a topological space this way, see
\cite[Section~III.1]{DM:convergence}.

Convergence spaces form a category $\Conv$.
% , and preconvergence spaces
% form a larger category $\PreConv$.
A morphism $f\colon X\to Y$ is a \emph{continuous map}, meaning one
that preserves convergence: if $\F\to x$ then $f [\F]\to f(x)$, where
the \emph{image filter} $f[\F]$ is
$\{B\subseteq Y\mid f^{-1}(B)\in\F\}$.  A map $f$ is
\emph{initial} (in the categorical sense \cite[Definition 8.6]{AHS:joycats}) if and only if that implication is an equivalence. The
injective initial maps are exactly the embeddings of %(pre)
convergence spaces.

% For more information about convergence spaces, see the recent book
% \cite{DM:convergence}.

\subsection{The pointfree analogue of convergence spaces}
\label{sec:conv-catc-objects}

We obtain an analogous theory of \emph{pointfree} convergence by
abstracting away from the lattice $\pow(X)$, and replacing it by an
inf-semilattice $L$, possibly with extra structure.
%   The initial constructions,
% e.g., the notion of filter, would work with structures as general as
% inf-semilattices, and as constrained as, say, complete Boolean
% algebras.  However, we shall find it convenient to require $L$ to be a
% \emph{coframe}.  This is where the whole theory works best.
Interesting inf-semilattices with extra structure are complete Boolean
algebras, such as $\pow (X)$, and also frames and coframes.  Let us
recall that a \emph{frame} is a complete lattice in which arbitrary
suprema distribute over binary infima.  A \emph{coframe} is a poset
$L$ whose opposite $L^{op}$ is a frame.  In other words, a coframe is
a complete lattice in which arbitrary infima distribute over binary
suprema.  Morphisms of coframes are required to preserve all infima,
and finite suprema.

% We shall develop the theory for general lattices $L$ in \cite{convlattice}
% in order to better clarify the relationship with classical pointfree
% topology, but the simplified setting of frames is more relevant for
% our purpose in the present paper. 

In order to justify our definitions, we shall often start with the
standard definition, using points, and massage it into a definition
that no longer mentions points, and recast it in an arbitrary poset,
lattice, or coframe, depending on what we require.

Honor to whom honor is due, let us start with the notion of
\emph{limit}.  Writing $\lm_{\pow(X)}\F$ for the set of limits of a
filter of subsets $\F$, and ordering both filters and $\pow(X)$ by
inclusion, the Monotonicity Axiom states that $\lm_{\pow(X)}$ is
monotonic.  The Point Axiom is meaningless in a pointfree setting, and
accordingly we ignore it for the time being.  It will resurface
naturally later.
\begin{defn}
  \label{defn:convposet}
  A \emph{convergence semilattice} is an inf-semilattice $L$
  together with a monotonic map $\lm_{L}\colon\Filt (L)\to L$.
\end{defn}

\begin{rem}
  \label{rem:ideal}
  There is another view on convergence semilattices, obtained as
  follows.  A filter on $L$ is an \emph{ideal} of $L^{op}$.  Then
  $\Filt L$ is equal to $\Idl (L^{op})$, where $\Idl$ denotes ideal
  completion (the poset of all ideals, ordered by inclusion).  A
  convergence semilattice $L$ is then exactly the same thing as a
  sup-semilattice $\Omega := L^{op}$, together with an \emph{antitone}
  map $\lm \colon \Idl (\Omega) \to \Omega$.
\end{rem}

By extension, a convergence lattice is a lattice $L$ with a monotonic
map, and similarly for convergence Boolean algebras, convergence
frames, convergence coframes.  By varying the underlying category of
inf-semilattices, we obtain:
%  We will see what kind
% of pointfree analogue can be considered later on.
\begin{defn}[Convergence $\catc$-object]
  \label{defn:cvlat}
  % A \emph{convergence semilattice} is an inf-semilattice $L$ together
  % with a monotonic map $\lm_{L}\colon\Filt (L)\to L$.  A
  % \emph{convergence coframe} is a convergence semilattice that is also
  % a coframe.
  Let $\catc$ be category of inf-semilattices.  A \emph{convergence
    $\catc$-object} is an object $L$ of $\catc$ together with a
   map ${\lm_{L}\colon\Filt (L)\to L}$ that is monotonic.

  The category $\Cv \catc$ has as objects the convergence
  $\catc$-objects, and as morphisms $\varphi \colon L \to L'$ the
  morphisms from $L$ to $L'$ in $\catc$ that are \emph{continuous}, in
  the sense that for every $\F\in\Filt (L')$,
  \begin{equation}
    \lm_{L'}\F\leq\varphi(\lm_{L}\varphi^{-1}(\F)).
    \label{eq:morphism}
  \end{equation}

  The morphism $\varphi$ is \emph{final} if and only if equality
  holds in (\ref{eq:morphism}).
\end{defn}
We shall often write $\lm\F$ instead of $\lm_{L}\F$, when the
ambient inf-semilattice $L$ is clear from context.  We shall also
write $L$ instead of the convergence $\catc$-object $(L, \lm_L)$, when
$\lm_L$ is unambiguous.

The morphism condition (\ref{eq:morphism}) is obtained in such a way
that for every continuous map $f$ between %(pre)
convergence spaces, the inverse image map is a morphism of convergence
posets.  Let us do the exercise.  A map $f \colon X \to Y$
between %preconvergence
convergence spaces is continuous if and only if, for every filter $\F$
of subsets of $X$, every point $x$ such that $\F \to x$, namely every
point $x$ of $\lm_{\pow (X)} \F$, is such that
$f (x) \in \lm_{\pow (Y)} f [\F]$: in other words, the continuity
condition for $f$ reads
$\lm_{\pow (X)} \F \subseteq f^{-1} (\lm_{\pow (Y)} f [\F])$, and that
no longer mentions points.  Write $\varphi$ for
$f^{-1} \colon \pow (Y) \to \pow (X)$.  Then
$f [\F] = \{B \subseteq Y \mid \varphi (B) \in \F\} = \varphi^{-1}
(\F)$, and the condition
$\lm_{\pow (X)} \F \subseteq f^{-1} (\lm_{\pow (Y)} f [\F])$ is then
exactly (\ref{eq:morphism}).

\begin{rem}
  \label{rem:final:conv}
  Our notion of final morphism agrees with the usual categorical
  notion \cite[Definition 8.10]{AHS:joycats}, and is a pointfree
  analogue of \emph{initial} morphisms in $\Conv$.  Categorically, a
  morphism $\varphi \colon L \to L'$ in $\Cv \catc$ is final if and
  only if for every object $L''$ in $\catc$, for every morphism
  $\psi \colon L' \to L''$ in $\catc$ such that $\psi \circ \varphi$
  is continuous, $\psi$ is continuous.  If $\varphi$ is final in this
  sense, then take $L''=L'$, $\psi = \identity {L'}$, and
  $\lm_{L''} \F = \varphi(\lm_{L}\varphi^{-1}(\F))$: then $\varphi$ is
  final in the sense of Definition~\ref{defn:cvlat}.  The converse
  implication is immediate.  Here $\lm_{L''}$ is an example of a final
  convergence structure on $L'$.  We shall generalize this argument in
  Corollary \ref{corl:coarsest}.
\end{rem}

%Convergence coframes and their morphisms form a category $\CCF$.

% Let $\Foc=\CCF^{op}$ denote the opposite category, and call it the
% category of \emph{focales}. The name is meant to remind you of
% locales, with a smell of convergence.
Our constructions will work on various categories of lattices, but not
all.  We shall use the following, non-minimal but convenient set of
assumptions.
\begin{defn}
  \label{defn:adm}
  A category $\catc$ of inf-semilattices is \emph{admissible} if and
  only if:
  \begin{itemize}
  \item for every set $X$, $\pow (X)$ is an object of $\catc$;
  \item there are two classes of index sets $\mathcal I$ and
    $\mathcal J$ such that, for all objects $L$ and $L'$ of $\catc$,
    the morphisms from $L$ to $L'$ are exactly the monotonic maps that
    preserve all $I$-indexed infima that exist in $L$,
    $I \in \mathcal I$, and all $J$-indexed suprema that exist in
    $L$, $J \in \mathcal J$.
  \end{itemize}
\end{defn}
For example, the categories of complete Boolean algebras, of frames,
of coframes, are admissible.  In the case of frames, $\mathcal I$ can
be taken to be the class of finite sets and $\mathcal J$ can be taken
to be the class of all sets.

\begin{prop}
  \label{prop:Conv->Foc}
  Let $\catc$ be an admissible category of inf-semilattices.  There
  is a functor $\pow\colon\Conv %\PreConv
  \to{(\Cv \catc)}^{op}$ defined:
  \begin{itemize}
  \item on objects by: $\pow (X)$ is the powerset of $X$, with
    inclusion ordering $\subseteq$, and with
    $\lm_{\pow (X)}\F=\{x\in X\mid\F\to x\}$;
  \item on morphisms by $\pow (f) (S)=f^{-1}(S)$, for every continuous
    map $f\colon X\to Y$ and every $S\in\pow(Y)$.
  \end{itemize}
\end{prop}
\begin{proof}
  The main point consists in checking that, for a continuous map
  $f\colon X\to Y$, $\pow (f)$ is a morphism of convergence
  $\catc$-objects from $\pow (Y)$ to $\pow (X)$.  $\pow (f)$ preserves
  all infima and all suprema, hence is certainly a morphism in $\catc$
  by admissibility.  Checking (\ref{eq:morphism}) was the purpose of
  our preliminary exercise.
  % For every $\F\in\Filt (\pow X)$,
  % we need to check that $\lm\F\subseteq\pow f(\lm{(\pow f)}^{-1}(\F))$.
  % Note that ${(\pow f)}^{-1}(\F)=\{B\subseteq Y\mid\pow f(B)\in\F\}=\{B\subseteq Y\mid f^{-1}(B)\in\F\}=f[\F]$,
  % so that what we have to check is $\lm\F\subseteq f^{-1}(\lm f[\F])$.
  % Since every limit of $\F$ is mapped by $f$ to a limit of
  % $f[\F]$, the result follows. 
\end{proof}
% The functor $\pow$ naturally restricts to one from the smaller
% category $\Conv$ to $\Cv \catc$.

\subsection{The point functor}
\label{sec:point-functor}

We now go the other way around, and define a so-called point functor
$\pt\colon\Cv\catc\to\Conv$.

To do so, we shall need to replace $\catc$ by an (admissible) category
\emph{of lattices}.  By that we mean that all objects of $\catc$
should be lattices, but also that all morphisms should be lattice
morphisms, that is, they should preserve all finite suprema and finite
infima.  A similar convention will apply later, when we further
restrict to categories of frames, or of coframes.  We will never need
to restrict to categories of complete Boolean algebras (despite the
temptation provided by $\pow (X)$), and coframes will turn out to be
the right notion, in particular to define the functor $\Sl$ in the
final part of this paper (Lemma~\ref{lemma:Sl}).

There is a general category-theoretic notion of point of an object $L$
in a category.  A point is a morphism from a terminal object $1$ to
$L$.  In $\Conv$, the terminal object is the one-element set
$1 = \{*\}$, with the only possible notion of convergence on it.
$\pow (1)$, where $\pow$ is the functor given in
Proposition~\ref{prop:Conv->Foc}, is the two-element coframe
$\{\emptyset, 1\}$, with $\lm_{\pow (1)}$ mapping every filter to $1$.
(We must have $\lm_{\pow (1)} \{1\} = \{*\}=1$, since
$\{1\} = \dot *$, and by monotonicity
$\lm_{\pow (1)} \{\emptyset, 1\} = 1$.  This completes the argument,
since $\{1\}$ and $\{\emptyset, 1\}$ are the only two filters of
$\pow (1)$.)
\begin{lem}
  \label{lemma:1}
  Let $\catc$ be an admissible category of lattices.  $\pow (1)$ is a
  terminal object in $(\Cv \catc)^{op}$.
\end{lem}
\begin{proof}
  In other words, we claim that $\pow (1)$ is an initial object in
  $\Cv \catc$.  It is an object at all, by admissibility.  Let $L$ be
  an arbitrary convergence $\catc$-object.  Any morphism
  $\varphi \colon \pow (1) \to L$, being a lattice morphism, must map
  $1$ to the top element $\top$ of $L$, and $\emptyset$ to the bottom
  element $\bot$ of $L$.  (That would not work with a mere
  inf-semilattice morphism, hence our assumption that $\catc$ is a
  category of lattices.)  Now let $\varphi$ be that map.  For every
  $\F \in \Filt (L)$,
  $\varphi (\lm_{\pow (1)} \varphi^{-1} (\F)) = \varphi (1) = \top
  \geq \lm_L \F$, establishing (\ref{eq:morphism}).
\end{proof}
Accordingly:
\begin{defn}[Point]
  \label{defn:point}
  Let $\catc$ be an admissible category of lattices.  For every
  convergence $\catc$-object $L$, the \emph{points} of $L$ are the
  morphisms from $L$ to $\pow (1)$ in $\Cv \catc$.  We write $\pt L$
  for the set of points of $L$.
\end{defn}
The notion of point therefore depends on the chosen category $\catc$.
We explore the notion, in a few selected cases, in the following
remarks.

\begin{rem}
  \label{rem:point}
  In all cases, it will be profitable to note that
  $\varphi \colon L \to \pow (1)$ is continuous, in the sense that it
  satisfies (\ref{eq:morphism}), if and only if
  $\varphi (\lm_L \varphi^{-1} (\G))=1$ for every filter $\G$ on
  $\pow (1)$ (since $\lm_{\pow (1)} \G=1$), if and only if
  $\varphi (\lm_L \varphi^{-1} (\{1\}))=1$ (since all involved maps
  are monotonic, and $\{1\}$ is the smallest filter on $\pow (1)$), if
  and only if $\lm_L\F \in \F$, where $\F$ is the filter
  $\varphi^{-1} (\{1\})$.  Hence it is fair to equate the points of
  $L$ in $\Cv\catc$ with certain filters $\F$ such that
  $\lm_L \F \in \F$.
\end{rem}

\begin{rem}[Points of convergence lattices]
  \label{rem:latt:point}
  Let $\catc$ be the category of lattices.  A map
  $\varphi \colon L \to \pow (1)$ is a 
   if and only if
  $\F = \varphi^{-1} (\{1\})$ is a \emph{prime} filter, namely a
  proper filter that does not contain $\ell_1 \vee \ell_2$ unless it
  already contains $\ell_1$ or $\ell_2$.  Points can then be equated
  with prime filters $\F$ such that $\lm_L \F \in \F$.
\end{rem}

\begin{rem}[Points of convergence frames]
  \label{rem:cf:point}
  Let $\catc$ be the category of frames, or more generally any
  admissible category of frames.  A map
  $\varphi \colon L \to \pow (1)$ is a frame morphism if and only if
  $\F = \varphi^{-1} (\{1\})$ is a \emph{completely prime} filter, in
  the sense that $\bigvee_{i \in I} \ell_i \in \F$ implies
  $\ell_i \in \F$ for some $i \in I$.  The space of completely prime
  elements of the frame $L$ will be written $\ptS L$, and is one half
  of the famous Stone adjunction between topological spaces and
  frames.  Hence points of convergence frames can be equated with
  those completely prime filters $\F$ such that $\lm_L \F \in \F$.

  In turn, the completely prime filters $\F$ are exactly the families
  $\complement {\dc \ell}$, where $\ell$ is a meet-prime element, the
  largest element of $L$ that is not in $\F$.  (We write $\complement$
  for complement.)  An element $\ell$ of $L$ is \emph{meet-prime} if
  and only if $\bigwedge_{i=1}^n \ell_i \leq \ell$ implies
  $\ell_i \leq \ell$ for some $i \in I$.  We can then equate points of
  convergence frames with meet-primes $\ell$ such that
  $\lm_L {\complement {\dc \ell}} \not\leq \ell$.
\end{rem}

\begin{rem}
  \label{rem:cf:point:s}
  Following up on Remark~\ref{rem:cf:point}, $\pt L$ is included in
  $\ptS L$ for every frame $L$, but one should stress that the
  inclusion is in general proper.  Consider for example the
  \emph{discrete} convergence defined by $\lm_L \F = \bot$, for every
  filter $\F$ on $L$.  In that case, $\pt L$ is empty, although $\ptS
  L$ can be arbitrary large.
\end{rem}

\begin{rem}[Points of convergence coframes]
  \label{rem:ccf:point}
  Let $\catc$ be the category of coframes, or more generally any
  admissible category of coframes.  A map
  $\varphi \colon L \to \pow (1)$ is a coframe morphism if and only if
  $\F = \varphi^{-1} (\{1\})$ is a filter that is closed under
  arbitrary infima.  Such filters are exactly those of the form
  $\upc \ell$, where $\ell$ is a \emph{join-prime} of $L$, namely: for any
  finite family of elements $\ell_{1},\cdots,\ell_{n}$ of $L$ such
  that $\ell \leq \bigvee_{i=1}^n \ell_i$, there is an index $i$ such
  that $\ell \leq \ell_i$ already.  Equivalently, a join-prime is an
  element different from $\bot$ such that $\ell \leq \ell_1 \vee \ell_2$
  implies $\ell \leq \ell_1$ or $\ell \leq \ell_2$.

  % A point of a convergence coframe $L$ is therefore a coframe morphism
  % $\varphi \colon L \to \pow (1)$ such that, for every
  % $\F \in \Filt (\pow (1))$,
  % $\lm_{\pow (1)} \F \subseteq \varphi (\lm_L \varphi^{-1} (\F))$.
  % Since $\lm_{\pow (1)} \F=1$, the latter condition reduces to
  % $\varphi (\lm_L \varphi^{-1} (\F))=1$ for every
  % $\F \in \Filt (\pow (1))$, and since there are only two filters on
  % $\pow (1)$, one being smaller than the other, this further reduces
  % to $\varphi (\lm_L \varphi^{-1} (\{1\}))=1$.

  % The coframe morphism $\varphi$ is entirely characterized by the
  % subset $\varphi^{-1} (\{1\})$ of $L$, and because $\varphi$
  % preserves the top element and all infima, there is a least element
  % in $\varphi^{-1} (\{1\})$, call it $x$.  Then
  % $\varphi^{-1} (\{1\}) = \upc x$, the upward closure of $x$ in $L$.

  % Since $\varphi$ also preserves finite suprema, $x$ is a
  % \emph{join-prime} of $L$, namely:

  % Conversely, any join-prime $x \in L$ gives rise to a coframe
  % morphism $\varphi \colon L \to \pow (1)$ by letting $\varphi$ map
  % every element of $\upc x$ to $1$, and all other elements to
  % $\emptyset$.

  It is therefore natural to (re)define the points of a convergence
  coframe $L$ as the join-primes $\ell$ of $L$ such that
  $\ell \leq \lm_L {\upc \ell}$.
\end{rem}

In general, given an admissible category $\catc$ of lattices, it is
worthwhile to see what the points of the convergence $\catc$-object
$\pow (X)$ are, for a given %(pre)
convergence space $X$.
We shall see the Point Axiom emerge naturally here.
\begin{rem}
  \label{rem:ptpow}
  Let $\catc$ be an admissible category of lattices, and $X$ be a
  convergence space.  For every $x \in X$, consider the constant map
  $\overline x \colon * \mapsto x$ from $1$ to $X$.  By
  Proposition~\ref{prop:Conv->Foc}, $\pow (\overline x)$ is a morphism
  from $\pow (X)$ to $\pow (1)$ in $\catc$.  By
  Remark~\ref{rem:point}, this can be equated with the filter
  $\pow (\overline x)^{-1} (\{1\}) = \{S \in \pow (X) \mid {\overline
    x}^{-1} (S) = 1\} = \{S \in \pow (X) \mid x \in S\} = \dot x$.
  That always defines a point, in $\pt L$: the condition of being a
  point reads $\lm_{\pow (X)} \dot x \in \dot x$, and is exactly what
  the Point Axiom states.  Hence $\pt {\pow (X)}$ always contains the
  points $\pow (\overline x) \cong \dot x$, for every $x \in X$.
\end{rem}

\begin{rem}[Points of $\pow (X)$ qua convergence lattice]
  \label{rem:latt:ptpow}
  Let $\catc$ be the category of all lattices.  The prime filters of
  $\pow (X)$ are exactly the ultrafilters on $X$.  Hence the points of
  $\pow (X)$, qua convergence lattice, are the ultrafilters $\U$ such
  that $\lm_{\pow (X)} \U \in \U$.  Those are exactly the
  \emph{compact} ultrafilters on $X$.  (We let the reader check that
  this notion agrees with the usual notion of compactness for filters,
  once specialized to ultrafilters
  \cite[Definition~IX.7.1]{DM:convergence}.)  Among all compact
  ultrafilters, one finds the principal ultrafilters $\dot x$, if only
  because of Remark~\ref{rem:ptpow}, but there are others in general.
  Look indeed at the special case where $X$ is topological: then every
  compact ultrafilter is principal if and only if $X$ is Noetherian;
  hence any non-Noetherian topological space will have
  non-principal compact ultrafilters.
\end{rem}

\begin{rem}[Points of $\pow (X)$ qua convergence frame]
  \label{rem:cf:ptpow}
  Let $\catc$ be the category of frames, or more generally any
  admissible category of frames.  The situation is much simpler here.
  We rely on Remark~\ref{rem:cf:point}.  The meet-primes $\ell$ of
  $\pow (X)$ are exactly the complements of one-element sets $\{x\}$,
  $x \in X$.  The condition
  $\lm_L {\complement {\dc \ell}} \not\leq \ell$ rewrites as
  $x \in \lm_{\pow (X)} {\dot x}$.  The latter is just the Point Axiom
  $\dot x \to x$.  Hence $\pt {\pow (X)}$ can be equated with $X$
  itself.
\end{rem}

\begin{rem}[Points of $\pow (X)$ qua convergence coframe]
  \label{rem:ccf:ptpow}
  Let $\catc$ be the category of coframes, or more generally any
  admissible category of coframes.  The situation is equally simple.
  We rely on Remark~\ref{rem:ccf:point}.  The join-primes $\ell$ of
  $\pow (X)$ are exactly the one-element sets $\{x\}$, $x \in X$.  The
  condition $\ell \leq \lm_L {\upc \ell}$ rewrites as
  $x \in \lm_{\pow (X)} {\dot x}$.  As for frames, this is the Point
  Axiom.  Hence, again, $\pt {\pow (X)}$ can be equated with $X$
  itself.
\end{rem}

In order to define a $\lm$ operator on $\pt L$, we introduce the
following.  For a filter $\F$ of subsets of $\pt L$, $\kow \F$ will be
a corresponding filter on $L$.
\begin{defn}
  \label{defn:kow}
  Let $\catc$ be a category of lattices, and $L$ be a convergence
  $\catc$-object.  For every $\ell \in L$, let $\ell^\bullet$ be the
  set of points $\varphi \in \pt L$ such that $\varphi (\ell)=1$.

  For every $\F \in \Filt \pow (\pt L)$, let:
  \[
    \kow{\F} := \{\ell \in L \mid \ell^\bullet \in \F\}.
  \]
\end{defn}

\begin{lem}
  \label{lemma:kow:basic}
  Let $\catc$ be an admissible category of lattices, and $L$ be a
  convergence $\catc$-object.  The following hold:
  \begin{enumerate}
  \item For all $\ell, \ell'$ in $L$, if $\ell \leq \ell'$ then
    $\ell^\bullet \subseteq {\ell'}^\bullet$.
  \item For all $\F, \G \in \Filt\pow (\pt L)$, if $\F\subseteq\G$
    then $\kow{\F}\subseteq\kow{\G}$.
  \item For all $\ell_{1},\ell_{2},\cdots,\ell_{n}\in L$ ($n\in\nat$),
    $\bigcup_{i=1}^{n}\ell_{i}^{\bullet}={(\bigvee_{i=1}^{n}\ell_{i})}^{\bullet}$.
  \item For all $\ell_{1},\ell_{2},\cdots,\ell_{n}\in L$ ($n\in\nat$),
    $\bigcap_{i=1}^{n}\ell_{i}^{\bullet}={(\bigwedge_{i=1}^{n}\ell_{i})}^{\bullet}$.
  \item For every $\F\in\Filt \pow(\pt L)$, $\kow{\F}$ is a filter on
    $L$, i.e., $\mathcal{\kow{F}}\in\Filt L$.
  \item For every $\varphi \in \pt L$,
    $\kow {\dot \varphi}=\varphi^{-1} (\{1\})$.
  \item For every $x \in \pt L$, $x$ is in
    $(\lm_L \kow {\dot x})^\bullet$.
  \end{enumerate}
\end{lem}
\begin{proof}
  (1) and (2) are obvious.  (3) For every point $\varphi$, $\varphi$
  is in ${(\bigvee_{i=1}^{n}\ell_{i})}^{\bullet}$ if and only if
  $\varphi (\bigvee_{i=1}^n \ell_i)=1$, if and only if
  $\bigvee_{i=1}^n \varphi (\ell_i)=1$ (since $\varphi$ is a lattice
  morphism), if and only if $\varphi (\ell_i)=1$ for some $i$, if and
  only if $\varphi$ is in some $\ell_i^\bullet$.  (4) is proved
  similarly.

  (5) We first check that $\F$ is upwards-closed.  Let
  $\ell\in\kow{\F}$ and $\ell\leq\ell'$.  By definition,
  $\ell^{\bullet}$ is in $\F$.  Using (1), ${\ell'}^\bullet$ is also
  in $\F$, so $\ell'$ is in $\kow\F$.  $\kow\F$ is non-empty: since
  $\top^\bullet = \pt L$ is in $\F$, $\top$ is in $\kow\F$.  Finally,
  for any two elements $\ell$, $\ell'$ of $\kow\F$, $\ell^\bullet$ and
  ${\ell'}^\bullet$ are in $\F$, so
  $\ell^\bullet \cap {\ell'}^\bullet$ is in $\F$.  Using (4),
  ${(\ell \wedge \ell')}^\bullet$ is in $\F$, so $\ell \wedge \ell'$
  is in $\kow\F$.

  (6) The set $\kow {\dot\varphi}$ is the set of elements $\ell \in L$ such
  that $\ell^\bullet \in \dot\varphi$, equivalently such that
  $\varphi \in \ell^\bullet$, equivalently such that
  $\varphi (\ell)=1$.

  (7) For every $\varphi \in \pt L$, by Remark~\ref{rem:point},
  $\varphi (\lm_L \varphi^{-1} (\{1\}))=1$.  Using (6), this means
  that $\varphi (\lm_L \kow {\dot\varphi})=1$, and by definition this
  is equivalent to $\varphi \in (\lm_L \kow {\dot\varphi})^\bullet$.
\end{proof}
This allows us to define:
\begin{defn}
  \label{defn:pt:lim}
  Let $\catc$ be a category of lattices.  For every convergence
  $\catc$-object $L$, define $\lm_{\pow (\pt L)} \F$ as
  $(\lm_L \kow\F)^\bullet$.  In other words, $\F \to x$ in $\pt L$ if
  and only if $x \in (\lm_L \kow\F)^\bullet$.
\end{defn}

\begin{lem}
  \label{lemma:pt:lim:conv}
  Let $\catc$ be a category of lattices.  For every convergence
  $\catc$-object $L$, $(\pt L, \to)$ as defined in
  Definition~\ref{defn:pt:lim} is a convergence space.
\end{lem}
\begin{proof}
  By (1) and (2) of Lemma~\ref{lemma:kow:basic}, and since $\lm_L$ is
  monotonic, $\lm_{\pow (\pt L)}$ is monotonic as well.  The Point
  Axiom is item (7) of the same Lemma.
\end{proof}
In the following, $\pt L$ will always define the convergence space
obtained with that notion of convergence.

% \begin{rem}
%   \label{rem:pt:lim:conv}
%   Since points of $L$ are certain (compact ultra)filters $\U$, it is
%   useful to translate the above constructions in terms of filters.
%   For every $\ell \in L$, $\ell^\bullet$ is the set of points
%   $\U \in \pt L$ such that $\ell \in \U$.  For a filter $\F$ of
%   subsets of $\pt L$,
%   $\kow \F = \{\ell \in L \mid \ell^\bullet \in \F\}$
% \end{rem}

\subsection{The $\pow\dashv\pt$ adjunction}
\label{sec:powd-adjunct}

The $\pt$ construction satisfies the following universal property.
\begin{prop}
  \label{prop:pt:univ}
  Let $\catc$ be an admissible category of lattices, $X$ be a
  convergence space, and $L$ be a convergence $\catc$-object.  For
  every morphism $\varphi \colon L \to \pow (X)$ in $\Cv\catc$, there
  is a unique map $\varphi^\dagger \colon X \to \pt L$ such that, for
  every $\ell \in L$,
  $\pow (\varphi^\dagger) (\ell^\bullet) = \varphi (\ell)$, and it is
  continuous.
\end{prop}
\begin{proof}
  If $\varphi^\dagger$ exists, then for every $\ell \in L$,
  \begin{eqnarray}\label{eq:varphidagg}
    \pow (\varphi^\dagger) (\ell^\bullet)
    & = & {\varphi^\dagger}^{-1} (\ell^\bullet)\nonumber \\
    & = & \{x \in X \mid \varphi^\dagger (x) \in \ell^\bullet\} \\
    & = & \{x \in X \mid \varphi^\dagger (x) (\ell)=1\}\nonumber 
  \end{eqnarray}
  must be equal to $\varphi (\ell)$.  This forces
  $\varphi^\dagger (x) (\ell)$ to be equal to $1$ if $x \in \varphi
  (\ell)$, and to $\emptyset$ otherwise, establishing uniqueness.

  Now define $\varphi^\dagger$ as mapping each $x \in X$ to the map:
  \begin{equation}
    \label{eq:dagger}
    \ell \mapsto \left\{
      \begin{array}{ll}
        1 & \text{if }x \in \varphi (\ell) \\
        \emptyset & \text{otherwise}.
      \end{array}
    \right.
  \end{equation}
  We show that $\varphi^\dagger (x)$ is a point of $L$, that is, an
  element of $\pt L$, by the following argument.  Recall from
  Remark~\ref{rem:ptpow} that, for every $x \in X$,
  $\pow (\overline x)$ is a point of $\pow (X)$, i.e., a morphism from
  $\pow (X)$ to $\pow (1)$ in $\Cv\catc$.  Now we check that
  $\varphi^\dagger (x) = \pow (\overline x) \circ \varphi$, and is
  therefore a morphism from $L$ to $\pow (1)$ in $\Cv\catc$, that is,
  a point of $L$.

  For every $\ell \in L$, in view of \eqref{eq:varphidagg},
  \begin{equation*}
    \pow (\varphi^\dagger) (\ell^\bullet)= \{x \in X \mid \varphi^\dagger (x) (\ell)=1\}= \varphi (\ell),
  \end{equation*}
  by \eqref{eq:dagger}.
  It remains to show that $\varphi^\dagger$ is continuous.  For every
  filter $\F$ of subsets of $X$, the image filter
  $\varphi^\dagger [\F]$ is the set of subsets $B$ of $\pt L$ such
  that ${\varphi^\dagger}^{-1} (B) \in \F$.  In particular,
  $\kow {\varphi^\dagger [\F]}$ is the set of elements $\ell \in L$
  such that $\ell^\bullet \in \varphi^\dagger [\F]$, namely such that
  ${\varphi^\dagger}^{-1} (\ell^\bullet) \in \F$.  We have just seen
  that ${\varphi^\dagger}^{-1} (\ell^\bullet) = \varphi (\ell)$.
  Therefore $\kow {\varphi^\dagger [\F]} = \varphi^{-1} (\F)$.  The
  continuity condition (\ref{eq:morphism}) reads
  $\lm_{\pow (X)} \F \subseteq \varphi (\lm_L \varphi^{-1} (\F))$.  In
  particular, if $\F$ converges to $x$ in $X$, $x$ must belong to
  $\varphi (\lm_L \varphi^{-1} (\F)) = \varphi (\lm_L \kow
  {\varphi^\dagger [\F]})$.  Recall from (\ref{eq:dagger}) that
  $x \in \varphi (\ell)$ if and only if
  $\varphi^\dagger (x) (\ell)=1$, for every $\ell \in L$.  Hence
  $\varphi^\dagger (x) (\ell)=1$ where
  $\ell = \lm_L \kow {\varphi^\dagger [\F]}$.  By definition, this
  means that $\varphi^\dagger (x)$ is in $\ell^\bullet$.  However,
  $\ell^\bullet = {(\lm_L \kow {\varphi^\dagger [\F]})}^\bullet =
  \lm_{\pow (\pt L)} \varphi^\dagger [\F]$ by
  Definition~\ref{defn:pt:lim}.  Therefore $\varphi^\dagger (x)$ is in
  $\lm_{\pow (\pt L)} \varphi^\dagger [\F]$, and this completes our
  continuity argument.
\end{proof}

\begin{lem}
  \label{lemma:epsilon}
  Let $\catc$ be an admissible category of lattices, and $L$ be a
  convergence $\catc$-object.  The map
  $\epsilon_L \colon L \to  \pow (\pt L)$ defined by $\epsilon_L(\ell)=\ell^\bullet$  is a morphism in $\Cv\catc$, and is final.
\end{lem}
\begin{proof}
  Consider the classes $\mathcal I$ and $\mathcal J$ posited in
  Definition~\ref{defn:adm}.  By the same argument as for
  Lemma~\ref{lemma:kow:basic}~(3) and~(4), we show that $\epsilon_L$
  preserves $I$-indexed infima, $I \in \mathcal I$, and $J$-indexed
  suprema, $J \in \mathcal J$.  Explicitly,
  ${(\bigwedge_{i \in I} \ell_i)}^\bullet$ is the set of points
  $\varphi$ such that
  $\varphi (\bigwedge_{i \in I} \ell_i) = \bigwedge_{i \in I} \varphi
  (\ell_i)$ is equal to $1$ (since $\varphi$ preserves $I$-indexed
  infima), namely such that $\varphi (\ell_i)=1$ for every $i \in I$,
  and this is the set $\bigcap_{i \in I} \ell_i^\bullet$.  As far as
  suprema are concerned, the argument is similar, and uses the fact
  that a supremum of elements in the two-element lattice $\pow (1)$ is
  equal to $1$ if and only if one of those elements is equal to $1$.
  Therefore $\epsilon_L$ is a morphism in $\catc$.

  As far as continuity is concerned, let $\F$ be a filter of subsets
  of $\pt L$.  Noticing that $\kow\F = \epsilon_L^{-1} (\F)$, we have:
  \[
    \lm_{\pow (\pt L)} \F = {(\lm_L \kow \F)}^\bullet
    = \epsilon_L (\lm_L \epsilon_L^{-1} (\F))
  \]
  which shows that $\epsilon_L$ is not just continuous but final.
  % on veut que \ell \mapsto \{\varphi \mid \varphi (\ell)=1\} soit un morphisme
\end{proof}

By general categorical arguments (e.g., \cite[Chap. IV, Sec. 1, Theorem 2]{mac1971category}), one can define an adjoint pair of
functors $F \dashv U$ by specifying a functor $U$, a family of
morphisms $\epsilon_A \colon A \to UF (A)$ for each object $A$, in
such a way that for every morphism $f \colon A \to U (B)$, there is a
unique morphism $f^\dagger \colon F (A) \to B$ such that
$U (f^\dagger) \circ \epsilon_A = f$.  Here $U = \pow$, $F = \pt$,
$\epsilon_L$ maps $\ell$ to $\ell^\bullet$ and is a morphism by
Lemma~\ref{lemma:epsilon}, so that Proposition~\ref{prop:pt:univ}
defines an adjoint pair of functors $\pt \dashv \pow$ between
$\Conv^{op}$ and $\Cv\catc$.

We take the opposite categories, and obtain a dual adjunction where
$\pt$ and $\pow$ are exchanged.
\begin{thm}
  \label{thm:A-|pt}
  Let $\catc$ be an admissible category of lattices.  Then
  $\pow\dashv\pt$ is an adjunction between $\Conv$ and
  ${(\Cv\catc)}^{op}$.
\end{thm}
Referring again to Proposition~\ref{prop:pt:univ}, the unit
$\epsilon_L \colon  L \to  \pow (\pt L)$ (defined by $\epsilon_L(\ell)=\ell^\bullet$)  of
the dual adjunction $\pt\dashv\pow$ in the opposite categories is the
\emph{counit} of the adjunction $\pow\dashv\pt$.

In our recollection of adjoint pairs $F \dashv U$, we required $U$ to
be a functor, but not $F$.  It is a fact that $F$ automatically gives
rise to a functor in the converse direction, whose action on morphisms
$\varphi \colon A \to A'$ is given by
$F (\varphi) = {(\epsilon_{A'} \circ \varphi)}^\dagger$.   As a result:
\begin{lem}
 Let $\catc$ be an admissible category of lattices, and let $\varphi \colon L \to L'$ be a morphism in $\Cv\catc$. Then	
 \[\pt\varphi = \varphi^{-1}.\]
\end{lem}
\begin{proof}
In view of the discussion above,  $\pt \varphi \colon \pt L' \to \pt L$ is equal to
${(\epsilon_{L'} \circ \varphi)}^\dagger$. By (\ref{eq:dagger}), for
every $\psi \in \pt {L'}$ and for every $\ell \in L$,
$\pt \varphi (\psi) (\ell) = 1 $ if and only if
$\psi \in \epsilon_{L'} (\varphi (\ell)) = \varphi (\ell)^\bullet$, if
and only if $\psi (\varphi (\ell))=1$, from which we deduce that
$\pt \varphi (\psi) = \psi \circ \varphi$.  If we equate points
$\psi \in \pt {L'}$ with filters $\F := \psi^{-1} (\{1\})$ as in
Remark~\ref{rem:point}, we obtain that
$\pt\varphi (\F) = \varphi^{-1} (\F)$, that is,
$\pt\varphi = \varphi^{-1}$.
\end{proof}

The \emph{unit} $\eta_X \colon X \to \pt\pow(X)$ of the adjunction is
$\identity {\pow (X)}^\dagger$.  Recall that
$\varphi^\dagger (x) = \pow (\overline x) \circ \varphi$ for any
$\varphi$, so $\eta_X$ maps each $x \in X$ to
$\pow (\overline x) \in \pt\pow (X)$.  Equating points with certain
filters as in Remark~\ref{rem:ptpow}, $\eta_X$ simply maps $x$ to the
principal filter $\dot x$.

The map $\eta_X$ has additional properties.
\begin{lem}
  \label{lemma:eta}
  Let $\catc$ be an admissible category of lattices.  The map
  $\eta_X \colon X \to \pt (\pow X)$ is injective and initial.  It is
  an isomorphism if $\catc$ is an admissible category of frames or of
  coframes.
\end{lem}
\begin{proof}
  We equate points of $\pow (X)$ with certain filters $\U$, such that
  $\lm_{\pow (X)} \U \in \U$, following Remark~\ref{rem:point}.  Then
  $\eta_X (x) = \dot x$, and $\ell^\bullet = \{\U \in \pt \pow (X)
  \mid \ell \in \U\}$ for every $\ell \in \pow (X)$.  That $\eta_X$ is
  injective is clear.
  % We know that $\eta_X$
  % is continuous, being a morphism in $\Conv$.
  To establish initiality, let $\F \in \Filt\pow (X)$ be arbitrary,
  and compute $\kow {(\eta_X [\F])}$: this is the set of elements
  $\ell \in \pow (X)$ such that $\ell^\bullet \in \eta_X [\F]$, i.e.,
  such that $\eta_X^{-1} (\ell^\bullet) \in \F$.  However,
  \begin{eqnarray}
    \nonumber
    \eta_X^{-1} (\ell^\bullet)
    & = & \{x \in X \mid \dot x \in \ell^\bullet\} \\    
    & = & \{x \in X \mid \ell \in \dot x\} \\
    & = & \{x \in X \mid x \in \ell\} = \ell.\nonumber
    \label{eq:eta:bullet}
  \end{eqnarray}
  Therefore $\kow {(\eta_X [\F])}$ is just $\F$.  It follows that
  \[\lm_{\pow (\pt\pow (X))} \eta_X [\F] = {(\lm_{\pow (X)} \kow
    {(\eta_X [\F])})}^\bullet = {(\lm_{\pow (X)} \F)}^\bullet.\]  In
  particular, for every $x \in X$, $\eta_X [\F]$ converges to
  $\eta_X (x)$ if and only if
  $\eta_X (x) \in {(\lm_{\pow (X)} \F)}^\bullet$, if and only if
  $x \in \lm_{\pow (X)} \F$ (using (\ref{eq:eta:bullet})), if and only
  if $\F$ converges to $x$ in $X$.

  If $\catc$ is an admissible category of frames, by
  Remark~\ref{rem:cf:point} and Remark~\ref{rem:cf:ptpow}, the only
  points of $\pt (\pow X)$ are the principal filters $\dot x$, so that
  $\eta_X$ is surjective, hence an isomorphism, in this case.  In the
  case of coframes, we use Remark~\ref{rem:ccf:point} and
  Remark~\ref{rem:ccf:ptpow} instead.
\end{proof}

A functor that has a fully faithful left adjoint is called a
coreflection.  It is well-known that a coreflection is also a right
adjoint functor whose unit is an isomorphism: see
\cite[Proposition~1.3]{GZ:fractions}, where this is stated for the
dual case of reflections.  We note that the proof does not make use of
the Axiom of Choice.

\begin{cor}
  \label{corl:corefl}
  Let $\catc$ be an admissible category of frames, or of coframes.
  $\Conv$ is a coreflective subcategory of ${(\Cv\catc)}^{op}$,
  through the coreflection $\pt$.
\end{cor}
We have just observed that $\pow \colon \Conv \to {(\Cv\catc)}^{op}$
is fully faithful in that case, allowing us to equate convergence
spaces $X$ with the convergence (co)frames $\pow (X)$.

\subsection{Classical convergence lattices}
\label{sec:class-conv-struct}

We shall progressively discover the importance of complemented
elements of lattices $L$.  Recall that an element $\ell$ of $L$ is
\emph{complemented} if and only if it has a \emph{complement}
$\overline\ell$, namely an element such that
$\ell \vee \overline\ell = \top$ and
$\ell \wedge \overline\ell = \bot$.  If $L$ is distributive, the
complement is unique if it exists, and
$\overline{\overline\ell} = \ell$. Let us denote by $\C_L$ the set of complemented elements of a lattice $L$.

\begin{defn}[Classical]
  \label{defn:classical}
  Let $\catc$ be a category of lattices.  A convergence $\catc$-object
  $L$ is \emph{classical} if and only if, for  two filters  $L$ that contain the same complemented elements have the same limit, that is, for every two filters $\F$ and $\G$ on $L$,
  \[\F\cap \C_L=\G\cap\C_L\then\lm_L \F = \lm_L \G.\]

  We write $\Cvcl\catc$ for the full subcategory of classical
  convergence $\catc$-objects in $\Cv\catc$.
\end{defn}
In other words, $\lm_L$ is classical if and only if $\lm_L \F$ only
depends on the complemented elements of $\F$.

For every upwards-closed subset $\A$ of $L$, there is a smallest
upwards-closed set $(\A)_c$ that contains the same complemented
elements as $\A$, namely:
\begin{equation}
  \label{eq:c}
  (\A)_c := \upc (\A\cap\C_L).
\end{equation}
If $\A$ is a filter $\F$, then $(\F)_c$ is a filter again, and $(\F)_c\subset \F$, hence $\lm_L (\F)_c\leq \lm_L\F$ if $L$ is a convergence $\catc$-object. Classical convergent $\catc$-objects are those for which the reverse inequality is true:
\begin{lem}
  \label{lemma:class}
  Let $\catc$ be a category of lattices.  A convergence $\catc$-object
  $L$ is classical if and only if, for every filter $\F$ on $L$,
  $\lm_L \F = \lm_L (\F)_c$ (equivalently,
  $\lm_L \F \leq \lm_L (\F)_c$).  \qed
\end{lem}

If $L$ is a Boolean algebra, that is, if every element is
complemented, then every convergence structure on $L$ is classical.
This is the case in particular of the convergence structure on
$\pow (X)$, for every convergence space $X$.  Using
Theorem~\ref{thm:A-|pt} and Lemma~\ref{lemma:eta}, it follows that:
\begin{prop}
  \label{prop:A-|pt:c}
  Let $\catc$ be an admissible category of lattices.  The adjunction
  $\pow\dashv\pt$ restricts to an adjunction between $\Conv$ and
  ${(\Cvcl\catc)}^{op}$.  This is a coreflection if $\catc$ is an
  admissible category of frames or of coframes.  \qed
\end{prop}

\section{Completeness, Cocompleteness, and More}
\label{sec:compl-cocompl}

The category $\Frm$ of frames is complete and cocomplete
\cite[Section~IV.3]{framesandlocales}, and therefore so is the
category $\coFrm$ of coframes.  In order to extend this to $\Cv\Frm$
or $\Cv\coFrm$, or more generally to categories of the form
$\Cv\catc$, a swift approach is to show that the forgetful functor
$| \_ | \colon \Cv \catc \to \catc$, which maps any convergence
$\catc$-object $(L, \lm_L)$ to the underlying $\catc$-object $L$, is
topological (see \cite{AHS:joycats}, Chapter~VI, Section~21; we shall
recall the necessary elements as we progress along).  For that, we
shall see that $\catc$ had better be a category of coframes.

Given a $\catc$-object $L$, its \emph{fiber} along $|\_|$ is the
collection of convergence $\catc$-objects whose image by $|\_|$ is
equal to $L$.  Hence we can equate the fiber of $L$ with the set of
monotonic maps $\lm_L \colon L \to L$.  We call those maps the
\emph{convergence structures} on $L$.

Note that $|\_|$ is \emph{fiber-small}, in the sense that every fiber
is a set, not a proper class.  That set of convergence structures on
$L$ is ordered by the pointwise ordering: $\lm_L \leq \lm'_L$ if and
only if $\lm_L \F \leq \lm_{L'} \F$ for every $\F \in \Filt L$.  In
that case, we say that $\lm_L$ is \emph{finer than} $\lm_{L'}$, or
that $\lm_{L'}$ is \emph{coarser than} $\lm_L$
\cite[Definition~III.1.7]{DM:convergence}.  With that ordering, the
fiber of $L$ is a complete lattice.
% topology discrete (la plus fine): F -> x ssi {x} in F
% topology grossiere (la moins fine): F -> x ssi true (lm = top)

In the theory of topological functors, it is common usage to order the
fibers by $L \sqsubseteq L'$ if and only if the identity morphism on
$|L| = |L'|$ lifts to a morphism from $L$ to $L'$.  In view of
(\ref{eq:morphism}), $\sqsubseteq$ is the opposite of $\leq$.

The essence of topologicity is the following simple observation.  The
collection of indices $I$ is allowed to be any class, including a
proper class.
\begin{prop}
  \label{prop:coarsest}
  Let $\catc$ be a category of coframes, and $L$ be a $\catc$-object.
  Let also $\varphi_i \colon |L_i| \to L$ be morphisms of $\catc$,
  where each $L_i$ is a convergence $\catc$-object, $i \in I$.  There
  is a coarsest convergence structure $\lm_L$ on $L$ such that
  $\varphi_i$ is continuous for each $i \in I$.
\end{prop}
\begin{proof}
  By definition, $\varphi_i$ is continuous if and only if $\lm_L \F
  \leq \varphi_i (\lm_{L_i} \varphi_i^{-1} (\F))$ for every $\F \in
  \Filt L$.  The required map is then defined by:
  \begin{equation}
    \label{eq:coarsest}
    \lm_L \F := \bigwedge_{i \in I} \varphi_i (\lm_{L_i} \varphi_i^{-1} (\F)).
  \end{equation}
  Note that arbitrary infima exist because $L$ is a coframe.
\end{proof}

\begin{prop}
  \label{prop:coarsest:c}
  Under the same assumptions as Proposition~\ref{prop:coarsest}, if
  each $L_i$ is classical, then $(L, \lm_L)$ is classical.
\end{prop}
\begin{proof}
  Assume $\F$ and $\G$ contain the same complemented objects.  We note
  that, for every complemented element $\ell$ of $L_i$, $i \in I$, we
  have
  $\varphi_i (\ell) \vee \varphi_i (\overline\ell) = \varphi_i (\ell
  \vee \overline\ell) = \varphi_i (\top) = \top$, and similarly
  $\varphi_i (\ell) \wedge \varphi_i (\overline\ell) = \bot$, so
  $\varphi_i (\ell)$ is complemented.  For each $i \in I$,
  $\varphi_i^{-1} (\F)$ and $\varphi_i^{-1} (\G)$ then have the same
  complemented elements: if $\ell$ is complemented in
  $\varphi_i^{-1} (\F)$, then $\varphi_i (\ell)$ is complemented and
  in $\F$, hence in $\G$, so $\ell$ is in $\varphi_i^{-1} (\G)$.  It
  follows that
  $\lm_{L_i} \varphi_i^{-1} (\F) = \lm_{L_i} \varphi_i^{-1} (\G)$,
  since $L_i$ is classical.  From (\ref{eq:coarsest}),
  $\lm_L \F = \lm_L \G$.
\end{proof}

Categorically, a family ${(\varphi_i \colon |L_i| \to L)}_{i \in I}$
of morphisms with a fixed $L$ is a \emph{sink}.  A \emph{lift} of that
sink if an $I$-indexed family of morphisms whose images by $|\_|$
coincide with $\varphi_i$.  In our case, this is the same thing as a
convergence structure $\lm_L$ on $L$ making every $\varphi_i$
continuous.  Such a lift is \emph{final} if and only if, for every
$\catc$-morphism $\psi \colon L \to |L'|$, $\psi$ is continuous from
$(L, \lm_L)$ to $L'$ if and only if $\psi \circ \varphi_i$ is
continuous for every $i \in I$.

\begin{cor}
  \label{corl:coarsest}
  Let $\catc$ be a category of coframes.  Then
  $|\_| \colon \Cv\catc \to \catc$ is topological: every sink
  ${(\varphi_i \colon |L_i| \to L)}_{i \in I}$ has a unique final
  lift, and this is $\lm_L$, as given in
  Proposition~\ref{prop:coarsest}.  Similarly for the restriction of
  $|\_|$ to $\Cvcl\catc$.
\end{cor}
\begin{proof}
  In view of  (the dual of) \cite[Proposition 10.43]{AHS:joycats},  if it has a final lift then it is
    the coarsest convergence structure $\lm_L$ of
  Proposition~\ref{prop:coarsest}, and is therefore unique.

  To show existence, we check that
  $(L, \lm_L)$ is a final lift of
  ${(\varphi_i \colon |L_i| \to L)}_{i \in I}$.  Let
  $\psi \colon L \to |L'|$ be such that $\psi \circ \varphi_i$ is
  continuous for every $i \in I$.  We aim to show that $\psi$ is
  continuous from $(L, \lm_L)$ to $L'$, and for that we consider an
  arbitrary filter $\F$ on $L'$, and show that
  $\lm_{L'} \F \leq \psi (\lm_L (\psi^{-1} (\F)))$.  Indeed:
  \begin{eqnarray*}
    \psi (\lm_L (\psi^{-1} (\F)))
    & = & \psi (\bigwedge_{i \in I} \varphi_i (\lm_{L_i}
          \varphi_i^{-1} (\F)))
          \quad\text{by (\ref{eq:coarsest})} \\
    & = & \bigwedge_{i \in I} \psi (\varphi_i (\lm_{L_i}
          \varphi_i^{-1} (\F)))
          \quad\text{since $\psi$ is a morphism of coframes} \\
    & \geq & \lm_{L'} \F
  \end{eqnarray*}
  since each $\psi \circ \varphi_i$ is continuous.
\end{proof}

\begin{rem}
  \label{rem:topological}
  The usual definition of a topological functor is the dual property
  that every source has a unique initial lift (dual in the sense that
  a source is a sink in the opposite category, and that a lift is
  initial if and only if it is final in the opposite category).  Being
  pedantic, what we have shown in Corollary~\ref{corl:coarsest} is
  that $|\_| \colon {(\Cv\catc)}^{op} \to \catc^{op}$ is topological.
  However, that is equivalent to $|\_| \colon \Cv\catc \to \catc$
  being topological, by the Topological Duality Theorem
  \cite[Theorem~21.9]{AHS:joycats}.
\end{rem}

% We rephrase the fact that every source has a unique initial lift as follows.
% \begin{prop}
%   \label{prop:coarsest}
%   Let $\catc$ be a category of coframes, and $L$ be a $\catc$-object.
%   Let also $\varphi_i \colon L \to |L_i|$ be morphisms of $\catc$,
%   where each $L_i$ is a convergence $\catc$-object, $i \in I$.  There
%   is a finest convergence structure $\lm_L$ on $L$ such that
%   $\varphi_i$ is continuous for each $i \in I$.  For every 
% \end{prop}

Given a topological functor between two categories $\catc$ and
$\catd$, $\catc$ is (co)complete if and only if $\catd$ is
\cite[Theorem~21.16~(1)]{AHS:joycats}.  Hence, writing $\coFrm$ for
the category of coframes:
\begin{fact}
  \label{fact:cocompl}
  The categories $\Cv\coFrm$ and $\Cvcl\coFrm$ are complete and
  cocomplete.
\end{fact}
Topologicity has many other consequences.  For example, the category
of frames is not co-wellpowered
\cite[Section~IV.6.6]{framesandlocales}, hence neither is $\coFrm$.
Given a fiber-small topological functor between two categories $\catc$
and $\catd$, $\catc$ is (co-)wellpowered if and only if $\catd$ is
\cite[Theorem~21.16~(2)]{AHS:joycats}, so:
\begin{fact}
  \label{fact:cowell}
  The categories $\Cv\coFrm$ and $\Cvcl\coFrm$ are not co-wellpowered.
\end{fact}

\section{Limit Lattices and Variants}
\label{sec:limit-spaces}

A \emph{limit space}, also known as a \emph{finitely deep} convergence
space \cite[Section~III.1]{DM:convergence}, is a convergence space $X$
where, for any two filters $\F$ and $\G$ of subsets of $X$ that
converge to the same point $x$, $\F \cap \G$ also converges to $x$.
It is equivalent to require that
$\lm_{\pow (X)} (\F \cap \G) = \lm_{\pow (X)} \F \cap \lm_{\pow (X)}
\G$.  Those form a full subcategory $\Limc$ of $\Conv$.

Accordingly, call \emph{limit $\catc$-object} any convergence
$\catc$-object $L$ such that
$\lm_L (\F \cap \G) = \lm_L \F \wedge \lm_L \G$ for all
$\F, \G \in \Filt L$, that is, such that $\lm_L$ preserves binary
infima.  Those form a full subcategory $\Lim\catc$ of $\Cv\catc$.  The
classical limit $\catc$-objects form another full subcategory
$\Limcl\catc$.
\begin{prop}
  \label{prop:A-|pt:lim}
  Let $\catc$ be an admissible category of lattices.  The adjunction
  $\pow\dashv\pt$ restricts to an adjunction between $\Limc$ and
  ${(\Lim\catc)}^{op}$, resp.\ between $\Limc$ and
  ${(\Limcl\catc)}^{op}$.
\end{prop}
\begin{proof}
  Given a limit space $X$, $\pow (X)$ is a (classical) limit
  $\catc$-object by construction.  Conversely, let $L$ be a limit
  $\catc$-object, and consider two filters $\F$ and $\G$ of subsets of
  $\pt L$.  Then:
  \begin{eqnarray*}
    \lm_{\pow (\pt L)} (\F \cap \G)
    & = & (\lm_L \kow {(\F \cap \G)})^\bullet \\
    & = & (\lm_L (\kow\F \cap \kow\G))^\bullet \quad\text{by
          definition}\\
    & = & (\lm_L \kow\F \wedge \lm_L \kow\G)^\bullet \quad\text{since
          $\lm_L$ preserves binary infima} \\
    & = & (\lm_L \kow\F)^\bullet \cap (\lm_L \kow\G)^\bullet
          \quad\text{by Lemma~\ref{lemma:kow:basic}~(4)} \\
    & = & \lm_{\pow (\pt L)} \F \cap \lm_{\pow (\pt L)} \G.
  \end{eqnarray*}
\end{proof}

We recall that the traditional definitions of convergence and limit
spaces only consider convergence sets of \emph{proper} filters.  We
have chosen to also include the non-proper filter on $L$: this is the
top element of $\Filt L$, namely $L$ itself.  This raises the question
whether $\lm_L$ preserves top elements, namely whether
$\lm_L L = \top$.
%   Accordingly, we shall say that a convergence space
% $X$ is \emph{strict} if and only if $\lm_{\pow (X)} \pow (X) = X$, if
% and only if the non-proper filter $\pow (X)$ converges to every point.
A convergence $\catc$-object $L$ is \emph{strict} if and only if
$\lm_L L = \top$.  Those form the category $\sCv\catc$,
resp.\ $\sCvcl\catc$ for the classical variant.

There is no need to define strict convergence spaces: for every
convergence space $X$, $\lm_{\pow (X)}$ is strict, since for every $x
\in X$, $\lm_{\pow (X)} \pow (X) \supseteq \lm_{\pow (X)} \dot x
\supseteq \{x\}$, by the Point Axiom.
\begin{fact}
  \label{fact:A-|pt:strict}
  Let $\catc$ be an admissible category of lattices.  The adjunction
  $\pow\dashv\pt$ restricts to an adjunction between $\Conv$ and
  ${(\sCv\catc)}^{op}$, resp.\ between $\Conv$ and
  ${(\sCvcl\catc)}^{op}$.  \qed
\end{fact}
Finally, let %$\sLimc$ be the category of strict limit spaces, and
$\sLim\catc$ be the category of strict limit $\catc$-objects, i.e.,
convergence $\catc$-objects $L$ such that $\lm_L$ preserves finite
infima.  Let $\sLimcl\catc$ be the full subcategory of classical
strict limit $\catc$-objects.
\begin{fact}
  \label{fact:A-|pt:strict}
  Let $\catc$ be an admissible category of lattices.  The adjunction
  $\pow\dashv\pt$ restricts to an adjunction between $\Limc$ and
  ${(\sLim\catc)}^{op}$, resp.\ between $\Limc$ and
  ${(\sLimcl\catc)}^{op}$.  \qed
\end{fact}

All the categories considered in this section are complete and
cocomplete, assuming that $\catc$ is a category of coframes.  We show
this by appealing to the following result.  Let $|\_|$ be a
fiber-small topological functor from a category $\catd$ to a category
$\catc$.  A \emph{deflationary functor} on $\catd$ is a functor
$S^1 \colon \catd \to \catd$ such that $|S^1|$ is the identity functor
and such that $S^1 (C) \sqsubseteq C$ for every object $C$ of $\catd$
(i.e., the identity map on $|C|$ lifts to a morphism from $S^1 (C)$ to
$C$).  Let $S^\infty$ map each object $C$ of $\catd$ to the
$\sqsubseteq$-largest fixed point of $S^1$ below $C$ in the fiber of
$|C|$, and let $Fix (S^1)$ denote the full subcategory of $\catd$
whose objects are the fixed points of $S^1$.  Then $S^\infty$ extends
to a functor from $\catd$ to $\catc$ such that $|S^\infty|$ is the
identity functor, and $Fix (S^1)$ is a coreflective subcategory of
$\catd$, with coreflector $S^\infty$ \cite[Proposition~10]{GL-acs13}.
Moreover, the restriction of $|\_|$ to $Fix (S^1)$ is topological
again (Lemma~15, loc.\ cit.)

In our case, where $\catd = \Cv\catc$ or $\catd = \Cvcl\catc$, and
$|\_|$ maps every convergence $\catc$-object to the underlying
$\catc$-object, a deflationary functor $S^1$ on $\Cv\catc$ is given by
maps $S^1$ from convergence structures to convergence structures such
that $S^1 (\lm) \geq \lm$, and which act functorially, in the sense
that for every morphism $\varphi \colon L \to L'$ in $\Cv\catc$,
$\varphi$ should again be continuous from $S^1 (L) = (L, S^1 (\lm_L))$
to $S^1 (L') = (L', S^1 (\lm_{L'}))$.  We may define functors $S^1$ on
$\Cv\catc$, resp.\ $\Cvcl\catc$, by one of the following formulae:
\begin{eqnarray}
  \label{eq:S1:lim}
  S^1 (\lm_L)
  & \colon & \F \mapsto
             \bigvee_{\substack{\F_1, \F_2 \in \Filt L\\
  \F \supseteq \F_1 \cap \F_2}} \lm_L \F_1 \wedge \lm_L \F_2 \\
  \label{eq:S1:strict}
  S^1 (\lm_L)
  & \colon & \F \mapsto \left\{
             \begin{array}{ll}
               \lm_L \F & \text{if }\F\text{ is proper} \\
               \top & \text{otherwise}
             \end{array}
                      \right. \\
  \label{eq:S1:slim}
  S^1 (\lm_L)
  & \colon & \F \mapsto
             \bigvee_{\substack{  n \in \nat\\
  \F_1, \ldots, \F_n \in \Filt L \\
  \F \supseteq \bigcap_{i=1}^n \F_i}} \bigwedge_{i=1}^n \lm_L \F_i.
\end{eqnarray}
\begin{prop}
Each of the formulas \eqref{eq:S1:lim}, \eqref{eq:S1:strict}, and \eqref{eq:S1:slim} defines a deflationary functor, that is, $S^1(\lm_L
)\geq \lim_L$.
\end{prop}

\begin{proof} We check that $S^1$ defines a
functor in the case of (\ref{eq:S1:lim}), the other cases being
similar.  Let $\varphi$ be a morphism in $\Cv\catc$.  In view of
showing the continuity condition
$S^1 (\lm_{L'}) (\F) \leq \varphi (S^1 (\lm_L) (\varphi^{-1} (\F)))$,
we pick two arbitrary filters $\F_1$ and $\F_2$ such that
$\F \supseteq \F_1 \cap \F_2$.  Then
$\varphi^{-1} (\F) \supseteq \varphi^{-1} (\F_1) \cap \varphi^{-1}
(\F_2)$, and that implies that
$\lm_L \varphi^{-1} (\F_1) \wedge \lm_L \varphi^{-1} (\F_2) \leq S^1
(\lm_L) (\varphi^{-1} (\F))$ by definition of $S^1 (\lm_L)$.  We apply
$\varphi$ on both sides (recall that $\varphi$ commutes with finite
infima, and that
$\lm_{L'} \F_i \leq \varphi (\lm_L (\varphi^{-1} (\F_i)))$ for each
$i$ by the continuity of $\varphi$) so as to obtain
$\lm_{L'} \F_1 \wedge \lm_{L'} \F_2 \leq \varphi (S^1 (\lm_L)
(\varphi^{-1} (\F)))$.  We now take suprema over all $\F_1$, $\F_2$
such that $\F \supseteq \F_1 \cap \F_2$ and conclude.

It is also the case that $S^1 (\lm_L)$ is classical whenever $L$ is
classical.  We show this as follows, again in the case of
(\ref{eq:S1:lim}) only.  Given a filter $\F$ on $L$, and two filters
$\F_1$ and $\F_2$ such that $\F \supseteq \F_1 \cap \F_2$, we have
$(\F)_c \supseteq (\F_1)_c \cap (\F_2)_c$, so
$\lm_L \F_1 \wedge \lm_L \F_2 = \lm_L (\F_1)_c \wedge \lm_L (\F_2)_c
\leq S^1 (\lm_L) ((\F)_c)$ since $L$ is classical.  Taking suprema over $\F_1$, $\F_2$,
$S^1 (\lm_L) (\F) \leq S^1 (\lm_L) ((\F)_c)$, and we conclude by
Lemma~\ref{lemma:class}.
\end{proof}

The pointwise supremum of any family of classical convergence
structures is again classical.  It follows that $S^\infty (\lm_L)$,
which is obtained as a transfinite supremum of iterates of $S^1$, is
classical whenever $L$ is.

The convergence $\catc$-objects that are fixed points of $S^1$ are
exactly the limit $\catc$-objects, the strict convergence
$\catc$-objects, and the strict limit $\catc$-objects, respectively
(and similarly for the classical variants, as we have just argued).
This allows us to conclude:
\begin{prop}
  \label{prop:cocompl:lim}
  Let $\catc$ be a category of lattices.  The functor $|\_|$ restricts
  to a topological functor from $\Lim\catc$, resp.\ $\sCv\catc$,
  resp.\ $\sLim\catc$ to $\catc$.  Those categories are reflective
  subcategories of $\Cv\catc$.

  The functor $|\_|$ restricts to a topological functor from
  $\Limcl\catc$, resp.\ $\sCvcl\catc$, resp.\ $\sLimcl\catc$ to
  $\catc$.  Those categories are reflective subcategories of
  $\Cvcl\catc$.

  $\Lim\coFrm$, $\sCv\coFrm$, $\sLim\coFrm$, and their classical
  variants $\Limcl\coFrm$, $\sCvcl\coFrm$ and $\sLimcl\coFrm$ are
  complete and cocomplete.  \qed
\end{prop}

There is a well-known adjunction between $\Conv$ and $\Limc$.  The
\emph{finitely deep modification} $M_{\mathrm{lim}} X$ of a
convergence space $X$ is $X$ together with the finest finitely deep
convergence $\lm$ coarser than $\lm_{\pow (X)}$
\cite[Definition~III.5.3]{DM:convergence}; namely
$\lm = S^\infty (\lm_{\pow (X)})$ where $S^\infty$ is the least fixed
point functor derived from the $S^1$ functor of (\ref{eq:S1:lim}).
This is left-adjoint to the forgetful functor from $\Limc$ to $\Conv$,
and by definition the left-adjoints commute in the following diagram:
\[
  \xymatrix{
    {(\Cv\catc)}^{op}
    \ar@<1ex>[r]^{S^\infty} \ar@{}[r]|{\perp}
    \ar@<1ex>[d]^{\pt}\ar@{}[d]|{\dashv}
    &
    {(\Lim\catc)}^{op}
    \ar@<1ex>[l]^{\supseteq}
    \ar@<1ex>[d]^{\pt}\ar@{}[d]|{\dashv}
    \\
    \Conv
    \ar@<1ex>[u]^{\pow}
    \ar@<1ex>[r]^{M_{\mathrm{lim}}} \ar@{}[r]|{\perp}
    &
    \Limc
    \ar@<1ex>[l]^{\supseteq}
    \ar@<1ex>[u]^{\pow}
  }
\]
It follows that the right-adjoints commute, too, since they are
uniquely determined by the left-adjoints.  A similar diagram is
obtained with $\sLim\catc$ in place of $\Lim\catc$, or with
$\Cvcl\catc$ instead of $\Cv\catc$ and $\Limcl\catc$ (or
$\sLimcl\catc$) instead of $\Lim\catc$.

The latter yields the leftmost square of (\ref{eq:adj}) in the case
$\catc = \coFrm$.

\section{Pretopological Coframes}
\label{sec:pret-cofr}

A \emph{pretopological space} is a convergence space $X$ in which
$\lm_{\pow (X)}$ commutes with arbitrary intersections
\cite[Proposition~V.1.4]{DM:convergence}.  Those form a full
subcategory $\PreTop$ of $\Conv$.

Similarly:
\begin{defn}[Pretopological]
  \label{defn:pretop}
  Given a category $\catc$ of coframes, a $\catc$-object $L$ is
  \emph{pretopological} if and only if
  $\lm_L {\bigcap_{i \in I} \F_i} = \bigwedge_{i \in I} \lm_L \F_i$,
  for every family ${(\F_i)}_{i \in I}$ of filters on $L$.  They
  define a full subcategory $\preTop\catc$ of $\Cv\catc$, and a
  further full subcategory $\preTopcl\catc$ of classical pretopological
  $\catc$-objects.
\end{defn}

\begin{lem}
  \label{lemma:kow:basic:cof}
  Let $\catc$ be a category of coframes, and $L$ be a convergence
  $\catc$-object.  The following hold:
  \begin{enumerate}
  \item For every family ${(\ell_i)}_{i \in I}$ of elements of $L$,
    $\bigcap_{i \in I} \ell_i^\bullet = {(\bigwedge_{i \in I}
      \ell_i)}^\bullet$.
  \end{enumerate}
\end{lem}
\begin{proof}
  Recall that $\epsilon_L$ defined by $\epsilon_L(\ell)=\ell^\bullet$ is a $\Cv\catc$ morphism by Lemma~\ref{lemma:epsilon}, hence preserves
  arbitrary infima.
%
  % For every point $\varphi$, $\varphi$ is in
  % ${(\bigwedge_{i \in I} \ell_i)}^\bullet$ if and only if
  % $\varphi (\bigwedge_{i \in I} \ell_i)=1$, if and only if
  % $\bigwedge_{i \in I} \varphi (\ell_i)=1$ (since $\varphi$ preserves
  % arbitrary infima, being a morphism of coframes), if and only if
  % $\varphi (\ell_i)=1$ for every $i \in I$, if and only if $\varphi$
  % is in every $\ell_i^\bullet$.
\end{proof}

\begin{prop}
  \label{prop:A-|pt:pretop}
  Let $\catc$ be an admissible category of coframes.  The coreflection
  $\pow\dashv\pt$ restricts to a coreflection of $\PreTop$ into
  ${(\preTop\catc)}^{op}$, resp.\ into ${(\preTopcl\catc)}^{op}$.
\end{prop}
\begin{proof}
  Given a pretopological space $X$, $\pow (X)$ is a (classical)
  pretopological $\catc$-object by construction.  Conversely, let $L$
  be a pretopological $\catc$-object, and consider any family of
  filters $\F_i$ of subsets of $\pt L$, $i \in I$.  Then:
  \begin{eqnarray*}
    \lm_{\pow (\pt L)} (\bigcap_{i \in I} \F_i)
    & = & (\lm_L \kow {(\bigcap_{i \in I} \F_i)})^\bullet \\
    & = & (\lm_L (\bigcap_{i \in I} \kow{\F_i}))^\bullet \quad\text{by
          definition}\\
    & = & (\bigwedge_{i \in I} \lm_L \kow{\F_i})^\bullet \quad\text{since
          $\lm_L$ preserves arbitrary infima} \\
    & = & \bigcap_{i \in I} (\lm_L \kow{\F_i})^\bullet
          \quad\text{by Lemma~\ref{lemma:kow:basic:cof}} \\
    & = & \bigcap_{i \in I} \lm_{\pow (\pt L)} \F_i.
  \end{eqnarray*}
\end{proof}

Using the technology of topological functors developed in
Section~\ref{sec:limit-spaces}, define:
\begin{eqnarray}
  \label{eq:S1:pretop}
  S^1 (\lm_L)
  & \colon & \F \mapsto
             \bigvee_{\substack{\Phi \in \pow (\Filt L) \\
  \F \supseteq \bigcap\Phi}} \bigwedge_{\G \in \Phi} \lm_L \G.
\end{eqnarray}
\begin{lem}
 $S^1$ given by \eqref{eq:S1:pretop} is a functor. Moreover, $S^1$, hence $S^\infty$, preserves classicality.
 \end{lem}
 \begin{proof}
 We proceed as we did for limit spaces and
(\ref{eq:S1:lim}).  Let $\varphi$ be a morphism in $\Cv\catc$, and
assume that $\catc$ is a category of coframes.  In view of showing the
continuity condition
$S^1 (\lm_{L'}) (\F) \leq \varphi (S^1 (\lm_L) (\varphi^{-1} (\F)))$,
we pick a family $\Phi$ of filters on $L$ such that
$\F \supseteq \bigcap \Phi$.  Then
$\varphi^{-1} (\F) \supseteq \bigcap_{\G \in \Phi} \varphi^{-1} (\G)$,
and that implies that
$\bigwedge_{\G \in \Phi} \lm_L \varphi^{-1} (\G) \leq S^1 (\lm_L)
(\varphi^{-1} (\F))$ by definition of $S^1 (\lm_L)$.  We apply
$\varphi$ on both sides (recall that $\varphi$ commutes with arbitrary
infima, and that
$\lm_{L'} \G \leq \varphi (\lm_L (\varphi^{-1} (\G)))$ for each
$\G \in \Phi$ by the continuity of $\varphi$) so as to obtain
$\bigwedge_{\G \in \Phi} \lm_{L'} \G \leq \varphi (S^1 (\lm_L)
(\varphi^{-1} (\F)))$.  We now take suprema over all $\Phi$ such that
$\F \supseteq \bigcap \Phi$ and conclude.

We also prove that $S^1$ preserves classicality
similarly as for (\ref{eq:S1:lim}).
\end{proof}

The pretopological convergence $\catc$-objects are exactly the fixed
points of $S^1$, whence:
\begin{prop}
  \label{prop:cocompl:pretop}
  Let $\catc$ be a category of coframes.  The functor $|\_|$ restricts
  to a topological functor from $\preTop\catc$ (resp.,
  $\preTopcl\catc$) to $\catc$.  $\preTop\catc$ and $\preTopcl\catc$
  are reflective subcategories of $\Cv\catc$.  $\preTop\coFrm$ and
  $\preTopcl\catc$ are complete and cocomplete.  \qed
\end{prop}
There is a well-known adjunction between $\Limc$ (or $\Conv$) and
$\PreTop$.  The \emph{pretopological modification}
$M_{\mathrm{pretop}} X$ of a convergence space $X$ is $X$ together
with the finest pretopological convergence $\lm$ coarser than
$\lm_{\pow (X)}$ \cite[Proposition~V.1.6]{DM:convergence}.  This is
left-adjoint to the forgetful functor from $\PreTop$ to $\Limc$, and by
definition the left-adjoints commute in the following diagram:
\[
  \xymatrix{
    {(\Lim\catc)}^{op}
    \ar@<1ex>[r]^{S^\infty} \ar@{}[r]|{\perp}
    \ar@<1ex>[d]^{\pt}\ar@{}[d]|{\dashv}
    &
    {(\preTop\catc)}^{op}
    \ar@<1ex>[l]^{\supseteq}
    \ar@<1ex>[d]^{\pt}\ar@{}[d]|{\dashv}
    \\
    \Limc
    \ar@<1ex>[u]^{\pow}
    \ar@<1ex>[r]^{M_{\mathrm{lim}}} \ar@{}[r]|{\perp}
    &
    \PreTop
    \ar@<1ex>[l]^{\supseteq}
    \ar@<1ex>[u]^{\pow}
  }
\]
where $\catc$ is an admissible category of coframes.  The
right-adjoints commute, too, since they are uniquely determined by the
left-adjoints.  A similar diagram is obtained with $\sLim\catc$ in
place of $\Lim\catc$, or with $\preTopcl\catc$ instead of
$\preTop\catc$ and $\Limcl\catc$ (or $\sLimcl\catc$) instead of
$\Lim\catc$.

The latter yields the second square of (\ref{eq:adj}) from the left,
in the case $\catc = \coFrm$.

In the realm of convergence spaces, we know that it is equivalent to
assume a convergence space to be pretopological, or to assume that is
arises from a non-idempotent closure operator, i.e., an inflationary
operator $\nu$ that preserves finite suprema (the \emph{adherence}
operator).  The situation is more complicated in the pointfree world.
We shall deal with it in Section~\ref{sec:adherence-coframes}.
Classicality will play an important role there.  More importantly, we
shall need to define suitable notions of closed elements and adherence
operators, a task that will occupy Section~\ref{sec:sierp-conv-cofr}.

\section{The Sierpi\'nski Convergence Coframe, Open and Closed
  Elements, and Adherence}
\label{sec:sierp-conv-cofr}

Every topological space $X$ gives rise to a convergence space by
defining $\F \to x$ if and only if every open neighborhood of $x$
belongs to $\F$.

The Sierpi\'nski space $\Sierp$ is the topological space with two
elements $0$ and $1$, with open sets $\emptyset$, $\{1\}$, and the whole
space.  We can therefore also view $\Sierp$ as a convergence space,
where convergence is given by $\F \to x$ if and only if $x=0$, or
$x=1$ and $\{1\} \in \F$.
% only $\F = \{1\}$ has to be considered
% hence $\F \to x$ iff $\{1\}$ contains $x$ implies $\{1\} \in \F$.
Explicitly: there are four filters of subsets of $\Sierp$, and they
are $\upc \{0, 1\}$, $\upc \{0\}$, $\upc \{1\}$, and $\upc \emptyset$.
The first two converge to $0$, the last two converge to both $0$ and
$1$.

Using the $\pow$ functor, we obtain the \emph{Sierpi\'nski convergence
  coframe} $\pow (\Sierp)$, with
$\lm_{\pow (\Sierp)} \upc \{0, 1\} = \lm_{\pow (\Sierp)} \upc \{0\}
= \{0\}$, and
$\lm_{\pow (\Sierp)} \upc \{1\} = \lm_{\pow (\Sierp)} \upc \emptyset
= \{0, 1\}$.  Said more briefly, $\lm_{\pow (\Sierp)} \F$ equals
$\{0, 1\}$ if $\F$ contains $\{1\}$, and $\{0\}$ otherwise.

% In the following, we shall make use of the following assumption.
% \begin{assum}
%   \label{assum:sierp}
%   $\pow (\Sierp)$ is an object of $\catc$, and for every object $L$ of
%   $\catc$, every lattice morphism from $\pow (\Sierp)$ to $L$ is a
%   morphism in $\catc$.
% \end{assum}
% \begin{rem}
%   \label{rem:open:assum}
%   Assumption~\ref{assum:sierp} is satisfied when $\catc$ is a full
%   subcategory of the category of lattices, but also when it is a full
%   subcategory of the category of frames, of coframes, or of complete
%   lattices.  Indeed, due to the low cardinality of $\pow (\Sierp)$,
%   any lattice morphism from $\pow (\Sierp)$ to $L$ automatically
%   preserves all infima and all suprema, which are all finite.
% \end{rem}

One way of defining an open subset $U$ of a topological space $X$ is
by giving its indicator function $\chi_U \colon X \to \Sierp$.  That
is a continuous function, and all continuous functions from $X$ to
$\Sierp$ are of this form.  This observation naturally leads us to
study the morphisms $\varphi \colon L \to \pow (\Sierp)$ in
$\Cv\catc$, that is, the morphisms of convergence $\catc$-objects
$\varphi \colon \pow (\Sierp) \to L$, and to attempt to retrieve a
suitable notion of open element of $L$.
\begin{lem}
  \label{lemma:openclosed}
  Let $\catc$ be a category of lattices, and $L$ be a $\catc$-object.
  \begin{enumerate}
  \item For every morphism $\varphi \colon \pow (\Sierp) \to L$ in
    $\Cv\catc$, the elements $u := \varphi (\{1\})$ and
    $c := \varphi (\{0\})$ satisfy:
    \begin{enumerate}
    \item[$(i)$] $u \wedge c = \bot$,
    \item[$(ii)$] $u \vee c = \top$, and
    \item[$(iii)$] for every filter $\F$ on $L$, $u \in \F$ or
      $\lm_L \F \leq c$.
    \end{enumerate}
  \item Conversely, if $\catc$ is admissible, then for any pair of
    elements $u$, $c$ of $L$ satisfying $(i)$--$(iii)$, the map
    $\varphi \colon \pow (\Sierp) \to L$ that sends $\emptyset$ to
    $\bot$, $\{0, 1\}$ to $\top$, $\{0\}$ to $c$ and $\{1\}$ to $u$,
    is a morphism in $\Cv\catc$.
  \end{enumerate}
\end{lem}
\begin{proof}
  (1) We have
  $u \wedge c = \varphi (\{1\}) \wedge \varphi (\{0\}) = \varphi
  (\{1\} \cap \{0\}) = \varphi (\emptyset) = \bot$, proving $(i)$.
  Claim $(ii)$ is proved similarly.  By the continuity condition
  (\ref{eq:morphism}), for every filter $\F$ on $L$,
  $\lm_L \F \leq \varphi (\lm_{\pow (\Sierp)} \varphi^{-1} (\F))$.  If
  $\varphi^{-1} (\F)$ contains $\{1\}$, then $u := \varphi (\{1\})$ is
  in $\F$.  Otherwise,
  $\lm_{\pow (\Sierp)} \varphi^{-1} (\F) = \{0\}$, and then
  $\lm_L \F \leq \varphi (\{0\}) = c$.

  (2) $\varphi$ preserves top and bottom by definition, and binary
  infima and binary suprema by $(i)$ and $(ii)$.  Due to the low
  cardinality of $\pow (\Sierp)$, this is enough to guarantee that
  $\varphi$ preserves all infima and all suprema, hence is a morphism
  of $\catc$ by admissibility.

  We need to check continuity.  Let $\F$ be a filter on $L$.  If
  $\varphi^{-1} (\F)$ contains $\{1\}$, then
  $\lm_L \F \leq \varphi (\lm_{\pow (\Sierp)} \varphi^{-1} (\F))$
  holds trivially, since the right-hand side is
  $\varphi (\{0, 1\}) = \top$.  Otherwise,
  $\varphi (\lm_{\pow (\Sierp)} \varphi^{-1} (\F)) = \varphi (\{0\}) =
  c$.  By $(iii)$ we must have $\lm_L \F \leq c$, that is,
  $\lm_L \F \leq \varphi (\lm_{\pow (\Sierp)} \varphi^{-1} (\F))$.
\end{proof}
In a pair $(u, c)$ satisfying $(i)$--$(iii)$, $u$ is \emph{open} and
$c$ is \emph{closed}.  Conditions $(i)$ and $(ii)$ express that $u$
and $c$ are complements.  In a distributive lattice, complements are
determined uniquely, so one may try to define open elements, and
closed elements, without referring to their complement.  This can be
done by using the notion of \emph{grill}, imitated from the eponymous
notion introduced by Choquet on convergence spaces
\cite{choquetgrille} and widely used since then (e.g.,
\cite{DM:convergence}).

\begin{defn}
  \label{defn:mesh}
  Let $L$ be a lattice and $\A$, $\B$ two subsets of $L$.

  We say that $\A$ and $\B$ \emph{mesh}, or that $\A$ \emph{meshes
    with} $\B$, in notation $\A\mathrel{\#}\B$, if and only if for
  every $a \in \A$, for every $b \in \B$, $a \wedge b \neq \bot$.

  The \emph{grill} of $\A$, $\A^\#$ is the set of elements $\ell \in
  L$ such that, for every $a \in A$, $a \wedge \ell \neq \bot$.
\end{defn}

\begin{lem}
  \label{lemma:grill}
  Let $L$ be a lattice.
  \begin{enumerate}
  \item For every subset $\A$, $\A^\#$ is upwards-closed.
  \item If $L$ is distributive, then for every filter $\F$, $\F^\#$ is
    prime, in the sense that for every finite family of elements
    $\ell_1$, \ldots, $\ell_n$ of $L$ such that
    $\bigvee_{i=1}^n \ell_i \in \F^\#$, some $\ell_i$ is already in
    $\F^\#$.
  \item For any two subsets $\A$ and $\B$ of $L$, if $\A \subseteq \B$
    then $\B^\# \subseteq \A^\#$.
  \item For any two subsets $\A$ and $\B$ of $L$, $\A\mathrel{\#}\B$
    if and only if $\A \subseteq \B^\#$.
  \item For every filter $\G$ on $L$, $\G$ is proper if and only if
    $\G \subseteq \G^\#$.
%  \item For every filter $\F$ on $L$, $\F = \F^{\#\#}$.
  \end{enumerate}
\end{lem}
\begin{proof}
  (1), (3) and (4) are obvious.  (2) Let
  $\bigvee_{i=1}^n \ell_i \in \F^\#$, and assume for the sake of
  contradiction that no $\ell_i$ is in $\F^\#$.  For every $i$,
  there is an element $m_i$ in $\F$ such that
  $\ell_i \wedge m_i = \bot$.  Let $m = \bigwedge_{i=1}^n m_i$.  Then
  $m$ is in $\F$, since $\F$ is a filter, and $\ell_i \wedge m = \bot$
  for every $i$ since $\ell_i \wedge m \leq \ell_i \wedge m_i = \bot$.
  Taking suprema over all $i$, and using distributivity,
  $\bigvee_{i=1}^n \ell_i \wedge m = \bot$, contradicting
  $\bigvee_{i=1}^n \ell_i \in \F^\#$.

  (5) If $\G$ is proper, then $\G$
  meshes with itself: for all $\ell_1 \in \G$ and $\ell_2 \in \G$,
  $\ell_1 \wedge \ell_2$ is in $\G$, since $\G$ is a filter, and is
  different from $\bot$ since $\G$ is proper.  It follows that
  $\G \subseteq \G^\#$, using (4).  Conversely, if
  $\G \subseteq \G^\#$, then $\G$ meshes with itself by (4).  In
  particular, for every $\ell \in \G$, $\ell = \ell \wedge \ell$ is
  different from $\bot$, so $\G$ is proper.
\end{proof}
The $\#$ relation is better understood on filters.  The ideal
completion of a sup-semilattice is a complete lattice.  It follows
that, given an inf-semilattice $L$, $\Filt L = \Idl (L^{op})$ (see
Remark~\ref{rem:ideal}) is a complete lattice as well. Infima are
given by intersections.  For suprema, we have the following.  The
supremum of the empty family is the smallest filter, $\{\top\}$.
Directed suprema are unions. For binary suprema, $\F \vee \G$ is the
upward closure of the set of elements $m \wedge \ell$, $m \in \F$,
$\ell \in \G$.  It follows that:
\begin{fact}
  \label{fact:mesh}
  Two filters $\F$ and $\G$ on a lattice $L$ mesh if and only if
  $\F \vee \G$ is proper.
\end{fact}

A \emph{pseudocomplement} of $\ell \in L$ is an element $\ell^*$ such
that $m \leq \ell^*$ if and only if $m \wedge \ell = \bot$.  If $\ell$
has a complement $\overline\ell$ in a distributive lattice $L$, then
$\overline\ell$ is a pseudocomplement of $\ell$.  In a frame, every
element has a unique pseudocomplement.  In a coframe $L$, every
element has a unique pseudocomplement in $L^{op}$; if it is also a
pseudocomplement (in $L$), then it is a complement.
\begin{lem}
  \label{lemma:grill:compl}
  Let $L$ be a lattice, and $\F$ be an up-set on $L$.  For every
  element $\ell$ with a pseudocomplement $\ell^*$, the following are
  equivalent:
  \begin{enumerate}
  \item $\ell \in \F^\#$;
  \item $\ell^* \not\in \F$.
  \end{enumerate}
\end{lem}
\begin{proof}
  We show instead that $\ell \not\in \F^\#$ is equivalent to $\ell^*
  \in \F$.

  If $\ell^* \in \F$, then $\ell \wedge \ell^* = \bot$ prevents $\ell$
  from being in $\F^\#$.  Conversely, if $\ell$ is not in
  $\F^\#$, then there is an $m \in \F$ such that
  $\ell \wedge m = \bot$.  Equivalently, $m \leq \ell^*$, which
  implies that $\ell^*$ is in $\F$ because $\F$ is an up-set.
\end{proof}
%%%%%%%%%%%ADDED
\begin{cor}\label{cor:imagepseudocompl}
Let $\varphi:L\to L'$  be a lattice morphism where $L'$ is a distributive lattice, and let $a\in L$ be a complemented element. Then  $\varphi(a)$ is complemented and $\varphi(a)^*=\varphi(a^*)$.  Moreover, if $\F$ is an up-set of $L'$ then
\[a\in\varphi^{-1}(\F)^\#\iff \varphi(a)\in\F^\#.\] 

\end{cor}
\begin{proof}
Since $\varphi(a\wedge a^*)=\varphi(a)\wedge \varphi(a^*)$ and $\varphi(a\wedge a^*)=\varphi(\bot_L)=\bot_{L'}$, we have $\varphi(a)\wedge \varphi(a^*)=\bot_{L'}$. Similarly, $\varphi(a\vee a^*)=\varphi(a)\vee \varphi(a^*)$ and $\varphi(a\vee a^*)=\varphi(\top_L)=\top_{L'}$. Thus $\varphi(a^*)$ is the complement of $\varphi(a)$ in $L'$ (which is unique because $L'$ is distributive).

In view of Lemma \ref{lemma:grill:compl}, $\varphi(a)\in \F^\#$ if and only if $\varphi(a)^*\notin \F$, that is, $\varphi(a^*)\notin \F$, equivalently, $a^*\notin \varphi^{-1}(\F)$. By Lemma \ref{lemma:grill:compl}, this is equivalent to $a\in\varphi^{-1}(\F)^\#$ because $\varphi^{-1}(\F)$ is an up-set of $L$ if $\F$ is an up-set of $L'$. 
\end{proof}
%%%%%%%%%%%
\begin{cor}
  \label{corl:grill:compl:=}
  Let $L$ be a lattice, and $\F$, $\G$ be two filters on $L$.  
  \[\F\cap\C_L\subset\G\cap \C_L\then \G^\#\cap \C_L\subset \F^\#\cap\C_L.
  \]  In particular, if
  $\F\cap\C_L=\G\cap\C_L$  then $\F^\#\cap\C_L=\G^\#\cap \C_L$.
\end{cor}
\begin{proof}
  For every complemented element $a$ of $\G^\#$, $\overline a$ is not
  in $\G$ by Lemma~\ref{lemma:grill:compl}.  Then $\overline a$ cannot
  be in $\F$, since every complemented element of $\F$ is in $\G$.  By
  Lemma~\ref{lemma:grill:compl} again, it follows that $a$ is in
  $\F^\#$.
\end{proof}

Following Lemma~\ref{lemma:openclosed}, one is tempted to define
closed elements as those complemented elements $c$ such that every
filter $\F$ on $L$ that does not contain the complement of $c$
satisfies $\lm_L \F \leq c$.  In light of
Lemma~\ref{lemma:grill:compl}, we define:
\begin{defn}[Closed]
  \label{defn:closed}
  Let $\catc$ be a category of lattices, and $L$ be a convergence
  $\catc$-object.  An element $c$ of $L$ is \emph{quasi-closed} if and
  only if:
  \begin{quote}
    for every filter $\F$ on $L$ such that $c \in \F^\#$, $\lm_L
    \F \leq c$.
  \end{quote}
  A \emph{closed element} of $L$ is a complemented quasi-closed
  element of $L$.
\end{defn}

Summing up:
\begin{fact}
  \label{fact:closed=phi}
  Let $\catc$ be an admissible category of distributive lattices.  The
  morphisms $\varphi \colon \pow (\Sierp) \to L$ in $\Cv\catc$ are in
  one-to-one correspondence with the closed elements of $L$.
\end{fact}

\begin{rem}
  \label{rem:closed}
  Since every element is complemented in $\pow (X)$, its quasi-closed
  elements and its closed elements coincide.
  % As a consequence of the
  % forthcoming Proposition~\ref{prop:closed:subcoframe}, they will also
  % coincide with the closed elements.
  Explicitly, the closed elements are the subsets $C$ that contain all
  the limits of filters $\F$ such that $C \in \F^\#$.  This is the
  standard definition of a closed subset of a convergence space.
\end{rem}

Closed sets can instead be obtained as (pre-)fixed points of a
so-called adherence operator.  The following definition is the obvious
choice, but a better one will be given in Definition~\ref{defn:adh},
which takes complemented elements into account more finely.
\begin{defn}[Raw adherence]
  \label{defn:adh:raw}
  Let $\catc$ be a category of lattices, and $L$ be a convergence
  $\catc$-object.  For every $\ell \in L$, the \emph{raw adherence} of
  $\ell$ is defined by:
  \begin{equation}
    \label{eq:adh:raw}
    \adhr_L \ell := \bigvee_{\substack{\F \in \Filt L\\ \ell \in \F^\#}} \lm_L \F.
  \end{equation}
\end{defn}

\begin{fact}
  \label{fact:adhr:qclosed}
  The quasi-closed elements of $L$ are the pre-fixpoints of
  $\adhr_L \colon L \to L$, that is, the elements $\ell \in L$ such
  that $\adhr_L \ell \leq \ell$.
\end{fact}

\begin{prop}
  \label{prop:adh:raw}
  Let $\catc$ be a category of lattices, and $L$ be a convergence
  $\catc$-object.  The operator $\adhr_L \colon L \to L$ preserves finite
  suprema, and is in particular monotonic.
\end{prop}
\begin{proof}
  Let $\ell_1$, \ldots, $\ell_n$ be finitely many elements of $L$.
  For every filter $\F$ on $L$, by Lemma~\ref{lemma:grill}~(1)
  and~(2), $\bigvee_{i=1}^n \ell_i \in \F^\#$ if and only if
  $\ell_i \in \F^\#$ for some $i$, $1\leq i\leq n$.  Hence:
  \begin{eqnarray*}
    \adhr_L (\bigvee_{i=1}^n \ell_i)
    & = & \bigvee_{\F \in \Filt L, \bigvee_{i=1}^n \ell_i \in
          \F^\#} \lm_L \F \\
    & = & \bigvee_{\F \in \Filt L, \ell_i \in
          \F^\#\text{ for some }i} \lm_L \F \\
    & = & \bigvee_{i=1}^n \bigvee_{\F \in \Filt L, \ell_i \in
          \F^\#} \lm_L \F \\
    & = & \bigvee_{i=1}^n \adhr_L {\ell_i}.
  \end{eqnarray*}
  % Monotonicity follows, since every binary-supremum-preserving
  % operator is monotonic.
\end{proof}

\begin{rem}
  \label{rem:adh:center}
  The adherence operator $\adhr_L$ is not idempotent in general, even
  on convergence $\catc$-objects of the form $\pow (X)$, see
  \cite[Example~V.4.6]{DM:convergence}.  More surprising, perhaps, is
  the failure of the law $\adhr_L \ell \geq \ell$, since
  Lemma~\ref{lemma:pow:centered} below, for one, will state that this
  law holds on every convergence $\catc$-object of the form
  $\pow (X)$.  However, it certainly fails on the \emph{bottom}
  convergence $\catc$-objects $L$, namely those such that
  $\lm_L \F = \bot$ for every $\F \in \mathbb F L$.
\end{rem}

\begin{defn}[Centered]
  \label{defn:centered}
  Let $\catc$ be a category of lattices.  The convergence
  $\catc$-object $L$ is \emph{centered} if and only if, for every
  $\ell \in L$, $\ell \leq \adh_L \ell$.
\end{defn}
\begin{rem}
A \emph{preconvergence space} is defined like a convergence space, but without requiring the Point Axiom \cite{DM:convergence}. It is easily seen that a preconvergence space $X$ is a convergence space if and only if its adherence operator $\adh:\mathbb P X\to \mathbb P X$ defined by $\adh A=\bigcup_{A\in\F\in\mathbb{FP}X}\lim\F$ satisfies $A\subset \adh A$ for all $A\in\mathbb P X$. Hence the condition of being centered is a pointfree analog of the Point Axiom!
\end{rem}
\begin{lem}
  \label{lemma:pow:centered}
  Every convergence $\catc$-object of the form $\pow (X)$, where $X$
  is a convergence space, is centered.
\end{lem}
\begin{proof}
  Let $\ell \in \pow (X)$.  For every $x \in \ell$, $x$ is in
  $\lm_{\pow (X)} \F$ where $\F = \dot x$.  Clearly $\ell \in \F^\#$,
  so $x$ is in $\adh_{\pow (X)} \ell$.
\end{proof}

\begin{fact}
  \label{fact:adhr:qclosed:centered}
  The quasi-closed elements of a centered convergence $\catc$-object
  $L$ are the fixed points of $\adhr_L \colon L \to L$.  \qed
\end{fact}

Every morphism of coframes $\varphi \colon L \to L'$ is an
(order-theoretic) right adjoint, since it preserves all infima.  Write
its left adjoint as $\varphi_!$.  When $\varphi$ is the inverse image
map $f^{-1}$, where $f \colon X \to Y$, $\varphi_!$ is the
corresponding direct image map. 
\begin{lem}\label{lem:adjointofmapgrill}
If $\varphi:L\to L'$ is a morphism of coframes, $\varphi_!:L'\to L$ denotes its left adjoint, and $\F\in\mathbb F L'$, then 
\[\varphi_{!}\left(\F^\#\right)\subset \left(\varphi^{-1}(\F)\right)^\#.
\]
\end{lem}
\begin{proof}
Let $\ell' \in {\F}^\#$.  
  Since $\varphi_!$ is left adjoint to $\varphi$,
  $\ell' \leq \varphi (\varphi_!  (\ell'))$.  Then, for every
  $\ell \in \varphi^{-1} (\F)$,
  $\varphi (\varphi_! (\ell') \wedge \ell) = \varphi (\varphi_!
  (\ell')) \wedge \varphi (\ell) \geq \ell' \wedge \varphi (\ell)$.
  Since $\ell'$ is in ${\F}^\#$ and $\varphi (\ell)$ is in $\F$,
  $\ell' \wedge \varphi (\ell)$ is different from $\bot$, which
  implies $\varphi (\varphi_! (\ell') \wedge \ell) \neq \bot$ (for otherwise $\bot=\varphi (\varphi_! (\ell') \wedge \ell)=\varphi (\varphi_! (\ell')) \wedge \varphi(\ell)\geq \ell' \wedge \varphi (\ell)$ because $\varphi$ is a coframe morphism). Thus $\varphi_! (\ell')$  and we
  conclude that $\varphi_! (\ell')$ is in $\left(\varphi^{-1} (\F)\right)^\#$.
\end{proof}
 The following is then the pointfree
analogue of the standard fact that the images of adherences are
contained in the adherences of the images, in convergence spaces.
\begin{prop}
  \label{prop:adh:raw:cont}
  Let $\catc$ be a category of coframes, and let $\varphi \colon L \to
  L'$ be a morphism in $\Cv\catc$.  Then, for every $\ell' \in L'$,
  \begin{eqnarray}
    \label{eqn:adh:raw:image}
    \varphi_! (\adhr_{L'} (\ell')) \leq \adhr_{L} (\varphi_! (\ell')),
    \\
    \label{eqn:adh:raw:img:inv}
    \adhr_{L'} (\ell') \leq \varphi (\adhr_L (\varphi_! (\ell'))).
  \end{eqnarray}
\end{prop}
\begin{proof}
  It follows from Lemma \ref{lem:adjointofmapgrill} that, if $\ell' \in {\F'}^\#$, then
  $\lm_L \varphi^{-1} (\F') \leq \adhr_L (\varphi_! (\ell'))$.  Indeed,
  the left-hand side appears as one of the disjuncts in the definition
  of the right-hand side.  Since $\varphi$ is continuous,
  $\lm_{L'} \F' \leq \varphi (\lm_L \varphi^{-1} (\F'))$, so
  $\lm_{L'} \F' \leq \varphi (\adhr_L (\varphi_! (\ell')))$.  Taking
  suprema over all filters $\F'$ such that $\ell' \in {\F'}^\#$, we
  obtain
  $\adhr_{L'} (\ell') \leq \varphi (\adhr_L (\varphi_! (\ell')))$.  That
  is (\ref{eqn:adh:raw:img:inv}). The inequality (\ref{eqn:adh:raw:image}) is
  equivalent, by the definition of left adjoints.
\end{proof}

There is another natural notion of adherence on convergence coframes.
While it is less natural at first sight, this is the one we shall need
in the sequel.
\begin{defn}[Adherence]
  \label{defn:adh}
  Let $\catc$ be a category of coframes, and $L$ be a convergence
  $\catc$-object.  For every $\ell \in L$, the \emph{adherence} of
  $\ell$ is:
  \begin{equation}
    \label{eq:adh}
    \adh_{\lm_L}\ell := \adh_L \ell := \bigwedge_{a\in\C_L, a \geq \ell}
    \adhr_L a.
  \end{equation}
\end{defn}
\begin{prop}
  \label{prop:adh}
  Let $\catc$ be a category of coframes, and $L$ be a convergence
  $\catc$-object.  The operator $\adh_L \colon L \to L$:
  \begin{enumerate}
  \item satisfies $\adh_L \ell \geq \adhr_L \ell$ for every
    $\ell \in L$,
  \item coincides with $\adhr_L$ on complemented elements,
  \item is monotonic,
  \item preserves finite suprema of complemented elements,
  \item and satisfies
    $\adh_L \ell = \bigwedge_{a\in\C_L, a\geq \ell} \adh_L
    a$ for every $\ell \in L$.
  \end{enumerate}
\end{prop}
\begin{proof}
  (1) For every complemented $a \geq \ell$,
  $\adhr_L a \geq \adhr_L \ell$ by monotonicity.  Taking infima over
  $a$ yields the result.  (2) If $\ell$ is complemented, then one can
  take $a=\ell$ in (\ref{eq:adh}), so that
  $\adh_L \ell \leq \adhr_L \ell$.  Equality follows from (1).  (3) is
  obvious. (4) On complemented elements, $\adh_L$ and $\adhr_L$
  coincide by (2), so the conclusion follows from
  Proposition~\ref{prop:adh:raw}.  (5) By (\ref{eq:adh}) and (2).
\end{proof}

Proposition~\ref{prop:adh} together with Fact~\ref{fact:adhr:qclosed}
together imply the following.
\begin{fact}
  \label{fact:adh:closed}
  The closed elements of $L$ are the complemented pre-fixpoints of
  $\adh_L \colon L \to L$, that is, the complemented elements
  $\ell \in L$ such that $\adh_L \ell \leq \ell$.
\end{fact}

\begin{prop}
  \label{prop:adh:cont}
  Let $\catc$ be a category of coframes, and let $\varphi \colon L \to
  L'$ be a morphism in $\Cv\catc$.  Then, for every $\ell' \in L'$,
  \begin{eqnarray}
    \label{eqn:adh:image}
    \varphi_! (\adh_{L'} (\ell')) \leq \adh_{L} (\varphi_! (\ell')),
    \\
    \label{eqn:adh:img:inv}
    \adh_{L'} (\ell') \leq \varphi (\adh_L (\varphi_! (\ell'))).
  \end{eqnarray}
\end{prop}
\begin{proof}
  Let $\ell' \in L'$.  By Proposition~\ref{eqn:adh:raw:image}, in
  particular by (\ref{eqn:adh:raw:img:inv}),
  $\adhr_{L'} (a') \leq \varphi (\adhr_L (\varphi_! (a')))$ for every
  complemented $a' \geq \ell'$.  Hence:
  \begin{eqnarray*}
    \adh_{L'} (\ell')
    & = & \bigwedge_{a'\in\C_{L'}, a'\geq \ell'} \adhr_{L'}
          (a') \\
    & \leq & \bigwedge_{a'\in\C_{L'}, a'\geq \ell'} \varphi
             (\adhr_L (\varphi_! (a'))) \\
    & = & \varphi
          (\bigwedge_{a'\in\C_{L'}, a'\geq \ell'} \adhr_L (\varphi_! (a')))
  \end{eqnarray*}
  since $\varphi$ preserves all infima.  Among the complemented
  elements $a'$, we find those of the form $\varphi (a)$ with $a$
  complemented in $L$, so:
  \begin{eqnarray*}
    \adh_{L'} (\ell')
    & \leq & \varphi (\bigwedge_{\substack{a \in\C_L\\\ell' \leq \varphi (a)}} \adhr_L
    (\varphi_! (\varphi (a)))) \\
    & \leq & \varphi (\bigwedge_{\substack{a \in\C_L\\\ell' \leq \varphi (a)}}  \adhr_L
    (a)) \quad \text{since $\varphi_! (\varphi (a)) \leq a$} \\
    & = & \varphi (\bigwedge_{\substack{a \in\C_L\\\varphi_! (\ell') \leq a}}  \adhr_L (a))
    = \varphi (\adh_L (\varphi_! (\ell'))).
  \end{eqnarray*}
  This shows (\ref{eqn:adh:img:inv}).  (\ref{eqn:adh:image}) is an
  equivalent formulation.
\end{proof}
%%ADDED
\begin{cor}\label{cor:imageofclosed}
Let $\catc$ be a category of coframes, and let $\varphi \colon L \to
  L'$ be a morphism in $\Cv\catc$. If $c\in L$ is closed then $\varphi(c)$ is closed in $L'$.
\end{cor}
\begin{proof}
We use Fact \ref{fact:adh:closed}. In view of \eqref{eqn:adh:img:inv} with $\ell'=\varphi(c)$,
\[
\adh_{L'}(\varphi(c))\leq \varphi (\adh_L(\varphi_!(\varphi(c))))\leq \varphi(\adh_L c)\leq \varphi(c),\]
where the second inequality follows from $\varphi_!(\varphi(c))\leq c$ because $\varphi_!$ is left-adjoint to $\varphi$ and the last inequality follows from Fact \ref{fact:adh:closed} because $c$ is closed in $L$. We conclude that $\varphi(c)$ is closed in $L'$.
\end{proof}
\section{Adherence coframes}
\label{sec:adherence-coframes}

We have already mentioned the fact that it is equivalent to assume a
convergence space $X$ to be pretopological, or to assume that is
arises from a non-idempotent closure operator, i.e., an inflationary
operator $\nu$ on $\pow (X)$ that preserves finite suprema.  This is
also known as a \emph{\v{C}ech closure operator}; inflationary means
that $\ell \leq \nu (\ell)$ for every $\ell$.

To make this formal, call \emph{adherence space} any space $X$ with a
\v{C}ech closure operator $\nu_{\pow (X)}$.  Adherence spaces form a
category $\Adhc$, whose morphisms $f \colon X \to X'$ are the
\emph{continuous} maps, namely those maps such that
$f (\nu_{\pow (X)} (A)) \subseteq \nu_{\pow (Y)} (f (A))$ for every
$A \in \pow (X)$.  The fact that pretopological convergence spaces can
be described equivalently as adherence spaces translates to the
fact that there is an equivalence between the categories $\PreTop$ and
$\Adhc$.

\subsection{Pretopological coframes and adherence coframes}
\label{sec:pret-cofr-adher}

That equivalence is lost in the pointfree case, in general.  What will
remain is an adjunction between $\preTop\catc$ (resp.,
$\preTopcl\catc$) with a new category $\Adh\catc$ of $\catc$-objects
$L$ with an adherence-like operator $\nu \colon L \to L$.  We shall
drop the inflationary requirement, since $\nu = \adh_L$ need not be
inflationary, in view of Remark~\ref{rem:adh:center}.  This is
analogous to our dropping the Point Axiom in the definition of
convergence lattices.

The axioms that $\nu$ should satisfy are obtained by looking at the
properties of $\adh_L$, summarized in Proposition~\ref{prop:adh} and
Proposition~\ref{prop:adh:cont}.
\begin{defn}[Adherence coframe]
  \label{defn:Adh}
  Let $\catc$ be a category of coframes.  An \emph{adherence
    $\catc$-object} is an object $L$ of $\catc$ together with an
  operator $\nu_L \colon L \to L$ that is monotonic, preserves finite
  suprema of complemented elements, and satisfies $\nu_L (\ell) =
  \bigwedge_{a\in\C_L, a\geq \ell} \nu_L (a)$ for every
  $\ell \in L$.

  The adherence $\catc$-objects form a category $\Adh\catc$, whose
  morphisms $\varphi \colon L \to L'$ are the $\catc$-morphisms
  that are \emph{continuous} in that they satisfy:
  \begin{equation}
    \label{eqn:nu:image}
    \nu_{L'} (\ell') \leq \varphi (\nu_{L} (\varphi_! (\ell')))
  \end{equation}
  for every $\ell' \in L'$.
  % If equality holds, 
  % $\varphi$ is \emph{final}. % (See Corollary \ref{corl:coarsest:adh}).
\end{defn}
We shall call \emph{adherence structure} on $L$ any operator
$\nu \colon L \to L$ that is monotonic, preserves finite suprema of
complemented elements, and satisfies
$\nu (\ell) = \bigwedge_{a\in\C_L, a\geq \ell} \nu (a)$
for every $\ell \in L$.  The adherence $\adh_{\lm}$ of a convergence
structure $\lm$ on $L$ is always an adherence structure, by
Proposition~\ref{prop:adh}.

% \begin{rem}
%   \label{rem:final:adh}
%   As with Remark~\ref{rem:final:conv}, our final morphisms
%   $\varphi \colon L \to L'$ are exactly the categorical notion.
%   Categorically, $\varphi$ is final if and only if for every object
%   $L''$ in $\catc$, for every morphism $\psi \colon L' \to L''$ in
%   $\catc$ such that $\psi \circ \varphi$ is continuous, $\psi$ is
%   continuous.  If $\varphi$ is final in this sense, then take
%   $L''=L'$, $\psi = \identity {L'}$, and
%   $\nu_{L''} (\ell') = \varphi (\nu_{L} (\varphi_! (\ell')))$: then
%   $\varphi$ is final in the sense of Definition~\ref{defn:Adh}.  The
%   converse implication is immediate.  Here $\nu_{L''}$ is an example
%   of a final adherence structure on $L'$.  We shall generalize this
%   argument in Corollary \ref{corl:coarsest}.
% \end{rem}

The following shows that adherence structures on coframes $L$ are,
equivalently, those finite supremum preserving self-maps
$\nu \colon L \to L$ such that
$\nu (\ell) = \bigwedge_{a\in\C_L, a \geq \ell} \nu (a)$ for every $\ell \in L$.
\begin{lem}
  \label{lemma:sup:compl}
  Let $L$ be a coframe.  Every adherence structure $\nu$ on $L$
  preserves all finite suprema.
\end{lem}
\begin{proof}
  Let $\ell_1$, \ldots, $\ell_n$ be finitely many elements.  The
  inequality $\nu (\ell_1 \vee \cdots \vee \ell_n) \geq \nu (\ell_1)
  \vee \cdots \vee \nu (\ell_n)$ follows from monotonicity.  In the
  other direction, and letting $a$, $a_1$, \ldots, $a_n$ range over
  complemented elements:
  \begin{eqnarray*}
    \nu (\ell_1 \vee \cdots \vee \ell_n)
    & = & \bigwedge_{a \geq \ell_1 \vee \cdots \vee \ell_n} \nu (a) \\
    & \leq & \bigwedge_{a_1 \geq \ell_1, \ldots, a_n \geq \ell_n} \nu (a_1
             \vee \cdots \vee a_n) \\
    && \text{since }a_1 \geq \ell_1, \ldots, a_n \geq \ell_n\text{ imply }
       a:=a_1\vee \cdots \vee a_n \geq \ell_1 \vee \cdots \vee \ell_n \\
    & = & \bigwedge_{a_1 \geq \ell_1, \ldots, a_n \geq \ell_n} (\nu (a_1)
             \vee \cdots \vee \nu (a_n)) \\
    & = & \left(\bigwedge_{a_1 \geq \ell_1} \nu (a_1)\right) \vee
          \cdots
          \vee \left(\bigwedge_{a_n \geq \ell_n} \nu (a_n)\right) \\
    && \text{by the coframe distributivity law} \\
    & = & \nu (\ell_1) \vee \cdots \vee \nu (\ell_n).
  \end{eqnarray*}
\end{proof}

\begin{lem}
  \label{lemma:adh:preconv}
  Let $\catc$ be a category of coframes, and $L$ be a $\catc$-object.
  Every adherence structure $\nu$ on $L$ defines a classical
  pretopological convergence structure $\lm_\nu$ by:
  \begin{equation}
    \label{eq:limnu}
    \lm_\nu \F := \bigwedge_{a\in\C_L, a\in \F^\#} \nu (a).
  \end{equation}
  Moreover:
  \begin{enumerate}
  \item The mapping $\nu \mapsto \lm_\nu$ is monotonic: if $\nu\leq \nu'$ then $\lim_\nu\leq \lim_{\nu'}$.
  \item The mapping $\lm \mapsto \adh_{\lm}$ that sends every
    convergence structure to its adherence is monotonic: if $\lim_1\leq\lim_2$ then $\adh_{\lim_1}\leq \adh_{\lim_2}$.
  \item For every convergence structure $\lm$ on $L$, for every filter
    $\F$ on $L$, $\lm \F \leq \lm_{\adh_{\lm}} \F$.
  \item For every adherence structure $\nu$ on $L$, for every $\ell
    \in L$, $\adh_{\lm_\nu} \ell \leq \nu (\ell)$.
  \end{enumerate}
\end{lem}
In other words, the $\lm \mapsto \adh_{\lm}$ construction is left
adjoint to the $\nu \mapsto \lm_\nu$ construction.

\begin{proof}
  Clearly $\lm_\nu$ is monotonic: if $\F \subseteq \G$ then
  $\G^\# \subseteq \F^\#$ by Lemma~\ref{lemma:grill}~(3), so
  $\lm_\nu \F \leq \lm_\nu \G$.  Pretopologicity is the reason why we
  only quantify over complemented elements $a$ in (\ref{eq:limnu}).
  Indeed, let ${(\F_i)}_{i \in I}$ be a family of filters on $L$, and
  let $\F := \bigcap_{i \in I} \F_i$.  If $a$ is complemented then
  $a$ is in $\F^\#$ if and only if its complement $a^*$ is not in
  $\F$, by Lemma~\ref{lemma:grill:compl}, if and only if there is an
  $i \in I$ such that $a^*$ is not in $\F_i$, if and only if there is
  an $i \in I$ such that $a$ is in $\F_i^\#$.  It follows that
  \begin{eqnarray*}
    \lm_\nu \F
    & = & \bigwedge_{a\in\C_L, a \in
          \bigcup_{i\in I}\F_i^\#} \nu (a) \\
    & = & \bigwedge_{i \in I}\; \bigwedge_{a\in\C_L, a \in
          \F_i^\#} \nu (a) = \bigwedge_{i \in I} \lm_\nu \F_i.
  \end{eqnarray*}
  The fact that $\lm_\nu$ is classical follows from the definition of
  $\lm_\nu$ and the second part of Corollary~\ref{corl:grill:compl:=}.

  Claims (1) and (2) are immediate.  For (3),
  \begin{eqnarray*}
    \lm_{\adh_{\lm}} \F
    & = & \bigwedge_{a\in\C_L, a\in \F^\#}\;
          \bigwedge_{b\in\C_L, b\geq a}\; \bigvee_{\substack{\G \in
          \Filt L\\b \in \G^\#}} \lm \G \\
    & = & \bigwedge_{a\in\C_L, a\in \F^\#}\;
             \bigvee_{\substack{\G \in
             \Filt L\\a \in \G^\#}} \lm \G \\
    &&\quad\text{since $\F^\#$ is upwards-closed by
    Lemma~\ref{lemma:grill}~(1)}\\
    & \geq & \lm \F \quad\text{since one can take $\G=\F$ in the supremum.}
  \end{eqnarray*}

  For (4),
  \begin{eqnarray*}
    \adh_{\lm_\nu} \ell
    & = & \bigwedge_{a\in\C_L, a\geq \ell}\;
          \bigvee_{\substack{\F \in \Filt L\\a \in \F^\#}}\;
    \bigwedge_{b\in\C_L, b\in \F^\#} \nu (b) \\
    & \leq & \bigwedge_{a\in\C_L, a\geq \ell}\;
          \bigvee_{\substack{\F \in \Filt L\\a \in \F^\#}} \nu (a)
    %\quad\text{since $a$ is among the possible $b$s}
    \\
    & = & \bigwedge_{a\in\C_L, a\geq \ell} \nu (a) = \nu (\ell)
  \end{eqnarray*}
  where the last equality is part of the definition of an adherence
  structure.  Note that we have used the equality:
  $$\bigvee_{\substack{\F \in \Filt L\\a \in \F^\#}} \nu (a) = \nu
  (a),$$ which looks obvious but actually requires some care.  That
  equality holds unless there is no filter $\F$ such that
  $a \in \F^\#$ (in which case the left-hand side is $\bot$) and
  $\nu (a) \neq \bot$.  If $a \neq \bot$, there there is a filter $\F$
  such that $a \in \F^\#$, for example $\{\top\}$.  If $a = \bot$,
  then $\nu (a)=\bot$ since $\nu$ preserves finite suprema (hence
  empty suprema) of complemented elements, so the equality holds in
  all cases.
\end{proof}

\begin{prop}
  \label{prop:Adh:univ}
  Let $\catc$ be an admissible category of coframes.  There is an
  identity-on-morphisms functor $(L, \lm_L) \mapsto (L, \adh_L)$ from
  $\preTop\catc$ to $\Adh\catc$, which is right adjoint to the
  identity-on-morphisms functor $(L, \nu) \mapsto (L, \lm_\nu)$ from
  $\Adh\catc$ to $\preTop\catc$.  Similarly with $\preTopcl\catc$ in
  lieu of $\preTop\catc$.
\end{prop}
\begin{proof}
  Write $U$ for the first functor, and $F$ for the second.  By
  ``identity-on-morphisms'', we mean that $U (\varphi) = \varphi$ and
  $F (\varphi)=\varphi$ for every morphism $\varphi$.  If $\varphi$ is
  a morphism from $(L, \lm_L)$ to $(L', \lm_{L'})$, $U (\varphi)$ is a
  morphism from $U (L, \lm_L) = (L, \adh_L)$ to
  $U (L', \lm_{L'}) = (L', \adh_{L'})$ by
  Proposition~\ref{prop:adh:cont}.

  As far as $F$ is concerned, consider a morphism
  $\varphi \colon (L, \nu_L) \to (L', \nu_{L'})$ of adherence
  $\catc$-objects.  We must show that $\varphi$ is continuous from
  $(L, \lm_{\nu_L})$ to $(L', \lm_{\nu_{L'}})$.  Fix a filter $\F$ on
  $L'$.
  \begin{eqnarray*}
    \lm_{\nu_L} \varphi^{-1} (\F)
    & = &
          \bigwedge_{a\in\C_L\cap \varphi^{-1}
          (\F)^\#} \nu_L (a) \\
    & \geq & \bigwedge_{a\in\C_L\cap \varphi^{-1} (\F)^\#} \nu_L (\varphi_! (\varphi (a)))
             \quad \text{since $\varphi_!$ is left adjoint to
             $\varphi$} \\
    & \geq & \bigwedge_{a'\in\C_{L'}\cap \F^\#} \nu_L (\varphi_! (a')),
  \end{eqnarray*}
  since by Corollary \ref{cor:imagepseudocompl}, for every complement $a$ in
  $\varphi^{-1} (\F)^\#$, one can find a complemented element
  $a' := \varphi (a)$ in $\F^\#$.

  Applying $\varphi$ and using the fact that $\varphi$ preserves
  infima,
  \begin{eqnarray*}
    \varphi (\lm_{\nu_L} \varphi^{-1} (\F))
    & \geq & \bigwedge_{a'\in\C_{L'}
             \cap \F^\#} \varphi (\nu_L (\varphi_! (a'))) \\
    & \geq & \bigwedge_{a'\in\C_{L'}
             \cap \F^\#} \nu_{L'} (a') = \lm_{\nu_{L'}} \F.
  \end{eqnarray*}
  The second inequality stems from the continuity of $\varphi$ as a
  morphism of adherence $\catc$-objects, namely condition
  (\ref{eqn:nu:image}).  We recognize the formula for continuity as a
  morphism of convergence $\catc$-objects, namely (\ref{eq:morphism}).

  Showing $F \dashv U$ is now easy.  The unit
  $\eta_L \colon (L, \nu_L) \to UF (L, \nu_L) = (L,
  \adh_{\lm_{\nu_L}})$ is the identity map.  Continuity
  (\ref{eqn:nu:image}) boils down to checking that for every
  $\ell \in L$, $\adh_{\lm_{\nu_L}} (\ell) \leq \nu_L (\ell)$, and
  that is Lemma~\ref{lemma:adh:preconv}~(4).  The counit
  $\epsilon_L \colon FU (L, \lm_L) = (L, \lm_{\adh_L}) \to (L, \lm_L)$
  is also the identity map.  Continuity (\ref{eq:morphism}) boils down
  to checking that for every filter $\F$ on $L$,
  $\lm_L \F \leq \lm_{\adh_L} \F$, and that is
  Lemma~\ref{lemma:adh:preconv}~(3).  The fact that $\eta_L$ and
  $\epsilon_L$ are natural in $L$, and that $\epsilon_{F(L)} \circ F
  (\eta_L)$ and $U (\epsilon_L) \circ \eta_{U (L)}$ are identities are
  obvious since all concerned maps are identities.
\end{proof}

The counit of that adjunction is the identity map from $(L,
\lm_{\adh_L})$ to $(L, \lm_L)$, and the unit is the identity map from
$(L, \nu_L)$ to $(L, \adh_{\lm_{\nu_L}})$.
% epsilon : FU L -> L  (L lim)
% L lim --> U L adh
% L adh --> F L lim
% sur FU L, on a lim_{adh_lim} >= lim

\subsection{Complete distributivity and the equivalence between
  pretopological and adherence structures}
\label{sec:compl-distr-equiv}

We now investigate sufficient conditions for the inequalities in
Lemma~\ref{lemma:adh:preconv}~(3), (4) to be equalities.

A frame $\Omega$ is \emph{spatial} if and only if every element is an
infimum of meet-prime elements.  Under the Axiom of Choice, this is equivalent
to requiring that $\Omega$ be order-isomorphic to the open-set lattice
of some topological space, but we will not use that.

\begin{defn}\label{def:spatialcoframe}
We say that a coframe $L$ is \emph{spatial} if and only if its dual
frame $L^{op}$ is spatial.  In other words, $L$ is spatial if and only
if every element is a supremum of join-primes.
\end{defn}

On a complete lattice $L$, and following \cite{contlattices}, we say
that $\ell$ is \emph{way-way-below} $\ell'$, in notation
$\ell \lll \ell'$, if and only if for every subset $S$ of $L$ such
that $\ell' \leq \bigvee S$, some element of $S$ is larger than or
equal to $\ell$.  $L$ itself is \emph{prime-continuous} if and only if
every element $\ell$ of $L$ is the supremum of elements way-way-below
$\ell$.  The prime-continuous complete lattices are exactly the
completely distributive complete lattices, as first observed by Raney
\cite{Raney:compl:distr}, see \cite[Exercice~IV-3.31]{contlattices},
but that observation requires the Axiom of Choice.

Again using the Axiom of Choice, prime-continuity implies spatiality.
In fact, prime-continuity implies continuity and distributivity, and
the continuous distributive spatial lattices are exactly the
completely distributive lattices
\cite[Proposition~VII-2.10~(7)]{contlattices}.
The following is choice-free, as is everything else in this paper.
\begin{prop}
  \label{prop:adh:preconv:=}
  Let $\catc$ be a category of coframes, and $L$ be a $\catc$-object.
  \begin{enumerate}
  \item If $L$ is spatial, then $\nu = \adh_{\lm_\nu}$ for every
    adherence structure $\nu$ on $L$.
  \item If $L$ is prime-continuous, then $\lm = \lm_{\adh_{\lm}}$ for
    every classical pretopological convergence structure $\lm$ on $L$.
  \item If $L$ is both, then the maps $\nu \mapsto \lm_\nu$ and
    $\lm \mapsto \adh_{\lm}$ define an order isomorphism between
    adherence structures and classical pretopological structures on
    $L$.
  \end{enumerate}
\end{prop}
\begin{proof}
  (1) Let $\nu$ be an adherence structure on $L$.  In order to show
  that $\nu = \adh_{\lm_\nu}$, and considering
  Lemma~\ref{lemma:adh:preconv}~(4), it is enough to show that
  $\adh_{\lm_\nu} \ell \geq \nu (\ell)$ for every $\ell \in L$.  Since
  $\nu_L (\ell) = \bigwedge_{a\in\C_L, a \geq \ell} \nu_L
  (a)$ for every $\ell \in L$, and by the definition $\adh_{\lm_\nu}$,
  it is enough to show $\adhr_{\lm_\nu} a \geq \nu (a)$ for every
  complemented element $a$ of $L$.

  For that, we consider an arbitrary join-prime $\ell \leq \nu (a)$,
  and we claim that there is a filter $\F$ on $L$ such that
  $a \in \F^\#$ and, for every complemented element $b$ of $\F^\#$,
  $\ell \leq \nu (b)$.  We simply define $\F$ as the upwards-closure
  of the set of complements $\overline b$ of complemented elements $b$
  such that $\ell \not\leq \nu (b)$.  Since $\ell$ is join-prime,
  $\ell \not\leq \bot$.  Since $\nu$ preserves finite suprema of
  complemented elements, in particular the supremum $\bot$ of the
  empty family, $\ell \not\leq \nu (b)$ for $b = \bot$.  Hence $\F$
  contains the complement $\top$ of $\bot$, hence is not empty.  In
  order to show that $\F$ is a filter, we only have to show that for
  any two complements $\overline b_1$ and $\overline b_2$ such that
  $\ell \not\leq \nu (b_1), \nu (b_2)$, their infimum
  $\overline b_1 \wedge \overline b_2 = \overline{b_1 \vee b_2}$ is
  such that $\ell \not\leq \nu (b_1 \vee b_2)$.  This is obtained by
  contraposition: if $\ell \leq \nu (b_1 \vee b_2)$, then
  $\ell \leq \nu (b_1)$ or $\ell \leq \nu (b_2)$ because $\nu$
  preserves finite suprema of complemented elements and because $\ell$
  is join-prime.  Therefore $\F$ is a filter.  Since
  $\ell \leq \nu (a)$, the complement $\overline a$ of $a$ is not in $\F$,
  and by Lemma~\ref{lemma:grill:compl}, $a$ is in $\F^\#$.

  Having proved this, we obtain that for every join-prime
  $\ell \leq \nu (a)$,
  \[
    \ell \leq \bigvee_{\F \in \Filt L, a \in \F^\#} \bigwedge_{b\in\C_L \cap \F^\#} \nu (b) = \adhr_{\lm_\nu} (a).
  \]
  Because every element is a supremum of join-primes,
  $\nu (a) \leq \adhr_{\lm_\nu} (a)$, and this concludes our argument.

  % Taking infima over all
  % complemented elements $a \geq \ell$, we obtain
  % $\nu (\ell) \leq \adh_{\lm_\nu} (\ell)$.

  (2) Let $\lm$ be a classical pretopological convergence structure on
  $L$.  In order to show that $\lm = \lm_{\adh_{\lm}}$, and
  considering Lemma~\ref{lemma:adh:preconv}~(3), it is enough to show
  that $\lm \F \geq \lm_{\adh_{\lm}} \F$ for every filter $\F$ on $L$.
  Since $L$ is prime-continuous, it is enough to show that, for every
  $\ell \lll \lm_{\adh_{\lm}} \F$, $\ell \leq \lm \F$.  Recall that
  \[\lm_{\adh_{\lm}} \F = \bigwedge_{a\in\C_L \cap \F^\#}
  \adh_{\lm} (a) = \bigwedge_{a\in\C_L  \cap \F^\#}
  \adhr_{\lm} (a)\] by Proposition~\ref{prop:adh}~(2), so that 
  \[\lm_{\adh_{\lm}} \F= \bigwedge_{a\in\C_L  \cap \F^\#} \bigvee_{\G \in
    \Filt L, a \in \G^\#} \lm \G.\]

  Since $\ell \lll \lm_{\adh_{\lm}} \F$, for every complemented
  element $a$ in $\F^\#$, there is a filter $\G$ such that
  $a\in \G^\#$ and $\ell \leq \lm \G$.  Consider the family $\A$ of
  all filters $\G$ such that $\ell \leq \lm \G$ and such that $\G^\#$
  contains a complemented element of $\F^\#$, and look at the filter
  $\bigcap \A$.  Since $\lm$ is pretopological,
  $\lm \bigcap \A = \bigwedge_{\G \in \A} \lm \G \geq \ell$.

  For every complemented element $a$ of $\bigcap \A$, we claim that
  $a$ is in $\F$.  Otherwise, by Lemma~\ref{lemma:grill:compl} its
  complement $\overline a$ is in $\F^\#$, and in that case we have
  seen that there is a filter $\G$ such that $\overline a \in \G^\#$
  and $\ell \leq \lm \G$, namely, $\G$ is in $\A$.  By
  Lemma~\ref{lemma:grill:compl} again and since
  $\overline a \in \G^\#$, $a$ is not in $\G$, hence not in
  $\bigcap \A$.

  Therefore $(\bigcap \A)_c \subseteq \F$ (see (\ref{eq:c})).  Since
  $\lm$ is classical,
  $\lm \bigcap \A = \lim (\bigcap \A)_c \leq \lm \F$, and since
  $\lm \bigcap \A \geq \ell$, we conclude that $\ell \leq \lm \F$.

  (3) is immediate from (1) and (2).
%   % For a complemented element $a$, recall from
%   % Lemma~\ref{lemma:grill:compl} that $a \in \F^\#$ if and only if $a^*
%   % \not\in \F$.
% For a given complemented $a$ in $\F^\#$ such that $a \neq \bot$, the family of filters $\G$ such
% that $a \in \G^\#$ is directed.  It is
% non-empty because $\G = \{\top\}$ fits, since $a \neq \bot$.
% For two filters $\G_1$, $\G_2$ such that $a$ is both in $\G_1^\#$ and
% $\G_2^\#$, we claim that $a$ is in $(\G_1 \vee \G_2)^\#$.
% Otherwise, $a^*$ would be in $\G_1 \vee \G_2$
%
% Assume that were not the case.
% For every complemented element $a \in \F^\#$, 
\end{proof}

\begin{cor}
  \label{corl:Adh:adj}
  Let $\catc$ be an admissible category of coframes.  
  The adjunction of Proposition~\ref{prop:Adh:univ} restricts to:
  \begin{enumerate}
  \item a coreflection of the full subcategory of spatial objects of
    $\preTop\catc$ (resp., of $\preTopcl\catc$) into that of spatial
    objects of $\Adh\catc$;
  \item a reflection of the full subcategory of prime-continuous
    objects of $\preTopcl\catc$ into that of prime-continuous objects
    of $\Adh\catc$;
  \item an adjoint equivalence between the full subcategory of spatial
    prime-continuous objects of $\preTopcl\catc$ and that of spatial
    prime-continuous objects of $\Adh\catc$.
  \end{enumerate}
\end{cor}
\begin{proof}
  The unit is an isomorphism in the first case, the counit is an
  isomorphism in the second, and both are in the final case.
\end{proof}

\begin{rem}
  \label{rem:adh:preconv:=}
  Even without using the Axiom of Choice, one can check that
  $\pow (X)$ is spatial and prime-continuous, for every set $X$.  Its
  join-primes are the one-element sets $\{x\}$, and spatiality boils
  down to the fact that every subset $A$ of $X$ is the union of its
  one-element subsets.  The way-way-below relation is given by
  $A \lll B$ if and only if $A \subseteq \{x\}$ for some $x \in B$.
  Then, each subset $B$ is the union of its one-element subsets
  $\{x\}$, and for each $x \in B$, $\{x\} \lll B$; this shows
  prime-continuity.
%
  % Specializing Proposition~\ref{prop:adh:preconv:=}, we reobtain the
  % announced fact that a pretopological convergence space $X$ is the
  % same thing as \v{C}ech-closure space, namely a set $X$ together with
  % a map $\nu \colon \pow (X) \to \pow (X)$ that preserves finite
  % suprema (namely, an adherence structure on $\pow (X)$, simplifying
  % the definition by way of the fact that every element of $\pow (X)$
  % is complemented).
% No: centeredness is missing.
\end{rem}

\subsection{A square of adjunctions}
\label{sec:square-adjunctions}

In a world with points, Remark~\ref{rem:adh:preconv:=} and
Proposition~\ref{prop:adh:preconv:=}, together with the fact that
$\pow (X)$ is a Boolean algebra, allow us to recover the well-known
fact that $\lm \mapsto \adh_{\lm}$ and $\nu \mapsto \lm_\nu$
are inverse bijections between pretopological convergence structures
on $\pow (X)$ and \v{C}ech closure operators on $\pow (X)$.  We refine
this as follows.
\begin{lem}
  \label{lemma:PreTop=Adh}
  The functors $(X, \lm) \mapsto (X, \adh_{\lm})$ and
  $(X, \nu) \mapsto (X, \lm_\nu)$ define an adjoint equivalence of
  categories between $\PreTop$ and $\Adhc$.
\end{lem}
\begin{proof}
  What remains to be proved is that those are actual functors.  Let
  $f$ be a continuous map from $(X, \lm_{\pow (X)})$ to
  $(Y, \lm_{\pow (Y)})$ in $\Conv$.  Then $\varphi := f^{-1}$ is
  continuous from $(\pow (Y), \lm_{\pow (Y)})$ to
  $(\pow (X), \lm_{\pow (X)})$ in $\Cv\coFrm$
  (Proposition~\ref{prop:Conv->Foc}).  By
  Proposition~\ref{prop:adh:cont}, $\varphi$ is also continuous from
  $(\pow (Y), \adh_{\lm_{\pow (Y)}})$ to
  $(\pow (X), \adh_{\lm_{\pow (X)}})$.  Now read
  (\ref{eqn:adh:image}), and recall that $\varphi_!$ is the direct
  $f$-image operator: this shows that
  $f (\adh_{\lm_{\pow (X)}} (A)) \subseteq \adh_{\lm_{\pow (Y)}} (f
  (A))$ for every subset $A$ of $X$, namely that $f$ is continuous in
  the sense of adherence spaces.

  Conversely, let
  $f \colon (X, \nu_{\pow (X)}) \to (X, \nu_{\pow (Y)})$ be a morphism
  of adherence spaces.  Similarly, $\varphi := f^{-1}$ is a morphism
  from $(\pow (Y), \nu_{\pow (Y)})$ to $(\pow (X), \nu_{\pow (X)})$ in
  $\Adh\catc$, hence one from $(\pow (Y), \lm_{\nu_{\pow (Y)}})$ to
  $(\pow (X), \lm_{\nu_{\pow (X)}})$ in $\preTopcl\catc$ by
  Proposition~\ref{prop:Adh:univ}.  Therefore, for every filter $\F$
  of subsets of $X$,
  $\lm_{\nu_{\pow (X)}} \F \subseteq \varphi (\lm_{\nu_{\pow (Y)}}
  \varphi^{-1} (\F))$.  In particular, for every point $x \in X$ that
  $\F$ converges to, $f (x)$ is in
  $\lm_{\nu_{\pow (Y)}} \varphi^{-1} (\F)$.  Since
  $\varphi^{-1} (\F) = f [\F]$, $f [\F]$ converges to $f (x)$ with
  respect to $\lm_{\nu_{\pow (Y)}}$.  
\end{proof}

Turning to the opposite categories, we obtain a functor
$(L, \lm_L) \mapsto (L, \adh_L)$ from ${(\preTop\catc)}^{op}$ (or
$(\preTopcl\catc)^{op}$) to ${(\Adh\catc)}^{op}$, left adjoint to the
functor $(L, \nu) \mapsto (L, \lm_\nu)$ from ${(\Adh\catc)}^{op}$ to
${(\preTop\catc)}^{op}$.  Composing that with the coreflection
$\pow\dashv\pt$ between $\PreTop$ and ${(\preTop\catc)}^{op}$, we
obtain an adjunction between $\PreTop$ and $\Adh\catc$, hence also
between $\Adhc$ and $\Adh\catc$.  This is the rightmost adjunction in
the following square, which is the third square from the left in
(\ref{eq:adj}) if we define $\catc$ as $\coFrm$.
\begin{equation}
  \label{eq:adj:adh}
  \xymatrix{
    {(\preTopcl\catc)}^{op}
    \ar@<1ex>[r]\ar@{}[r]|{\perp}
    \ar@<1ex>[d]^{\pt}\ar@{}[d]|{\dashv}
    &
    {(\Adh\catc)}^{op}
    \ar@<1ex>[l]
    \ar@<1ex>[d]^{\pt}\ar@{}[d]|{\dashv}
    \\
    \PreTop
    \ar@<1ex>[u]^{\pow}
    \ar@{=}[r]
    &
    \Adhc
    \ar@<1ex>[u]^{\pow}
    % \ar@{}[ru]|{\text{(Sec.~\ref{sec:adherence-coframes})}}
  }
\end{equation}
We again write $\pow\dashv\pt$ for that adjunction between $\Adhc$ and
${(\Adh\catc)}^{op}$, and our purpose is now to make it explicit.

The left adjoint $\pow$ maps every pretopological convergence space
$X$ to $(\pow (X), \adh_{\pow (X)})$ and every continuous map
$f \colon X \to Y$ to $\pow (f) = f^{-1}$.  Note that
$\adh_{\pow (X)}$ coincides with the raw adherence $\adhr_{\pow (X)}$
because every element of $\pow (X)$ is complemented, and that the raw
adherence is the usual notion of adherence on (pretopological)
convergence spaces.

Writing $U \dashv F$ for the adjunction between
${(\preTopcl\catc)}^{op}$ and ${(\Adh\catc)}^{op}$ on the top row of
(\ref{eq:adj:adh}), at least temporarily, the unit of the rightmost
$\pow\dashv\pt$ adjunction is:
\[
\xymatrix{
  X \ar[r]^{\eta_X} &
  \pt \pow (X) \ar[r]^{\pt u_{\pow (X)}} &
  \pt F U \pow (X)
}
\]
where $\eta_X \colon x \mapsto \dot x$ is the unit of the adjunction
$\pow\dashv\pt$ between $\Conv$ (or $\PreTop$) and ${(\Cv\catc)}^{op}$
(or ${(\preTopcl\catc)}^{op}$), and $u_L \colon FUL \to L$ is the unit
of the adjunction $U \dashv F$, written as a morphism in the opposite
category.  By Proposition~\ref{prop:adh:preconv:=} and
Remark~\ref{rem:adh:preconv:=}, $u_{\pow (X)}$ is an isomorphism, and
by Lemma~\ref{lemma:eta} $\eta_X$ is an isomorphism.  It follows that
the unit of the rightmost adjunction of (\ref{eq:adj:adh}) is an
isomorphism, whence:
\begin{prop}
  \label{prop:adh:corefl}
  Let $\catc$ be an admissible category of coframes.  $\Adhc$ is a
  coreflective subcategory of ${(\Adh\catc)}^{op}$.
\end{prop}

% Conv : P -| pt : Foc, Foc : U -| F : Adhop donc Conv : U P -| pt F : Adhop
% unit : X -> pt F U P X
% = 
% ou eta : X -> pt P X

The right adjoint, which we again write as $\pt$, maps every adherence
$\catc$-object $(L, \nu_L)$ to $\pt (L, \lm_{\nu_L})$.  Let us spell
out what this means.
\begin{lem}
  \label{lemma:nu:point}
  Let $\catc$ be an admissible category of coframes, and $L$ be an
  adherence $\catc$-object.  The points of $(L, \lm_{\nu_L})$ are
  exactly the join-prime elements $x$ of $L$ such that $x \leq \nu_L (x)$.
\end{lem}
\begin{proof}
  Following Remark~\ref{rem:ccf:point}, a point of $(L, \lm_{\nu_L})$
  is a join-prime $x$ of $L$ such that $x \leq \lm_{\nu_L} (\upc x)$.
  Now $\lm_{\nu_L} (\upc x)$ is equal to
  $$\bigwedge_{a \in (\upc x)^\#} \nu_L (a) = \bigwedge_{a\in\C_L:\overline a
    \not\geq x} \nu_L (a),$$ where  $\overline a$ is the complement of $a$.  Since $x$ is
  join-prime, and $\top = a \vee \overline a \geq x$, $a$ or
  $\overline a$ is larger than or equal to $x$.  Since $x$ is
  join-prime again, $x$ is different from $\bot$, so $a$ and
  $\overline a$ cannot be both larger than or equal to $x$.  It
  follows that the condition $\overline a \not\geq x$ is equivalent to
  $a \geq x$.  We obtain that $\lm_{\nu_L} (\upc x)$ is equal to
  $\bigwedge_{a \geq x} \nu_L (a)$, where $a$ ranges over the
  complemented elements of $L$.  That is equal to $\nu_L (x)$, whence
  the conclusion.
\end{proof}

We make that into a definition.
\begin{defn}[Point]
  \label{defn:nu:point}
  Let $\catc$ be an admissible category of coframes.  For every
  adherence $\catc$-object $L$, the \emph{points} of $L$ are the
  join-prime elements $x$ such that $x \leq \nu_L (x)$, namely the points of
  $(L, \lm_{\nu_L})$.  We write again $\pt L$ for the set of points of $L$.
\end{defn}
Define again $\ell^\bullet$ as $\{x \in \pt L \mid x \leq \ell\}$.
The only change compared to Definition~\ref{defn:kow} is that $\pt L$
now has to be understood per Definition~\ref{defn:nu:point}.
% The set $\pt L$ comes with a notion of convergence, given by
% $\lm_{\pow \pt L} \F = {(\lm_{\nu_L} \kow\F)}^\bullet$
% (Definition~\ref{defn:pt:lim}), where
% $\ell^\bullet = \{x \in \pt L \mid x \leq \ell\}$ and
% $\kow \F = \{\ell \mid \ell^\bullet \in \F\}$.
The following is an analogue of Lemma~\ref{lemma:kow:basic}~(1), (3),
(4) for adherence
coframes; item~(4) is new, and could have been observed back then,
too.
%; item~(6) is the analogue of Lemma~\ref{lemma:compl:bullet}.
\begin{lem}
  \label{lemma:kow:basic:adh}
  Let $\catc$ be an admissible category of coframes, and $L$ be an
  adherence $\catc$-object.  The following hold:
  \begin{enumerate}
  \item For all $\ell, \ell'$ in $L$, if $\ell \leq \ell'$ then
    $\ell^\bullet \subseteq {\ell'}^\bullet$.
  % \item For all $\F, \G \in \Filt\pow (\pt L)$, if $\F\subseteq\G$
  %   then $\kow{\F}\subseteq\kow{\G}$.
  \item For all $\ell_{1},\ell_{2},\cdots,\ell_{n}\in L$ ($n\in\nat$),
    $\bigcup_{i=1}^{n}\ell_{i}^{\bullet}={(\bigvee_{i=1}^{n}\ell_{i})}^{\bullet}$.
  \item For every family ${(\ell_i)}_{i \in I}$ in $L$,
    $\bigcap_{i\in I}\ell_{i}^{\bullet}={(\bigwedge_{i\in
        I}\ell_{i})}^{\bullet}$.
  % \item For every $\F\in\Filt \pow(\pt L)$, $\kow{\F}$ is a filter on
  %   $L$, i.e., $\mathcal{\kow{F}}\in\Filt L$.
  \item For every complemented element $a$ of $L$, the complement
    of $a^\bullet$ in $\pt L$ is ${(\overline a)}^\bullet$, where
    $\overline a$ is the complement of $a$ in $L$.
  \end{enumerate}
\end{lem}
\begin{proof}
  (1), (3) are obvious, and (2) follows from the fact that the
  elements of ${(\bigvee_{i=1}^{n}\ell_{i})}^{\bullet}$ are
  join-prime.  In order to show (4), we use (2) to obtain
  $a^\bullet \cup {(\overline a)}^\bullet = {(a \vee \overline
    a)}^\bullet = \top^\bullet = \pt L$, and (3) to obtain
  $a^\bullet \cap {(\overline a)}^\bullet = {(a \wedge \overline
    a)}^\bullet = \bot^\bullet = \emptyset$.
\end{proof}

The convergence space $\pt L$ is then pretopological.  Equivalently,
$\pt L$ is a space with a \v{C}ech closure operator
$\nu_{\pow \pt L} = \adh_{\lm_{\pow\pt L}}$, which we now describe
explicitly.
\begin{lem}
  \label{lemma:nu:adh}
  Let $\catc$ be an admissible category of coframes.  For every
  adherence $\catc$-object $L$, for every $S \in \pow \pt L$,
  $\nu_{\pow\pt L} (S) = (\nu_L (\bigvee S))^\bullet$.
\end{lem}
Here $S$ is a set of points, hence in particular a subset of $L$:
$\bigvee S$ is the supremum of that set in $L$.

\begin{proof}
  Define $\nu (S)$ as $(\nu_L (\bigvee S))^\bullet$.  The outline of
  the proof is as follows.  We check that $\nu$ is a \v{C}ech-closure
  operator, that its associated pretopological convergence $\lm_\nu$
  is exactly $\lm_{\pow\pt L}$, and we shall conclude that $\nu$ is
  equal to $\adh_{\lm_{\pow\pt L}}$ by
  Proposition~\ref{prop:adh:preconv:=}.

  Let us start by checking that $\nu$ is a \v{C}ech-closure operator,
  namely that $\nu$ commutes with finite suprema.  Since $\nu_L$ is
  monotonic and by Lemma~\ref{lemma:kow:basic:adh}~(1), $\nu$ is
  monotonic.  Let $S = S_1 \cup \cdots \cup S_n$.  Since $\nu$ is
  monotonic, $\bigcup_{i=1}^n \nu (S_i) \subseteq \nu (S)$.  To show
  the converse inequality, let $a_1$, \ldots, $a_n$ be arbitrary
  complemented elements such that $a_i \geq \bigvee S_i$,
  $1\leq i\leq n$.  Then
  $a := a_1 \vee \cdots \vee a_n \geq \bigvee S$, and
  $\nu_L (a) = \bigvee_{i=1}^n \nu_L (a_i)$.  Taking infima over
  $a_1$, \ldots, $a_n$, we obtain
  \[\nu (\bigvee S) \leq \bigwedge_{\substack{\forall i=1, \cdots, n
      .\\a_i \geq \bigvee S_i}} \bigvee_{i=1}^n \nu_L (a_i) =
  \bigvee_{i=1}^n \bigwedge_{a_i \geq \bigvee S_i} \nu_L (a_i)\] by
  the coframe distributivity law, and 
  $\bigvee_{i=1}^n \bigwedge_{a_i \geq \bigvee S_i} \nu_L (a_i)= \bigvee_{i=1}^n \nu_L (\bigvee S_i)$.  We now use
  Lemma~\ref{lemma:kow:basic:adh}~(1) and~(2) to obtain
  \[\nu (S) = {(\nu_L (\bigvee S))}^\bullet \subseteq {(\bigvee_{i=1}^n
    \nu_L (\bigvee S_i))}^\bullet = \bigcup_{i=1}^n \nu_L (\bigvee
  S_i)^\bullet = \bigcup_{i=1}^n \nu (S_i).\]

  Since $\pow\pt L$ is Boolean, $\nu$ is therefore a (centered)
  adherence structure on $\pow\pt L$.

  As a preparation to proving that $\lm_\nu = \lm_{\pow\pt L}$, we
  establish the following facts, where $\F$ is a filter of subsets of
  $\pt L$, and $S$ is a subset of $\pt L$:
  \begin{enumerate}
  \item[(A)] For every complemented element $a$ in
    ${(\kow\F)}^\#$, $a^\bullet$ is in $\F^\#$.  If
    $a^\bullet$ were not in $\F^\#$, then by
    Lemma~\ref{lemma:grill:compl} the complement of $a^\bullet$ would
    be in $\F$, that is, ${(\overline a)}^\bullet$ would be in $\F$,
    by Lemma~\ref{lemma:kow:basic:adh}~(4).  By definition, $\overline a$ would then be in $\kow\F$,
    contradicting the fact that $a \in {(\kow\F)}^\#$, in view of
    Lemma~\ref{lemma:grill:compl}.
  \item[(B)] If $S \in \F^\#$, then for every complemented element
    $a \geq \bigvee S$, $a$ is in ${(\kow\F)}^\#$.  Otherwise,
    $\overline a$ would be in $\kow \F$, namely
    ${(\overline a)}^\bullet$ would be in $\F$.  By
    Lemma~\ref{lemma:kow:basic:adh}~(4) again, the complement of
    $a^\bullet$ would be in $\F$, so $a^\bullet$ would not be in
    $\F^\#$.  We notice that $S$ is included in
    ${(\bigvee S)}^\bullet$, since every point of $S$ is by definition
    smaller than or equal to $\bigvee S$.  Using
    Lemma~\ref{lemma:kow:basic:adh}~(1),
    ${(\bigvee S)}^\bullet \subseteq a^\bullet$, so
    $S \subseteq a^\bullet$.  Since $\F^\#$ is upwards-closed
    (Lemma~\ref{lemma:grill}~(1)), it follows that $S$ cannot be in
    $\F^\#$, a contradiction.
  \end{enumerate}
  It follows that:
  \begin{enumerate}
  \item[(C)] For every filter $\F$ of subsets of $\pt L$,
    $\bigcap_{S \in \F^\#} \nu (S) = {(\bigwedge_{a \in {(\kow
          \F)}^\#} \nu_L (a))}^\bullet$.  Indeed, in one direction, for
    every $S \in \F^\#$,
    \[
    \nu_L (\bigvee S) = \bigwedge_{a\in\C_L, a \geq \bigvee S} \nu_L (a)\geq \bigwedge_{a \in {(\kow\F)}^\#} \nu_L (a)\] by (B).  Using
    Lemma~\ref{lemma:kow:basic:adh}~(1),
    $\nu (S) = \nu_L (\bigvee S)^\bullet \supseteq {(\bigwedge_{a \in
        {(\kow\F)}^\#} \nu_L (a))}^\bullet$.  In the other direction,
    for every complemented element $a \in {(\kow \F)}^\bullet$,
    $S := a^\bullet$ is in $\F^\#$ by (A).  Moreover,
    $\bigvee S \leq a$ since every point in $S$ is by definition less
    than or equal to $a$, so $\nu_L (\bigvee S) \leq \nu_L (a)$, and
    hence
    $\nu (S) = \nu_L (\bigvee S)^\bullet \subseteq \nu_L (a)^\bullet$
    by Lemma~\ref{lemma:kow:basic:adh}~(1).  It follows that
    $\nu_L (a)^\bullet \supseteq \bigcap_{S \in \F^\#} \nu (S)$.
    Taking infima over all complemented $a$, we obtain that:
    \[
    \bigcap_{S \in \F^\#} \nu (S) \subseteq \bigcap_{a \in {(\kow
        \F)}^\#} \nu_L (a)^\bullet = \left(\bigwedge_{a \in {(\kow
        \F)}^\#}
    \nu_L (a)\right)^\bullet,\] where the last equality is by
    Lemma~\ref{lemma:kow:basic:adh}~(3).
  \end{enumerate}
  For every filter $\F$ of subsets of $\pt L$,
  \[\lm_{\pow\pt L} \F = {(\lm_L {\kow \F})}^\bullet = {\left(\bigwedge_{a\in\C_L
      \cap {(\kow \F)}^\#} \nu_L (a)\right)}^\bullet.\]  By (C), this is equal to
  $\bigcap_{S \in \F^\#} \nu (S)$.  In other words, $\lm_{\pow\pt L}$
  is equal to $\lm_\nu$.

  Since $\pow\pt L$ is spatial and prime-continuous, by
  Remark~\ref{rem:adh:preconv:=}, we can apply
  Proposition~\ref{prop:adh:preconv:=}: $\nu$ is equal to
  $\adh_{\lm_{\pow\pt L}}$, which is by definition equal to
  $\nu_{\pow\pt L}$.
\end{proof}

\subsection{Completeness, cocompleteness}
\label{sec:compl-cocompl-1}

It is natural to compare adherence structures
$\nu, \nu' \colon L \to L$ on a $\catc$-object $L$ by saying that
$\nu$ is \emph{finer} than $\nu'$, in notation $\nu \leq \nu'$, or
that $\nu'$ is \emph{coarser} than $\nu$, if and only if
$\nu (\ell) \leq \nu' (\ell)$ for every $\ell \in L$.

The identity map from $(L, \nu)$ to $(L, \nu')$ is continuous (per
(\ref{eqn:nu:image})) if and only if $\nu' \leq \nu$, if and only if
$\nu$ is coarser than $\nu'$.  Let us write $|\_|$ for the forgetful
functor from $\Adh\catc$ to $\catc$ mapping $(L, \nu)$ to $L$.

\begin{prop}
  \label{prop:coarsest:adh}
  Let $\catc$ be a category of coframes, and $L$ be a $\catc$-object.
  Let also $\varphi_i \colon |L_i| \to L$ be morphisms of $\catc$,
  where each $L_i$ is an adherence $\catc$-object, $i \in I$.  There
  is a coarsest adherence structure $\nu_L$ on $L$ such that
  $\varphi_i$ is continuous for each $i \in I$.
\end{prop}
\begin{proof}
  By definition $\varphi_i$ is continuous if and only if
  $\nu_L (\ell) \leq \varphi_i (\nu_{L_i} (\varphi_{i!} (\ell)))$ for
  every $\ell \in L$.  In particular, if $\varphi_i$ is continuous for
  every $i \in I$, then for every $\ell \in L$, and for every finite
  family $a_1$, \ldots, $a_n$ of complemented elements such that
  $\ell \leq \bigvee_{j=1}^n a_j$,
  $\nu_L (\ell) \leq \bigvee_{j=1}^n \nu_L (a_j) \leq \bigvee_{j=1}^n
  \bigwedge_{i \in I} \varphi_i (\nu_{L_i} (\varphi_{i!} (a_j)))$.
  That means that $\nu_L (\ell)$ is less than or equal to:
  \begin{equation}
    \label{eq:coarsest:adh}
    \bigwedge_{\substack{a_1, \ldots, a_n\in\C_L\\
        \ell \leq \bigvee_{j=1}^n a_j}}
    \bigvee_{j=1}^n
    \bigwedge_{i \in I} \varphi_i (\nu_{L_i} (\varphi_{i!} (a_j))).
  \end{equation}
  In particular, if (\ref{eq:coarsest:adh}) defines an adherence
  structure on $L$ that makes every $\varphi_i$ continuous, then this
  is the desired coarsest adherence structure.

  Define $\nu_L (\ell)$ as (\ref{eq:coarsest:adh}).  Then $\nu_L$ is
  monotonic.  We check that $\nu_L$ preserves finite suprema of
  complemented elements.  It is enough to check this for the $0$-ary
  supremum, $\bot$, and for binary suprema.  By taking $n=1$ and
  $a_1 = \bot$ in (\ref{eq:coarsest:adh}), we obtain that
  $\nu_L (\bot) \leq \bigwedge_{i \in I} \varphi_i (\nu_{L_i}
  (\varphi_{i!}  (\bot)))$.  Since $\varphi_{i!}$ is a left-adjoint,
  it preserves all suprema, in particular the least element $\bot$.
  We know that $\nu_{L_i}$ preserves finite suprema of complemented
  elements, hence $\bot$, and $\varphi_i$ is a coframe morphism, so
  $\varphi_i (\nu_{L_i} (\varphi_{i!}  (\bot))) = \bot$ for every
  $i \in I$.  It follows that $\nu_L (\bot) \leq \bot$.  In the case
  of binary suprema of complemented elements $a$ and $a'$, it suffices
  to show that $\nu_L (a) \vee \nu_L (a') \geq \nu_L (a\vee a')$,
  since the converse inequality is implied by monotonicity.  We have
  the following, where $a_1$, \ldots, $a_m$ and $a'_1$, \ldots, $a'_n$
  range over finite families of complemented elements:
  \begin{eqnarray*}
    \nu_L (a) \vee \nu_L (a')
    & =
    & \mskip-20mu
      \bigwedge_{\substack{a_1, \ldots, a_m\\a\leq \bigvee_{j=1}^m
    a_j}} \bigvee_{j=1}^m \bigwedge_{i \in I} \varphi_i (\nu_{L_i}
    (\varphi_{i!} (a_j)))
    \vee \mskip-20mu
    \bigwedge_{\substack{a'_1, \ldots, a'_n\\a\leq \bigvee_{k=1}^n
    a'_k}} \bigvee_{k=1}^n \bigwedge_{i \in I} \varphi_i (\nu_{L_i}
    (\varphi_{i!} (a'_k)))
    \\
    & = &\mskip-20mu
      \bigwedge_{\substack{a_1, \ldots, a_m\\a'_1, \ldots, a'_n \\a\leq \bigvee_{j=1}^m
    a_j\\a' \leq \bigvee_{k=1}^n a'_k}}\mskip-20mu
    \left(
    \bigvee_{j=1}^m \bigwedge_{i \in I} \varphi_i (\nu_{L_i}
    (\varphi_{i!} (a_j)))
    \vee \bigvee_{k=1}^n \bigwedge_{i \in I} \varphi_i (\nu_{L_i}
    (\varphi_{i!} (a'_k)))
    \right)
  \end{eqnarray*}
  by the coframe distributivity law.  The conditions
  $a\leq \bigvee_{j=1}^m a_j$ and $a' \leq \bigvee_{k=1}^n a'_k$ imply
  $a \vee a' \leq \bigvee_{j=1}^m a_j \vee \bigvee_{k=1}^n a'_k$, so:
  \begin{eqnarray*}
    \nu_L (a) \vee \nu_L (a')
    & \geq
    & \mskip-40mu
      \bigwedge_{\substack{a_1, \ldots, a_m\\a'_1, \ldots, a'_n \\a \vee a' \leq \bigvee_{j=1}^m
    a_j \vee \bigvee_{k=1}^n a'_k}}\mskip-40mu
    \left(
    \bigvee_{j=1}^m \bigwedge_{i \in I} \varphi_i (\nu_{L_i}
    (\varphi_{i!} (a_j)))
    \vee \bigvee_{k=1}^n \bigwedge_{i \in I} \varphi_i (\nu_{L_i}
    (\varphi_{i!} (a'_k)))
    \right)\\
    & = & \nu_L (a \vee a').
  \end{eqnarray*}
  Next, we verify that
  $\nu_L (\ell) \geq \bigwedge_{a \geq \ell} \nu_L (a)$, where $a$
  ranges over complemented elements.  Equality will follow by
  monotonicity.  We merely observe that for every finite family of
  complemented elements $a_1$, \ldots, $a_n$ such that $\ell \leq
  \bigvee_{j=1}^n a_j$, there is a complemented element $a$ such that
  $\ell \leq a$ and $a \leq \bigvee_{j=1}^n a_j$, namely $a :=
  \bigvee_{j=1}^n a_j$.  It follows:
  \begin{eqnarray*}
    \nu_L (\ell)
    & =
    & \bigwedge_{\substack{a_1, \ldots, a_n\\
    \ell \leq \bigvee_{j=1}^n a_j}}
    \bigvee_{j=1}^n
    \bigwedge_{i \in I} \varphi_i (\nu_{L_i} (\varphi_{i!} (a_j)))
    \\
    & \geq
    & \bigwedge_{\substack{a \geq \ell\\a_1, \ldots, a_n\\
    a \leq \bigvee_{j=1}^n a_j}}
    \bigvee_{j=1}^n
    \bigwedge_{i \in I} \varphi_i (\nu_{L_i} (\varphi_{i!} (a_j)))
    = \bigwedge_{a \geq \ell} \nu_L (a)
  \end{eqnarray*}
  where again $a$, $a_1$, \ldots, $a_n$ range over complemented
  elements.

  Therefore $\nu_L$ is an adherence structure on $L$.  In order to
  show that each $\varphi_i$ is continuous, namely that
  $\nu_L (\ell) $ is less than or equal to
  $\varphi_i (\nu_{L_i} (\varphi_{i!} (\ell))$ for every $\ell \in L$
  and every $i \in I$, we proceed as follows.  Fix $\ell \in L$ and
  $i \in I$.  For every complemented $b \geq \varphi_{i!} (\ell)$,
  $a := \varphi_i (b)$ is complemented in $L$, since $\varphi_i$
  preserves binary suprema and binary infima.  Moreover, $a \geq \ell$
  since $\varphi_{i!}$ is left-adjoint to $\varphi_i$.  Therefore,
  still understanding $a$ and $b$ as complemented elements,
  \begin{eqnarray*}
    \bigwedge_{a \geq \ell} \nu_{L_i} (\varphi_{i!} (a))
    & \leq
    & \bigwedge_{b \geq \varphi_{i!} (\ell)} \nu_{L_i} (\varphi_{i!}
      (\varphi_i (b)))
    \\
    & \leq & \bigwedge_{b \geq \varphi_{i!} (\ell)} \nu_{L_i} (b)
    \quad\text{since }\varphi_{i!} (\varphi_i (b)) \leq b \\
    & = & \nu_{L_i} (\varphi_{i!} (\ell)).
  \end{eqnarray*}
  Applying $\varphi_i$ on both sides, and remembering that $\varphi_i$
  preserves arbitrary infima,
  \begin{eqnarray*}
    \bigwedge_{a \geq \ell} \varphi_i (\nu_{L_i} (\varphi_{i!} (a)))
    & \leq & \varphi_i (\nu_{L_i} (\varphi_{i!} (\ell))).
  \end{eqnarray*}
  Taking $n=1$ (and $a_1=a$) in (\ref{eq:coarsest:adh}) shows that
  $\nu_L (\ell)$ is smaller than or equal to the left-hand side of the
  latter inequality, hence also to the right-hand side.
\end{proof}

\begin{cor}
  \label{corl:coarsest:adh}
  Let $\catc$ be a category of coframes.  Then
  $|\_| \colon \Adh\catc \to \catc$ is topological: every sink
  ${(\varphi_i \colon |L_i| \to L)}_{i \in I}$ has a unique final
  lift, and this is $\nu_L$, as given in
  Proposition~\ref{prop:coarsest:adh}.
\end{cor}
\begin{proof}
  As noticed in the proof of Corollary \ref{corl:coarsest}, uniqueness follows from the dual of \cite[Proposition 10.43]{AHS:joycats}.
    To show existence, we check that
  $(L, \nu_L)$ of Proposition \ref{prop:coarsest} is a final lift of
  ${(\varphi_i \colon |L_i| \to L)}_{i \in I}$.  Let
  $\psi \colon L \to |L'|$ be such that $\psi \circ \varphi_i$ is
  continuous for every $i \in I$.  We aim to show that $\psi$ is
  continuous from $(L, \nu_L)$ to $L'$, and for that we consider an
  arbitrary element $\ell'$ of $L'$, and show that
  $\nu_{L'} (\ell') \leq \psi (\nu_L (\psi_! (\ell')))$.  Indeed,
  writing again $a_1$, \ldots, $a_n$ for complemented elements:
  \begin{eqnarray*}
    \psi (\nu_L (\psi_! (\ell')))
    & = & \psi \left(\bigwedge_{\substack{a_1, \ldots, a_n\\
    \psi_! (\ell') \leq \bigvee_{j=1}^n a_j}}
    \bigvee_{j=1}^n
    \bigwedge_{i \in I} \varphi_i (\nu_{L_i} (\varphi_{i!}
    (a_j)))\right)
    \quad\text{by (\ref{eq:coarsest:adh})} \\
    & = & \bigwedge_{\substack{a_1, \ldots, a_n\\
    \psi_! (\ell') \leq \bigvee_{j=1}^n a_j}}
    \bigvee_{j=1}^n
    \bigwedge_{i \in I} \psi (\varphi_i (\nu_{L_i} (\varphi_{i!}
    (a_j)))) \\
    && \quad\text{since $\psi$ is a morphism of coframes} \\
    & \geq & \bigwedge_{\substack{a_1, \ldots, a_n\\
    \psi_! (\ell') \leq \bigvee_{j=1}^n a_j}}
    \bigvee_{j=1}^n
    \bigwedge_{i \in I} \psi (\varphi_i (\nu_{L_i} (\varphi_{i!}
    (\psi_! (\psi (a_j)))))) \\
    && \quad\text{since $\psi_! (\psi (a_j)) \leq a_j$ by definition
       of left-adjoints} \\
    & \geq & \bigwedge_{\substack{a_1, \ldots, a_n\\
    \psi_! (\ell') \leq \bigvee_{j=1}^n a_j}}
    \bigvee_{j=1}^n
    \nu_{L'} (\psi (a_j)) \\
    && \quad\text{since $\psi \circ \varphi_i$ is continuous, and
       $\varphi_{i!} \circ \psi_! = (\psi \circ \varphi_i)_!$} \\
    & = & \bigwedge_{\substack{a_1, \ldots, a_n\\
    \ell' \leq \bigvee_{j=1}^n \psi (a_j)}}
    \bigvee_{j=1}^n
    \nu_{L'} (\psi (a_j))
  \end{eqnarray*}
  since $\psi_! (\ell') \leq \bigvee_{j=1}^n a_j$ is equivalent to
  $\ell' \leq \bigvee_{j=1}^n \psi (a_j)$, by the definition of a
  left-adjoint, and since $\psi$ preserves finite suprema.  The
  elements $b_j := \psi (a_j)$ are all complemented, so the latter
  infimum is
  larger than or equal to $\bigwedge_{\substack{b_1, \ldots, b_n\\
      \ell' \leq \bigvee_{j=1}^n b_j}} \bigvee_{j=1}^n
  \nu_{L'} (b_j)$.  That, in turn, is equal to $\bigwedge_{\substack{b_1, \ldots, b_n\\
      \ell' \leq \bigvee_{j=1}^n b_j}} \nu_{L'} (\bigvee_{j=1}^n
  b_j)$, which is larger than or equal to $\nu_{L'} (\ell')$.  Hence
  $\nu_{L'} (\ell') \leq \psi (\nu_L (\psi_! (\ell')))$, showing that
  $\psi$ is continuous.
\end{proof}

The usual consequences follow.
\begin{fact}
  \label{fact:cocompl:adh}
  The category $\Adh\coFrm$ is complete and cocomplete.
\end{fact}
\begin{fact}
  \label{fact:cowell:adh}
  The category $\Adh\coFrm$ is not co-wellpowered.
\end{fact}

\subsection{Closed elements}
\label{sec:clos-strongly-clos}

Fact~\ref{fact:adh:closed} shows that the notion of a closed element
of a convergence $\catc$-object only depends on the adherence
$\adh_L$.  Hence:
\begin{defn}[Closed]\label{defn:closed}
  Let $\catc$ be a category of coframes, and $L$ be an adherence
  $\catc$-object.  An element $c$ of $L$ is \emph{quasi-closed} if and
  only if $\nu_L (c) \leq c$.  A \emph{closed} element of $L$ is a
  complemented quasi-closed element of $L$.
\end{defn}

\begin{rem}
  \label{rem:openclosed}
  Let $L$ be a $\catc$-object, where $\catc$ is an admissible category
  of coframes.  The closed elements of $L$ can also be reconstructed
  as morphisms from $\pow (\Sierp)$ to $L$ in $\Adh\catc$, as we did
  for convergence lattices (Lemma~\ref{lemma:openclosed}).  For that,
  we equip $\pow (\Sierp)$ with the adherence operator
  $\nu_{\pow (\Sierp)} (A) := \dc A$, where $\dc$ is downward closure
  with respect to the ordering $0 \leq 1$.  One checks easily that the
  corresponding convergence structure $\lm_{\nu_{\pow (\Sierp)}}$
  coincides with the operator $\lm_{\pow (\Sierp)}$ of
  Section~\ref{sec:sierp-conv-cofr}.

  Given a morphism $\varphi \colon \pow (\Sierp) \to L$ in
  $\Adh\catc$, the elements $u := \varphi (\{1\})$ and
  $c := \varphi (\{0\})$ satisfy $u \wedge c = \bot$,
  $u \vee c = \top$, and $\nu_L (c) \leq c$.  The latter follows from
  the continuity condition (\ref{eqn:nu:image}):
  $\nu_L (c) \leq \varphi (\nu_{\pow (\Sierp)} (\varphi_! (c))$.
  Since $\varphi_! (c) = \varphi_! (\varphi (\{0\})) \subseteq \{0\}$
  by adjointness,
  $\nu_L (c) \leq \varphi (\dc \{0\}) = \varphi (\{0\}) = c$.

  Conversely, if $c$ is complemented and $\nu_L (c) \leq c$, let $u$
  be the complement of $c$, $\varphi$ be the unique lattice (hence
  coframe) morphism mapping $\{0\}$ to $c$ and $\{1\}$ to $u$.  We
  check that $\varphi$ is continuous, namely that for every
  $\ell \in L$,
  $\nu_L (\ell) \subseteq \varphi (\dc \varphi_! (\ell))$, as follows.
  If $\varphi_! (\ell) = \{0\}$, then
  $\varphi_! (\ell) \subseteq \{0\}$ hence
  $\ell \leq \varphi (\{0\}) = c$ by adjointness; in that case,
  $\nu_L (\ell) \leq \nu_L (c) \leq c = \varphi (\{0\}) = \varphi (\dc
  \varphi_! (\ell))$.  Otherwise,
  $\varphi (\dc \varphi_! (\ell)) = \varphi (\{0, 1\}) = \top \geq
  \nu_L (\ell)$.
\end{rem}

The inverse image of closed subsets by continuous maps are closed.
Analogously:
\begin{prop}
  \label{prop:closed:image}
  Let $\catc$ be a category of coframes.  For every morphism
  $\varphi \colon L \to L'$ in $\Adh\catc$, $\varphi$ maps
  quasi-closed elements of $L$ to quasi-closed elements of $L'$, and
  closed elements of $L$ to closed elements of $L'$.
\end{prop}
\begin{proof}
  Let $c$ be quasi-closed in $L$.  By continuity,
  $\nu_{L'} (\varphi (c)) \leq \varphi (\nu_L (\varphi_! (\varphi
  (c))))$.  By adjointness, $\varphi_! (\varphi (c)) \leq c$, so
  $\nu_{L'} (\varphi (c)) \leq \varphi (\nu_L (c))$, and that is less
  than or equal to $\varphi (c)$ since $c$ is quasi-closed.  If
  additionally $c$ is complemented, then $\varphi (c)$ is
  complemented, hence closed.
%
  % For every filter $\F$ on $L'$ such that $\varphi (c) \in \F^\#$, we
  % have that $c \in \varphi^{-1} (\F)^\#$: for every
  % $\ell \in \varphi^{-1} (\F)$,
  % $\varphi (c \wedge \ell) = \varphi (c) \wedge \varphi (\ell)$ is
  % different from $\bot$, since $\varphi (c) \in \F^\#$ and
  % $\varphi (\ell) \in \F$, and that implies that $c \wedge \ell$
  % cannot be equal to $\bot$.  Since $c$ is quasi-closed,
  % $\lm_L \varphi^{-1} (\F) \leq c$.  Using the continuity of
  % $\varphi$,
  % $\lm_{L'} \F \leq \varphi (\lm_L \varphi^{-1} (\F)) \leq \varphi
  % (c)$.  This shows that $\varphi (c)$ is quasi-closed.
%
  % If $u$ and $c$ are complements, then $\varphi (u)$ and $\varphi (c)$
  % are complements, too, since
  % $\varphi (u) \wedge \varphi (c) = \varphi (u \wedge c) = \varphi
  % (\bot) = \bot$, and similarly with $\vee$ and $\top$.  Therefore the
  % image of a closed element is closed.
%
  % Finally, if $\catc$ is a category of coframes, then $\varphi$
  % preserves infima.  Given a closed elements
  % $c = \bigvee_{i \in I} c_i$ where each $c_i$ is strongly closed,
  % $\varphi (c) = \bigvee_{i \in I} \varphi (c_i)$ is closed in $L'$.
\end{proof}

\begin{rem}
  \label{rem:closed:image}
  That the images of closed elements are closed is obvious (assuming
  $\catc$ admissible): in the light of Remark~\ref{rem:openclosed},
  that boils down to the fact that given a morphism
  $\varphi \colon L \to L'$ and a morphism
  $\psi \colon \pow (\mathbb S) \to L$ in $\Adh\catc$, $\varphi \circ \psi$ is
  a morphism from $\pow (\mathbb S)$ to $L'$.
\end{rem}

\begin{lem}
  \label{lemma:closed:pt}
  Let $\catc$ be a category of coframes, and $L$ be an adherence
  $\catc$-object.  For every quasi-closed element $c$ of $L$,
  $c^\bullet = \{x \in \pt L \mid x\leq c\}$ is a closed subset of
  $\pt L$.
\end{lem}
\begin{proof}
  We use Lemma~\ref{lemma:nu:adh}:
  $\nu_{\pow\pt L} (c^\bullet) = (\nu_L (\bigvee c^\bullet))^\bullet$.
  Since every element of $c^\bullet$ is less than or equal to $c$,
  $\bigvee c^\bullet \leq c$, so
  $\nu_L (\bigvee c^\bullet) \leq \nu_L (c) \leq c$.  It follows that
  $\nu_{\pow\pt L} (c^\bullet) \subseteq c^\bullet$, so that $c^\bullet$ is
  quasi-closed.  It is closed because every element of $\pow (\pt L)$
  is complemented.
\end{proof}

% \begin{rem}
%   \label{rem:closed:pt}
% We obtain a similar result on convergence $\catc$-objects $L$, under the
% weaker assumption that $\catc$ is a category of lattices.  Recall that
% in that case $c^\bullet = \{\varphi \in \pt L \mid \varphi (c)=1\}$.
% For every quasi-closed element $c$ of $L$,
%   $c^\bullet$ is a strongly closed subset of $\pt L$.

%   % Since $c$ is quasi-closed, $\lm_L {\kow\F} \leq c$, and therefore
%   % $\lm_{\pow (\pt L)} \F = {(\lm_L {\kow\F})}^\bullet \leq c^\bullet$
%   % by Lemma~\ref{lemma:kow:basic}~(1).
%   By Proposition~\ref{prop:closed:image}, since
%   $\epsilon_L \colon \ell \in L \mapsto \ell^\bullet$ is a morphism of
%   $\Cv\catc$ (Lemma~\ref{lemma:epsilon}).  Finally, every quasi-closed
%   element of $\pow (\pt L)$ is strongly closed, by Remark~\ref{rem:closed}.
% \end{proof}

A \emph{subcoframe} of a coframe $L$ is a subset of $L$ that is closed
under finite suprema and arbitrary infima.  A \emph{sublattice} is
only closed under finite suprema and finite infima.
\begin{prop}
  \label{prop:closed:subcoframe}
  Let $\catc$ be a category of coframes, and $L$ be an adherence
  $\catc$-object.
  \begin{enumerate}
  \item The quasi-closed elements form a subcoframe of $L$.
  \item The closed elements of $L$ form a sublattice of $L$.
  \end{enumerate}
\end{prop}
\begin{proof}
  (1) Let $c_i$ be quasi-closed, $i \in I$.  Then
  $\nu_L (\bigwedge_{i \in I} c_i) \leq \bigwedge_{i \in I} \nu_L
  (c_i) \leq \bigwedge_{i \in I} c_i$, where the first inequality is
  by the monotonicity of $\nu_L$, and the second one is because each
  $c_i$ is quasi-closed.

  Let now $c_1$, \ldots, $c_n$ be quasi-closed.  By
  Lemma~\ref{lemma:sup:compl}, $\nu (c_1 \vee \cdots \vee c_n) = \nu
  (c_1) \vee \cdots \vee \nu (c_n)$, and this is less than or equal to
  $c_1 \vee \cdots \vee c_n$ because each $c_i$ is quasi-closed.

  (2) follows from (1), using the fact that finite infima and finite
  suprema of complemented elements in a distributive lattice are
  complemented.
\end{proof}

\section{Topological Coframes}
\label{sec:topological-coframes}

The standard pointfree analogue of topological spaces are locales
\cite{Johnstone:Stone,framesandlocales}, i.e., the objects of the
opposite category of frames.  The idea is that all we need to know
about its topological spaces is its frame of open subsets, not its
points.  We might as well study coframes (of closed subsets).
Proposition~\ref{prop:closed:subcoframe} leads to another pointfree analogue
of topological spaces, where the closed elements are embedded in a
larger coframe, and only form a sublattice, not a subcoframe.

% A topological space can be defined, equivalently, as a set $X$
% together with a subframe of $\pow (X)$, consisting of the so-called
% open subsets of $X$, or as a \emph{closure space}, namely an adherence
% space whose adherence $\nu$ is idempotent.

% The standard pointless reading of the first view naturally leads to the
% fruitful theory of locales \cite{Johnstone:Stone,framesandlocales}.
% We may equivalently study the subcoframe of closed subsets.
% Our study of closed elements of adherence coframes naturally leads to
% the following different definition.
\begin{defn}[Topological coframe]
  \label{defn:Top}
  Let $\catc$ be a category of coframes.  A \emph{topological
    $\catc$-object} is an object $L$ of $\catc$ together with a
  sublattice $C(L)$ of complemented elements, called its \emph{closed
    elements}.

  The topological $\catc$-objects form a category $\Top\catc$, whose
  morphisms $\varphi \colon L \to L'$ are the $\catc$-morphisms that
  are \emph{continuous}, namely that map closed elements to closed
  elements: for every $c \in C (L)$, $\varphi (c) \in C (L')$.
%
  % The map $\varphi$ is \emph{final} if and only if every element
  % $c' \in C (L')$ is equal to $\varphi (c)$ for some $c \in C (L)$.
  % It is \emph{weakly final} if and only if $\varphi (\ell) \in C (L')$
  % implies that $\varphi (\ell) = \varphi (c)$ for some $c \in C (L)$.
\end{defn}
We shall call \emph{topological structure} on a coframe $L$ any
sublattice $C$ of complemented (in $L$) elements of $L$. 

% \begin{rem}
%   \label{rem:final:top}
%   Our notion of `final' is the categorical notion, as in
%   Remark~\ref{rem:final:conv}.  One can check this by hand, or refer
%   to Corollary \ref{corl:coarsest:top}.  Every final morphism is
%   weakly final, but the converse fails.  Weakly final morphisms will
%   be used in Lemma~\ref{lemma:top:epsilon}.  Note that every
%   surjective weakly final morphisms is final, so that the isomorphisms
%   of topological coframes are exactly the bijective, weakly final
%   morphisms.
% \end{rem}

\subsection{The adherence of a topological coframe}
\label{sec:adher-topol-cofr}

Lemma~\ref{prop:closed:subcoframe} shows that every adherence
$\catc$-object $(L, \nu_L)$ defines a topological $\catc$-object
$(L, C (L))$, where $C (L)$ is the sublattice of closed elements of
$(L, \nu_L)$.  Conversely:
\begin{lem}
  \label{lemma:top:nuC}
  Let $\catc$ be a category of coframes, and $L$ be a $\catc$-object.
  Every topological structure $C$ on $L$ defines an adherence
  structure $\nu_C$ on $L$ by:
  \begin{equation}
    \label{eq:nuC}
    \nu_C (\ell) = \bigwedge \{c \in C \mid c \geq \ell\}.
  \end{equation}
  Moreover:
  \begin{enumerate}
  \item $\nu_C$ is centered, that is, $\nu_C (\ell) \geq \ell$ for
    every $\ell \in L$;
  \item for every $\ell \in L$, for every $c \in C$, $c \geq \nu_C
    (\ell)$ if and only if $c \geq \ell$;
  \item for every $c \in C$, $\nu_C (c) = c$;
  \item $\nu_C$ is idempotent, that is, $\nu_C (\nu_C (\ell)) = \nu_C
    (\ell)$ for every $\ell \in L$.
  \end{enumerate}
\end{lem}
% idempotent?
\begin{proof}
  Clearly, $\nu_C$ is monotonic.  For finitely many elements $\ell_1$,
  \ldots, $\ell_n$,
  $\nu_C (\ell_1 \vee \cdots \vee \ell_n) \geq \nu_C (\ell_1) \vee
  \cdots \vee \nu_C (\ell_n)$ by monotonicity.  In the converse
  direction, and using the coframe distributivity law,
  $\nu_C (\ell_1) \vee \cdots \vee \nu_C (\ell_n)$ is the infimum of
  the elements of the form $c_1 \vee \cdots \vee c_n$, where
  $c_1, \ldots, c_n \in C$ and $c_1 \geq \ell_1$, \ldots,
  $c_n \geq \ell_n$.  For each such choice of elements $c_1$, \ldots,
  $c_n$, $c := c_1 \vee \cdots \vee c_n$ is again in $C$ and larger
  than or equal to $\ell_1 \vee \cdots \vee \ell_n$, so
  $\nu_C (\ell_1) \vee \cdots \vee \nu_C (\ell_n) \geq \bigwedge \{c
  \in C \mid c \geq \ell_1 \vee \cdots \vee \ell_n\} = \nu_C (\ell_1
  \vee \cdots \vee \ell_n)$.

  In order to show that $\nu_C$ is an adherence operator, it remains
  to show that $\nu_C (\ell) \geq \bigwedge_{a \geq \ell} \nu_C (a)$,
  where $a$ ranges over complemented elements.  This follows from the
  fact that for every $c \in C$ such that $c \geq \ell$, there is a
  complemented element $a$ such that $c \geq a$ and $a \geq \ell$,
  namely $c$ itself.

  (1) The fact that $\nu_C$ is centered is obvious.  (2) If
  $c \geq \nu_C (\ell)$, then $c \geq \ell$ by (1).  Conversely, if
  $c \in C$ is such that $c \geq \ell$, then by definition
  $c \geq \bigwedge \{c \in C \mid c \geq \ell\} = \nu_C (\ell)$.  (3)
  Taking $\ell:=c$ in (2), we obtain $c \geq \nu_C (c)$, and the
  converse inequality is by (1).  (4)
  $\nu_C (\nu_C (\ell)) = \bigwedge \{c \in C \mid c \geq \nu_C
  (\ell)\} = \bigwedge \{c \in C \mid c \geq \ell\}$ (by (2))
  $= \nu_C (\ell)$.
\end{proof}
\begin{cor}
The set $\bigwedge C$ of infima of elements of $C$ is the set of fixed points of $\nu_C$.
\end{cor}
\begin{proof}
If  $\ell=\nu_C(\ell)$ then by definition $\ell$ is an infimum of elements of $C$. Conversely,  if $\ell=\bigwedge_{i\in I}c_i$  where each $c_i\in C$, then 
\[\nu_C(\ell)=\nu_C\left(\bigwedge_{i\in I}c_i\right)\leq \bigwedge_{i\in I}\nu_C(c_i)=\bigwedge_{i\in I}c_i=\ell, \]
because of Lemma \ref{lemma:top:nuC} (3). Since $\nu_C$ is centered (Lemma \ref{lemma:top:nuC} (1)), $\nu_C(\ell)=\ell$.
\end{proof}
We compare topological structures by inclusion $\supseteq$, and we say
that $C$ is \emph{finer} than $C'$ if and only if $C \supseteq C'$.
In that case, we also say that $C'$ is \emph{coarser} than $C$.
\begin{lem}
  \label{lemma:top:adh}
  Let $\catc$ be a category of coframes, and $L$ be a $\catc$-object.
  For an adherence structure $\nu$ on $L$, let $C_\nu$ denote its
  lattice of closed elements.  Then:
  \begin{enumerate}
  \item The mapping $\nu \mapsto C_\nu$ is monotonic: if $\nu\leq \nu'$ then $C_{\nu'}\subset C_\nu$.
  \item The mapping $C \mapsto \nu_C$ is monotonic: if $C'\subset C$ then $\nu_C\leq \nu_{C'}$.
  \item For every topological structure $C$, $C_{\nu_C} \supseteq C$.
  \item For every adherence structure $\nu$, $\nu \leq \nu_{C_\nu}$.
  \end{enumerate}
\end{lem}
\begin{proof}
  (1) if $\nu \leq \nu'$, then every complemented element $c$ such
  that $\nu' (c) \leq c$ is such that $\nu (c) \leq c$.  (2) if $C$
  and $C'$ are two lattices of complemented elements and $C$ contains
  $C'$, then certainly
  $\bigwedge \{c \in C \mid c \geq \ell\} \leq \bigwedge \{c \in C'
  \mid c \geq \ell\}$.  (3) Given a lattice $C$ of complemented
  elements, $C_{\nu_C}$ is the set of complemented elements $a$ such
  that $\nu_C (a) \leq a$.  Every element $c$ of $C$ is complemented
  by definition, and $\nu_C (c) = c$ by Lemma~\ref{lemma:top:nuC}~(3),
  hence is in $C_{\nu_C}$.  (4) For every $\ell \in L$,
  $\nu_{C_\nu} (\ell) = \bigwedge \{c \in C_\nu \mid c \geq \ell\} =
  \bigwedge \{c \text{ complemented} \mid \nu (c) \leq c, c \geq
  \ell\}$.  For each complemented $c$ such that $\nu (c) \leq c$ and
  $c \geq \ell$, $\nu (c) \geq \nu (\ell)$ since $\nu$ is monotonic,
  so $c \geq \nu (\ell)$.  It follows that
  $\nu_{C_\nu} (\ell) \geq \nu (\ell)$.
\end{proof}
It follows that the $\nu \mapsto C_\nu$ construction is
(order-theoretically) left-adjoint to the $C \mapsto \nu_C$
construction.

\begin{prop}
  \label{prop:Top:univ}
  Let $\catc$ be an admissible category of coframes.  There is an
  identity-on-morphisms functor $(L, \nu_L) \mapsto (L, C_{\nu_L})$
  from $\Adh\catc$ to $\Top\catc$, which is right adjoint to the
  identity-on-morphisms functor $(L, C) \mapsto (L, \nu_C)$.
\end{prop}
\begin{proof}
  Write $U$ for the first functor, and $F$ for the second.  For every
  morphism $\varphi \colon L \to L'$ in $\Adh\catc$,
  $U (\varphi) = \varphi$ is a morphism in $\Top\catc$ by
  Proposition~\ref{prop:closed:image}.  Conversely, for every morphism
  $\varphi \colon L \to L'$ in $\Top\catc$, we check that
  $\nu_{C (L')} (\ell') \leq \varphi (\nu_{C (L)} (\varphi_!
  (\ell')))$ for every $\ell' \in L'$:
  \begin{eqnarray*}
    \varphi (\nu_{C (L)} (\varphi_! (\ell')))
    & = & \varphi \left(\bigwedge \{c \in C (L) \mid c \geq \varphi_!
          (\ell')\}\right) \\
    & = & \varphi \left(\bigwedge \{c \in C (L) \mid \varphi (c) \geq
          \ell'\}\right) \\
    & = & \bigwedge \{\varphi (c) \mid c \in C (L), \varphi (c) \geq \ell'\}
  \end{eqnarray*}
  since $\varphi$ preserves all infima.  Since $\varphi$ maps $C (L)$
  to $C (L')$, this is larger than or equal to $\bigwedge \{c' \in C
  (L') \mid c' \geq \ell'\} = \nu_{C (L')} (\ell')$.  Therefore both
  $U$ and $F$ are functors.

  By Lemma~\ref{lemma:top:adh}~(3), the identity map is continuous
  from $(L, C (L))$ to $(L, C_{\nu_{C (L)}}) = UF (L, C (L))$ in
  $\Top\catc$: this is the unit. By Lemma~\ref{lemma:top:adh}~(4), the
  identity map is continuous from
  $(L, \nu_{C_{\nu_L}}) = FU (L, \nu_L)$ to $(L, \nu_L)$: this is the
  counit.  The laws that they should satisfy hold because all
  considered maps are identities.
\end{proof}

\subsection{A square of adjunctions}
\label{sec:square-adjunctions-1}

The latter is the topmost adjunction in the following square, which is the
rightmost square in (\ref{eq:adj}):
\begin{equation}
  \label{eq:adj:top}
  \xymatrix{
    {(\Adh\catc)}^{op}
    \ar@<1ex>[r]\ar@{}[r]|{\perp}
    \ar@<1ex>[d]^{\pt}\ar@{}[d]|{\dashv}
    &
    {(\Top\catc)}^{op}
    \ar@<1ex>[l]
    \ar@<1ex>[d]^{\pt}\ar@{}[d]|{\dashv}
    \\
    \Adhc
    \ar@<1ex>[u]^{\pow}
    \ar@<1ex>[r]^{M_{\mathrm{top}}} \ar@{}[r]|{\perp}
    %\ar@{}[ru]|{\text{(Sec.~\ref{sec:adherence-coframes})}}
    &
    \Topc
    \ar@<1ex>[u]^{\pow}
    \ar@<1ex>[l]^{\supseteq}
  }
\end{equation}
The bottommost adjunction is formed as follows.  The inclusion functor
$\Topc \to \Adhc$ maps every topological space $X$ to the adherence
space $(X, cl)$, where $cl$ maps every subset of $X$ to its closure,
namely the smallest closed set containing it.  The \emph{topological
  modification} functor $M_{\mathrm{top}}$ maps every adherence space
$(X, \nu_{\pow (X)})$ to the topological space $X$, whose closed
subsets are defined as the fixed points of $\nu_{\pow (X)}$.

The rightmost adjunction, which we write $\pow\dashv\pt$ again, is
built as follows.  For every topological space $X$, $\pow (X)$ is the
coframe of all subsets of $X$, and $C (\pow (X))$ is the sublattice of
closed subsets of $X$.  For every continuous map $f \colon X \to Y$,
$\pow (f) = f^{-1}$.  In the converse direction, we define:
\begin{defn}[Point]
  \label{defn:top:point}
  Let $\catc$ be an admissible category of coframes.  For every
  topological $\catc$-object $L$, the \emph{points} of $L$ are its
  join-prime elements.  Let $\pt L$ be the set of points of $L$.

  For each $\ell \in L$, let
  $\ell^\bullet := \{x \in \pt L \mid x \leq \ell\}$, and let
  $C (\pow\pt L)$ be the set of elements of the form $c^\bullet$,
  where $c$ is an infimum of elements of $C (L)$.
\end{defn}
Note that $c$ is not taken directly from $C (L)$ in the definition of
the set $C (\pow \pt L)$ of closed subsets of $\pt L$.  Let us write
$\bigwedge C (L)$ for the set of infima of elements of $C (L)$.
\begin{lem}
  \label{lemma:infclosed}
    Let $\catc$ be an admissible category of coframes, and $L$ be a
    topological $\catc$-object.  The set $\bigwedge C (L)$ of infima
    of elements of $C (L)$ is a subcoframe of $L$.
\end{lem}
\begin{proof}
  Clearly $\bigwedge C (L)$ is closed under arbitrary infima.  Since
  $C (L) \subseteq \bigwedge C (L)$ and $\bot \in C (L)$, the 0-ary
  supremum $\bot$ is in $\bigwedge C (L)$.  It remains to check that
  given two elements $c := \bigwedge_{i \in I} c_i$ and
  $c' := \bigwedge_{j \in J} c'_j$ where $c_i, c'_j \in C (L)$,
  $c \vee c'$ is in $\bigwedge C (L)$.  This follows from the coframe
  distributivity law:
  $c \vee c' = \bigwedge_{i \in I, j \in J} (c_i \vee c'_j)$, noticing
  that each $c_i \vee c'_j$ is in $C (L)$.
\end{proof}

\begin{lem}
  \label{lemma:top:point}
  Let $\catc$ be an admissible category of coframes, and $L$ be a
  topological $\catc$-object $L$.  Then:
  \begin{enumerate}
  \item For every family ${(\ell_i)}_{i \in I}$ in $L$,
    $\bigcap_{i \in I} \ell_i^\bullet = {(\bigwedge_{i \in I}
      \ell_i)}^\bullet$.
  \item For every finite family ${(\ell_i)}_{i=1}^n$ in $L$,
    $\bigcup_{i \in I} \ell_i^\bullet = {(\bigvee_{i \in I}
      \ell_i)}^\bullet$.
  \item $C (\pow\pt L)$ is closed under finite unions and arbitrary
    intersections.
  \item The elements of $C (\pow\pt L)$ define the closed sets of a
    topology on $\pt L$.
  \end{enumerate}
\end{lem}
\begin{proof}
  (1)
  $\bigcap_{i \in I} \ell_i^\bullet = \{x \in \pt L \mid \forall i \in
  I, x \leq \ell_i\} = \{x \in \pt L \mid x \leq \bigwedge_{i \in I}
  \ell_i\} = {(\bigwedge_{i \in I} \ell_i)}^\bullet$.  (2) The
  elements $x$ of ${(\bigvee_{i=1}^n \ell_i)}^\bullet$ are those
  join-primes such that $x \leq \bigvee_{i=1}^n \ell_i$, that is, such
  that $x \leq \ell_i$ for some $\ell_i$.  Those are exactly the
  elements of $\bigcup_{i=1}^n \ell_i^\bullet$.  (3) Let
  ${(c_i)}_{i \in I}$ be a family of elements of $\bigwedge C (L)$.
  We have
  $\bigcap_{i \in I} c_i^\bullet = {(\bigwedge_{i \in I}
    c_i)}^\bullet$ by (1), and clearly $\bigwedge_{i \in I} c_i$ is in
  $\bigwedge C (L)$.  Given finitely many elements $c_1$, \ldots,
  $c_n$ in $\bigwedge C (L)$,
  $\bigcup_{i=1}^n c_i^\bullet = {(\bigvee_{i=1}^n c_i)}^\bullet$ by
  (2) and $\bigvee_{i=1}^n c_i$ is in $\bigwedge C (L)$ by
  Lemma~\ref{lemma:infclosed}.  (4) follows directly from (3).
\end{proof}
Therefore $\pt L$ defines a topological space.  It satisfies the
following universal property.
\begin{prop}
  \label{prop:top:univ}
  Let $\catc$ be an admissible category of coframes.  For every
  topological space $X$, for every topological $\catc$-object $L$, and
  every morphism $\varphi \colon L \to \pow (X)$ in $\Top\catc$, there
  is a unique map $\varphi^\dagger \colon X \to \pt L$ such that, for
  every $\ell \in L$,
  $\pow (\varphi^\dagger) (\ell^\bullet) = \varphi (\ell)$, and it is
  continuous.
\end{prop}
\begin{proof}
  If $\varphi^\dagger$ exists, then for every $\ell \in L$, $\varphi
  (\ell) = \pow (\varphi^\dagger) (\ell^\bullet) = \{x \in X \mid
  \varphi^\dagger (x) \leq \ell\}$.  For each $x$, this forces
  $\varphi^\dagger (x)$ to be the least element $\ell \in L$ such that
  $x \in \varphi (\ell)$.  Since $x \in \varphi (\ell)$ is equivalent
  to $\{x\} \subseteq \varphi (\ell)$, hence to $\varphi_! (\{x\})
  \leq \ell$, this forces $\varphi^\dagger (x)$ to be equal to
  $\varphi_! (\{x\})$.

  Hence define $\varphi^\dagger (x)$ as $\varphi_! (\{x\})$.  We check
  that this is join-prime: if
  $\varphi_! (\{x\}) \leq\bigvee_{i=1}^n \ell_i$, then
  $\{x\} \subseteq \varphi (\bigvee_{i=1}^n \ell_i) = \bigvee_{i=1}^n
  \varphi (\ell_i)$, so $\{x\} \subseteq \varphi (\ell_i)$ for some
  $i$, from which $\varphi_!  (\{x\}) \leq \ell_i$.

  It remains to show that $\varphi^\dagger$ is continuous.  Consider
  an arbitrary closed subset $c^\bullet$ of $\pt L$, where $c$ is an
  infimum $\bigwedge_{i \in I} c_i$ of elements of $C (L)$.  Then
  $(\varphi^\dagger)^{-1} (c^\bullet) = \{x \in X \mid \varphi_!
  (\{x\}) \leq c\} = \{x \in X \mid \{x\} \subseteq \varphi (c)\} =
  \varphi (c) = \bigwedge_{i \in I} \varphi (c_i)$.  Since
  $\varphi (c_i) \in C (\pow (X))$ by continuity, this is a closed
  subset of $X$.
\end{proof}

\begin{lem}
  \label{lemma:top:epsilon}
  Let $\catc$ be an admissible category of coframes, and $L$ be a
  topological $\catc$-object.  The map
  $\epsilon_L \colon L \mapsto \pow (\pt L)$ defined by
  $\epsilon_L(\ell)=\ell^\bullet$ is a %weakly final
  morphism in $\Top\catc$, which is injective if and only if $L$ is
  spatial.
\end{lem}
\begin{proof}
  First, $\epsilon_L$ is a coframe morphism by
  Lemma~\ref{lemma:top:point}~(1) and~(2).  Second, $\epsilon_L$ maps
  each $c \in C (L)$ to $c^\bullet$, which is a closed element of
  $\pow (\pt L)$ by definition, so $\epsilon_L$ is continuous.
  % Conversely, if $c^\bullet$ is closed, then by definition
  % $c^\bullet = {c'}^\bullet$ for some closed element $c'$ of $L$, so
  % $\epsilon_L$ is weakly final.

  Finally, let $\ell, \ell' \in L$ be such that
  $\ell^\bullet = {\ell'}^\bullet$.  The join-prime elements below $\ell$ and
  below $\ell'$ are the same.  If $L$ is spatial (in the sense of Definition \ref{def:spatialcoframe}), each element is the
  supremum of join-prime elements below it, and that implies $\ell=\ell'$.
  Conversely, if $L$ is not spatial, then for some $\ell \in L$,
  $\ell$ is not the supremum of the set $\ell^\bullet$ of join-primes
  below it.  Let $\ell' = \bigvee \ell^\bullet$, so that
  $\ell' \neq \ell$.  We check that $\ell' \leq \ell$, since $\ell$ is
  an upper bound of $\ell^\bullet$ and $\ell'$ is the least one.  It
  follows that ${\ell'}^\bullet \subseteq \ell^\bullet$.  Conversely,
  every point in $\ell^\bullet$ is below
  $\bigvee \ell^\bullet = \ell'$, hence in ${\ell'}^\bullet$, so
  $\ell^\bullet = {\ell'}^\bullet$.
\end{proof}

It follows that $\pt \dashv \pow$ defines an adjunction between
$\Topc^{op}$ and $\Top\catc$, with unit $\epsilon_L$, hence by taking
opposite categories:
\begin{thm}
  \label{thm:top:A-|pt}
  Let $\catc$ be an admissible category of coframes.  Then
  $\pow\dashv\pt$ is an adjunction between $\Topc$ and
  ${(\Top\catc)}^{op}$.
\end{thm}
The counit is $\epsilon_L$, and the unit
$\eta_X \colon X \to \pt \pow (X)$ is equal to
$\identity {\pow (X)}^\dagger \colon x \mapsto {\identity {\pow
    (X)}}_! (x) = \{x\}$.
\begin{lem}
  \label{lemma:top:eta}
  Let $\catc$ be an admissible category of coframes.  For every
  topological space $X$, $\eta_X$ is an isomorphism.
\end{lem}
\begin{proof}
  Clearly $\eta_X$ is bijective, and continuous by construction.
  Explicitly, consider any closed subset of $\pt \pow (X)$.  This is a
  subset of the form $c^\bullet$, where $c$ is an infimum of closed
  elements of $\pow (X)$, i.e., of closed subsets of $X$.  In
  particular, $c$ is itself a closed subset of $X$.  Then
  $\eta_X^{-1} (c^\bullet) = \{x \in X \mid \{x\} \in c^\bullet\} =
  \{x \in X \mid \{x\} \subseteq c\} = c$.  The inverse of $c$ by the
  inverse map $\eta_X^{-1}$ is then $c^\bullet$, which is closed.
  Therefore both $\eta_X$ and $\eta_X^{-1}$ are continuous.
\end{proof}

\begin{cor}
  \label{corl:top:corefl}
  Let $\catc$ be an admissible category of coframes.  $\Topc$ is a
  coreflective subcategory of ${(\Top\catc)}^{op}$, through the
  coreflection $\pt$.
\end{cor}
Moreover, the second part of Lemma~\ref{lemma:top:epsilon} shows that
the objects that are isomorphic to $\pow (X)$ for some topological
space $X$ in ${(\Top\catc)}^{op}$ are exactly the topological
$\catc$-objects $(L, C (L))$ such that $L$ is spatial.

This finishes our description of the final side of
(\ref{eq:adj:adh}).  Clearly, the left-adjoints commute, hence also
the right-adjoints.

Note that the $\pow\dashv\pt$ coreflection embeds the whole of $\Topc$
inside ${(\Top\catc)}^{op}$, not just the subcategory of sober spaces,
as the familiar adjunction between topological spaces and locales
would do.  We examine the relation between topological coframes and
locales in Section~\ref{sec:topol-cofr-local}, and quickly examine
(co)completeness questions.

\subsection{Completeness, cocompleteness}
\label{sec:compl-cocompl-2}

We proceed along familiar lines.  Let now
$|\_| \colon \Top\catc \to \catc$ be the functor that maps every
topological $\catc$-object $(L, C (L))$ to the underlying
$\catc$-object $L$.
\begin{prop}
  \label{prop:coarsest:top}
  Let $\catc$ be a category of coframes, and $L$ be a $\catc$-object.
  Let also $\varphi_i \colon |L_i| \to L$ be morphisms of $\catc$,
  where each $L_i$ is a topological $\catc$-object, $i \in I$.  There
  is a coarsest topological structure $C (L)$ on $L$ such that
  $\varphi_i$ is continuous for each $i \in I$.
\end{prop}
\begin{proof}
  This must be the smallest sublattice of $L$ that contains all the
  elements $\varphi_i (c)$, $i \in I$, $c \in C (L_i)$.  Those are
  exactly the finite suprema of finite infima of such elements, and
  they are all complemented because $\varphi_i$, being a morphism of
  lattices, maps complemented elements to complemented elements.
\end{proof}

\begin{cor}
  \label{corl:coarsest:top}
  Let $\catc$ be a category of coframes.  Then
  $|\_| \colon \Top\catc \to \catc$ is topological: every sink
  ${(\varphi_i \colon |L_i| \to L)}_{i \in I}$ has a unique final
  lift, and this is the coarsest topological structure $C (L)$ given
  in Proposition~\ref{prop:coarsest:top}.
\end{cor}
\begin{proof}
  As for Corollary~\ref{corl:coarsest} and
  Corollary~\ref{corl:coarsest:adh} uniqueness is a general categorical fact.

  To show existence, we check that
  $(L, C (L))$ is a final lift of
  ${(\varphi_i \colon |L_i| \to L)}_{i \in I}$.  Let
  $\psi \colon L \to |L'|$ be such that $\psi \circ \varphi_i$ is
  continuous for every $i \in I$.  We aim to show that $\psi$ is
  continuous from $(L, C (L))$ to $L'$, and for that we consider an
  arbitrary element $c \in C (L)$.  This $c$ can be written as
  $\bigvee_{j=1}^m \bigwedge_{k=1}^{n_i} c_{jk}$, where each $c_{jk}$
  is of the form $\varphi_i (c_i)$, for some $i \in I$ and
  $c_i \in C (L_i)$.  Since $\psi \circ \varphi_i$ is continuous,
  $\psi (c_{jk}) = \psi (\varphi_i (c_i))$ is an element of $C (L')$.
  It follows that
  $\psi (c) = \bigvee_{j=1}^m \bigwedge_{k=1}^{n_i} \psi (c_{jk})$ is
  also in $C (L')$.
\end{proof}
The usual consequences follow.
\begin{fact}
  \label{fact:cocompl:top}
  The category $\Top\coFrm$ is complete and cocomplete.
\end{fact}
\begin{fact}
  \label{fact:cowell:top}
  The category $\Top\coFrm$ is not co-wellpowered.
\end{fact}
%%%%%%%%%%%%
%%%%%%ADDED
\subsection{Topological convergence coframes and topological modification}
Every convergence coframe $L$ defines a sublattice $C(L)$ of closed elements (via Definition \ref{defn:closed}), which in turn determines a topological coframe in the sense of Definition \ref{defn:Top}. What are the convergence coframes whose lattice of closed elements determines the convergence?

With a convergence coframe $(L,\lim_L)$, we associate another convergence structure on $L$ defined by 

\begin{equation}\label{eq:TL}
\lm_{\T(L)}\F=\bigwedge_{c\in\F^{\#}\cap C(L)}c.
\end{equation}

That \eqref{eq:TL} defines a convergence lattice structure is clear. Moreover, the two convergence structures share the same closed elements, that is,
\begin{equation}\label{eq:sameclosed}
C(\T(L))=C(L).
\end{equation} 
\begin{proof}
 Indeed, 
\begin{equation}\label{eq:TLcoarser}
\lm_{L}\leq \lm_{\T(L)}
\end{equation}
 because $\lim_L\F\leq c$ whenever $c\in\F\cap C(L)$, so that every $\T(L)$-closed set is also $L$-closed. Conversely, if $c\in C(L)$ and $c\in\F^{\#}$ then $\lim_{\T(L)}\F \leq c$ by definition, that is, $c$ is quasi-closed for $\T(L)$. As $c$ is also complemented, it is $\T(L)$-closed.
\end{proof}
We call a convergence coframe \emph{topological} if $\lim_L=\lim_{\T(L)}$. Clearly, topological convergence structures are determined by their lattice of closed elements.
\begin{prop}\label{prop:topmodif}
 If $(L,\lim_L)$ is a convergence coframe then $\lim_{\T(L)}$ is the finest topological convergence structure on $L$ that is coarser than $\lim_L$. 
\end{prop}
Thus we call $\T(L)$ the \emph{topological modification of } $L$.
\begin{proof}
In view of \eqref{eq:sameclosed}, $\lim_{\T(L)}$ is topological, and coarser than $\lim_L$ by \eqref{eq:TLcoarser}. If $\lim$ is another topological convergence structure on $L$ coarser than $\lim_L$, then $\lim$-closed elements are $\lim_L$-closed. As a result,
\[\lim\F=\bigwedge_{c\in\F^{\#}\cap C(\lim)}\F \geq \bigwedge_{c\in\F^{\#}\cap C(L)}\F=\lm_{\T(L)}\F.\]
\end{proof}
%%%%
\begin{lem}\label{lem:contclosed}
	If $L$ and $L'$ are two topological convergence coframes, then $\varphi:L\to L'$ is continuous if and only if $\varphi(c)$ is closed in $L'$ whenever $c$ is closed in $L$, that is, if and only if
	\[\varphi(C(L))\subset C(L').\]
\end{lem}
\begin{proof}
	Corollary \ref{cor:imageofclosed} states that $\varphi(C(L))\subset C(L')$ whenever $\varphi$ is continuous. Assume conversely, that $\varphi(C(L))\subset C(L')$.Because $L$ is topological and $\varphi$  is a morphism of coframes,
	\[\varphi\left(\lm_L\varphi^{-1}\F\right)=\varphi\left(\bigwedge_{c\in\varphi^{-1}\F^{\#}\cap C(L)}c\right)=\bigwedge_{c\in\varphi^{-1}\F^{\#}\cap C(L)}\varphi(c).\] 
	Moreover, by assumption, $\varphi(c)\in C(L')$, and $c\in\varphi^{-1}\F^{\#}$ if and only if $\varphi(c)\in \F^\#$ by Corollary \ref{cor:imagepseudocompl}. Thus
	\[\varphi\left(\lm_L\varphi^{-1}\F\right)\geq \bigwedge_{d\in\F^{\#}\cap C(L')}d=\lm_{L'}\F,
	\]
	because $L'$ is topological. Thus $\varphi$ is continuous.
\end{proof}
\begin{cor}
Let $\mathbf C$ be a category of coframes. The full subcategory $\mathbf{C}^{\mathrm{conv_T}}$ of $\mathbf{C}^{\mathrm{conv}}$ formed by topological convergence coframes is a reflective subcategory of $\mathbf{C}^{\mathrm{conv}}$. The reflector $\T$ acts on objects as $L\to\T(L)$ and acts as identity on morphisms.
\end{cor}
\begin{proof}
In view of Proposition \ref{prop:topmodif}, we only need to show that $\T$ is a concrete functor, that is, if $\varphi:L \to L'$ is continuous, then $\T\varphi=\varphi:\T(L)\to \T(L')$ is also continuous. In view of Lemma \ref{lem:contclosed} and \eqref{eq:sameclosed}, it is enough to show that $\varphi(C(L))\subset C(L')$, and this follows from  Corollary  \ref{cor:imageofclosed}.
\end{proof}
In view of Lemma \ref{lem:contclosed}, \[C:\mathbf{C}^{\mathrm{conv}_{\T}}\to \mathbf{C}^{\mathrm{top}}\]
that associates  with each $\mathbf C^{\mathrm{conv}_{\T}}$-object $L$ its lattice $C(L)$ of closed elements and acts as identity on morphisms is a full and faithful functor. Moreover, it is bijective on objects, because topological convergence coframes are determined by their lattice of closed elements.   Therefore,
\begin{thm}
	 The categories $:\mathbf{C}^{\mathrm{conv}_{\T}}$ and $\mathbf{C}^{\mathrm{top}}$ are isomorphic.
\end{thm}

On the other hand, \[\operatorname{Lim}:\mathbf{C}^{\mathrm{top}}\to\mathbf{C}^{\mathrm{conv}_{\T}}\]
  that associates  with each $\mathbf{C}^{\mathrm{top}}$-object $(L,C)$ the topological convergence $\mathbf{C}$-object $(L,\lim_C)$ defined by
\[
\lm_C\F:=\bigwedge_{c\in\F^{\#}\cap C} c,
\]
and acts as identity on morphisms is also an isomorphism of category (consider  Lemma \ref{lem:contclosed}). Moreover,
\begin{equation}
\operatorname{Lim}\circ C=Id_{\mathbf{C}^{\mathrm{conv}_{\T}}}\text{ and } C\circ \operatorname{Lim}=Id_{\mathbf{C}^{\mathrm{top}}}.
\end{equation}

%%%%%%%%%%%%%
\subsection{Topological coframes and locales}\label{sec:topol-cofr-local}

One would expect at this point that ${(\Top\coFrm)}^{op}$ and the
category $\Frm^{op}$ of locales to be strongly related, since both are
connected to $\Topc$ by an adjunction.  The fact that topological
spaces embed faithfully in the former but not in the latter is an
indication that the two pointfree categories differ.  They do not seem
to be related by an adjunction either, but at least, there are functors between the two.

In one direction, every topological coframe $L$ has a subcoframe
$\bigwedge C (L)$ by Lemma~\ref{lemma:infclosed}.  This defines a
frame ${(\bigwedge C (L))}^{op}$.  For every morphism
$\varphi \colon L \to L'$ in $\Top\coFrm$, the restriction of
$\varphi$ to $\bigwedge C (L)$ defines a coframe morphism from
$\bigwedge C (L)$ to $\bigwedge C (L')$.  Hence:
\begin{fact}
  \label{fact:wedgeC}
  Let $\catc$ be a category of coframes.  There is a functor $\bigwedge C^{op}
  \colon \Top\coFrm \to \Frm$ which maps each object $L$ to
  ${(\bigwedge C (L))}^{op}$ and acts on morphisms by restriction.  \qed
\end{fact}

In the other direction, given any frame $\Omega$, we can form the
coframe of sublocales $\Sl (\Omega)$
\cite[Section~III.3]{framesandlocales}.  For each $u \in \Omega$, the
\emph{closed sublocale} $\clos (u) := \upc u$ is a complemented
element of $\Sl (\Omega)$, and its complement is the \emph{open
  sublocale} $\ouv (u) := \{u \limp v \mid v \in \Omega\}$, where
$\limp$ is implication in the Heyting algebra $\Omega$
(Proposition~III.6.1.3, loc.cit.).  Moreover, the map
$\clos \colon \Omega^{op} \to \Sl (\Omega)$ is an order-embedding
(Proposition~III.6.1.4, loc.cit.\@) and a coframe morphism
(Proposition~III.6.1.5, loc.cit.).  $\Sl (\Omega)$ then defines a
topological coframe, provided we define $C (\Sl (\Omega))$ as its
subcoframe of elements of the form $\clos (u)$, $u \in \Omega$.

We will use the following universal property of
$\Sl (\Omega)$ below (Proposition~III.6.3.1, loc.cit., slightly
reformulated): for every frame morphism
$\varphi \colon \Omega \to \mathcal{C}_{\Omega'}$, there is a unique
frame morphism $\varphi^* \colon {(\Sl (\Omega))}^{op} \to \Omega'$
such that $\varphi^* (\clos (u)) = \varphi (u)$ for every
$u \in \Omega$.

\begin{lem}
  \label{lemma:Sl}
  There is a functor $\Sl \colon \Frm \to \Top\coFrm$ which maps every
  frame $\Omega$ to $\Sl (\Omega)$, and every frame morphism
  $\varphi \colon \Omega \to \Omega'$ to the unique coframe morphism
  $\Sl (\varphi) := {(\clos \circ \varphi)}^* \colon \Sl (\Omega) \to
  \Sl (\Omega')$ that maps $\clos (u)$ to $\clos (\varphi (u))$ for
  every $u \in \Omega$.
\end{lem}
\begin{proof}
  We check types first.  Since
  $\clos \colon {\Omega'}^{op} \to \Sl (\Omega')$ is a coframe
  morphism, it is also a frame morphism from $\Omega'$ to
  $\Sl (\Omega')^{op}$.  Then ${(\clos \circ \varphi)}^*$ is
  well-defined, since $\clos \circ \varphi$ sends every element to a
  complemented element of $\Sl (\Omega')^{op}$, and is a frame
  morphism from $\Sl (\Omega)^{op}$ to $\Sl (\Omega')^{op}$, hence a
  coframe morphism from $\Sl (\Omega)$ to $\Sl (\Omega')$.

  We now need to check that ${(\clos \circ \varphi)}^*$ is continuous:
  it maps every closed element $\clos (u)$ of $\Sl (\Omega)$ to $\clos
  (\varphi (u))$, which is closed by definition.
\end{proof}

\begin{lem}
  \label{lemma:Fr:Sl}
  \sloppy For every frame $\Omega$, there is an isomorphism between $\Omega$
  and $\bigwedge C^{op} (\Sl (\Omega))$, which maps each
  $u \in \Omega$ to $\clos (u)$, and it is natural in $\Omega$.
\end{lem}
\begin{proof}
  The closed elements of $\Sl (\Omega)$ are exactly the elements
  $\clos (u)$, $u \in \Omega$.  Those elements are closed under
  arbitrary infima, since $\clos$ is a coframe morphism, hence
  $\bigwedge C^{op} (\Sl (\Omega)) = C (\Sl (\Omega))^{op}$.  Since
  $\clos$ is an order-embedding, this shows the isomorphism part.

  For every frame morphism $\varphi \colon \Omega \to \Omega'$,
  naturality is the fact that $\clos (\varphi (u)) = \Sl
  (\varphi) (\clos (u))$ for every $u \in \Omega$, and that is the
  property we stated of $\Sl (\varphi)$ in Lemma~\ref{lemma:Sl}.
\end{proof}
In other words, the category $\Frm^{op}$ of locales arises as a
retract (up to equivalence of categories) of ${(\Top\coFrm)}^{op}$.

One can characterize a natural full subcategory of
${(\Top\coFrm)}^{op}$ in which the category $\Frm^{op}$ of locales
will embed reflectively, as follows.
\begin{defn}
  \label{defn:strong:top}
  A \emph{strong} topological coframe is a topological coframe $L$
  such that $C (L)$ is closed under infima taken in $L$.
\end{defn}

\begin{fact}
  \label{fact:Sl:strong}
  For every frame $\Omega$, $\Sl (\Omega)$ is a strong topological
  coframe.  \qed
\end{fact}

\begin{prop}
  \label{prop:Sl:Fr}
  There is a coreflection $\Sl \dashv \bigwedge C^{op}$ between $\Frm$
  and the category of strong topological coframes.
\end{prop}
\begin{proof}
  The unit is the isomorphism $\clos$ of Lemma~\ref{lemma:Fr:Sl}.  We
  claim that the counit is
  $\identity {C (L)}^* \colon \Sl (\bigwedge C^{op} (L)) \to L$.  We
  reason as follows.  Since $L$ is strongly topological, the identity
  map $\identity {C (L)} \colon \bigwedge C (L)^{op} \to C (L)^{op}$
  makes sense.  Then $\identity {C (L)}^*$ is a coframe morphism, and
  $\identity {C (L)}^* (\clos (u)) = u$ for every $u \in L$.

  Let us check that this is natural in $L$.  For every morphism
  $\varphi \colon L \to L'$ in $\Top\coFrm$,
  $\identity {C (L)}^* \circ (\Sl (\bigwedge C^{op} (\varphi)))$ is
  the unique map that maps $\clos (u)$ to
  $\identity {C (L)}^* (\clos (\varphi (u))) = \varphi (u)$.  Hence it
  coincides with $\varphi \circ \identity {C (L)}^*$.

  It remains to check that the following compositions:
  \[
    \xymatrix@C=70pt{
      \Sl (\Omega) \ar[r]^{\Sl (\clos)}
      & \Sl (\bigwedge C^{op} (\Sl (\Omega))) \ar[r]^{\identity {C (\Sl
          (\Omega))}^*}
      & \Sl (\Omega) \\
      \bigwedge C^{op} (L) \ar[r]^(0.4){\clos}
      & \bigwedge C^{op} (\Sl (\bigwedge C^{op} (L)))
      \ar[r]^(0.6){\bigwedge C^{op} (\identity {C (L)}^*)}
      & \bigwedge C^{op} (L)
    }
  \]
  are identities.  The first one is the unique morphism that maps
  $\clos (u)$, for each $u \in \Omega$, to $\identity {C (\Sl
    (\Omega))}^* (\clos (\clos (u))) = \clos (u)$, hence indeed
  coincides with the identity morphism.  The second one maps every
  closed element $c$ of $L$ to $\bigwedge C^{op} (\identity {C (L)}^*)
  (\clos (c))$, which is equal to $\identity {C (L)}^* (\clos (c))$ since $\bigwedge
  C^{op}$ acts by restriction, and the latter is equal to $c$.
\end{proof}

\begin{cor}
  \label{corl:Fr:Sl}
  The category of locales is a reflective subcategory of the opposite
  of the category of strong topological coframes.
\end{cor}

% \section{Perspectives}
% \label{sec:perspectives}

% \begin{enumerate}
% \item Does the $\pow\dashv\pt$ adjunction restrict to an adjunction
%   between pseudotopological convergence spaces and some form of
%   convergence lattices?  One is tempted to say that a convergence
%   $\catc$-object is pseudotopological if and only if, for every filter
%   $\F$ on $L$, $\lm_L \F = \bigwedge_{\U \in \beta (\F)} \lm_L \U$,
%   where $\beta (\F)$ denotes the set of prime filters containing $\F$.
%   It is fairly easy to show that $\pow\dashv\pt$ restricts to an
%   adjunction between pseudotopological convergence spaces and
%   pseudotopological $\catc$-objects, provided that $\catc$ is a
%   category of Boolean algebras.  Does this extend to categories of
%   coframes?
% \item A famous theorem by Isbell states that the localic product of
%   compact locales is compact, and that does not depend on the Axiom of
%   Choice.  It is fairly easy to prove similar Tychonoff-like theorems
%   in our setting, but they all seem to require the Axiom of Choice.
%   For example, let $\catc$ be a category of coframes, and say that a
%   convergence $\catc$-object $L$ is compact if and only if for every
%   filter $\F$ on $L$, there is a filter $\G \supseteq \F$ such that
%   $\lim\G \neq \bot$.  Given a family of compact convergence
%   $\catc$-objects ${(L_i)}_{i \in I}$, form their coproduct $L$ in
%   $\Cv\catc$ (their product in ${(\Cv\catc)}^{op}$).
% \item Logic
% \item Semantics
% \item Bitopologies
% \end{enumerate}

\bibliographystyle{plain}

\begin{thebibliography}{10}
	
	\bibitem{AHS:joycats}
	Ji{\v r\'\i} Ad{\'a}mek, Horst Herrlich, and George~E. Strecker.
	\newblock {\em Abstract and Concrete Categories: The Joy of Cats}.
	\newblock Dover Publications, July 2009.
	
	\bibitem{choquetgrille}
	Gustave Choquet.
	\newblock Sur les notions de filtre et de grille.
	\newblock {\em C. R. Acad. Sci. Paris}, 224:171--173, 1947.
	
	\bibitem{DM:convergence}
	Szymon Dolecki and Fr{\'e}d{\'e}ric Mynard.
	\newblock {\em Convergence Foundations of Topology}.
	\newblock World Scientific, August 2016.
	\newblock ISBN 978-981-4571-52-4.
	
	\bibitem{GZ:fractions}
	Peter Gabriel and Michel Zisman.
	\newblock {\em Calculus of Fractions and Homotopy Theory}.
	\newblock Springer Verlag, 1967.
	
	\bibitem{contlattices}
	G.~Gierz, K.H. Hofmann, K.~Keimel, J.~Lawson, M.~Mislove, and D.~Scott.
	\newblock {\em Continuous Lattices and Domains}, volume~93 of {\em Encyclopedia
		of Mathematics}.
	\newblock Cambridge University Press, 2003.
	
	\bibitem{GL-acs13}
	Jean Goubault{-}Larrecq.
	\newblock Exponentiable streams and prestreams.
	\newblock {\em Applied Categorical Structures}, 22(3):515--549, June 2014.
	
	\bibitem{Johnstone:Stone}
	Peter Johnstone.
	\newblock {\em Stone Spaces}.
	\newblock Cambridge University Press, 1982.
	\newblock ISBN 978-0-521-23893-9.
	
	\bibitem{mac1971category}
	Saunders Mac~Lane.
	\newblock {\em Category theory for the working mathematician}.
	\newblock Springer-Verlag New York, 2nd ed edition, 1997.
	
	\bibitem{framesandlocales}
	J.~Picado and A.~Pultr.
	\newblock {\em Frames and Locales: Topology without points}.
	\newblock Frontiers in Mathematics. Birkha{\"u}ser, 2011.
	
	\bibitem{Raney:compl:distr}
	George~N. Raney.
	\newblock Completely distributive complete lattices.
	\newblock {\em Proceedings of the American Mathematical Society},
	3(5):677--680, 1952.
	
\end{thebibliography}

\end{document}